\documentclass[11pt]{amsart}
\usepackage{amssymb, amsfonts, latexsym, amsthm, amsmath, eucal, graphicx, verbatim, hyperref, xcolor, overpic}
\usepackage[margin=1in]{geometry}
\pdfoutput=1

\theoremstyle{plain}
\newtheorem{theorem}{Theorem}[section]
\newtheorem*{main}{Theorem}

\newtheorem{lemma}[theorem]{Lemma}
\newtheorem{cor}[theorem]{Corollary}
\newtheorem{prop}[theorem]{Proposition}

\newtheorem{claim}[theorem]{Claim}

\theoremstyle{definition}
\newtheorem{defn}[theorem]{Definition}
\newtheorem{example}[theorem]{Example}
\newtheorem{remark}[theorem]{Remark}

\newtheorem{notation}[theorem]{Notation}

\numberwithin{equation}{section}

\newcommand{\ba}{\backslash}
\newcommand{\Q}{\mathbb{Q}}
\newcommand{\R}{\mathbb{R}}
\newcommand{\C}{\mathbb{C}}

\newcommand{\Z}{\mathbb{Z}}
\newcommand{\F}{\mathbb{F}}
\newcommand{\ov}{\overline}

\newcommand{\val}{\operatorname{val}}

\definecolor{ltgreen}{rgb}{0.0, 0.5, 0.0}
\definecolor{dkgreen}{rgb}{0.0, 0.42, 0.24}

\begin{document}
\pagestyle{headings}

\author{Elizabeth Mili\'{c}evi\'{c}}
\title[Maximal Newton points and the quantum Bruhat graph]{Maximal Newton points and the quantum Bruhat graph}
\address{Department of Mathematics \& Statistics, Haverford College, Haverford, PA, 19041, USA}
\email{emilicevic@haverford.edu}

\thanks{This work was partially supported by Simons Collaboration Grant 318716, Australian Resarch Council Grant DP130100674, National Science Foundation Grant 1600982, and the Max-Planck-Institut f\"{u}r Mathematik.}

\begin{abstract}
We discuss a surprising relationship between the partially ordered set of Newton points associated to an affine Schubert cell and the quantum cohomology of the complex flag variety.  The main theorem provides a combinatorial formula for the unique maximum element in this poset in terms of paths in the quantum Bruhat graph, whose vertices are indexed by elements in the finite Weyl group.  Key to establishing this connection is the fact that paths in the quantum Bruhat graph encode saturated chains in the strong Bruhat order on the affine Weyl group.  This correspondence is also fundamental in the work of Lam and Shimozono establishing Peterson's isomorphism between the quantum cohomology of the finite flag variety and the homology of the affine Grassmannian.  One important geometric application of the present work is an inequality which provides a necessary condition for non-emptiness of certain affine Deligne-Lusztig varieties in the affine flag variety.
\end{abstract}


\maketitle

\section{Introduction}\label{S:intro}

This paper investigates connections between the geometry and combinatorics in two different, but surprisingly related contexts:  certain subvarieties of the affine flag variety in characteristic $p>0$ and the quantum cohomology of the complex flag variety. The main results establish explicit relationships among fundamental questions in both theories, using paths in the quantum Bruhat graph as the primary dictionary. We begin with a brief historical survey of these two geometric contexts in order to frame the informal statement of the main theorem.

\subsection{Newton polygons}

In the 1950s, Dieudonn\'{e} introduced the notion of isocrystals over perfect fields of characteristic $p>0$ (see \cite{Man}), which Grothendieck extended to families of $F$-crystals in \cite{Gro}.  Isogeny classes of $F$-crystals are indexed by combinatorial
objects called \emph{Newton polygons}, a partially ordered set of lattice polygons in the plane.  Kottwitz used the machinery of algebraic groups to explicitly study the set of Newton points associated to any connected reductive group $G$ over a discretely valued field in \cite{KotIsoI, KotIsoII}. In particular, he observed that there is a natural bijection between the set of Frobenius-twisted conjugacy
classes in $G$ and a suitably generalized notion of the set of Newton polygons.  The poset of Newton
points in the context of reductive group theory has interesting combinatorial and Lie-theoretic
interpretations, which were first described in \cite{Ch}.  For example, Chai proves that the poset of Newton points is ranked; \textit{i.e.},
any two maximal chains have the same length. In addition to the classification of $F$-crystals and Frobenius-twisted conjugacy classes, modern interest in the poset of Newton points is motivated by geometric applications to the study of two important families of varieties in arithmetic algebraic geometry: affine Deligne-Lusztig varieties and Shimura varieties; see \cite{GHKRadlvs, RapShimura}.

\subsection{Quantum cohomology}

Independently, physicists working in the field of superstring theory in the early 1990s observed that certain algebraic varieties seemed to present natural vacuum solutions to superstring equations, and thus developed a theory of \emph{quantum cohomology}; see \cite{WittenGrass}. Using the notion of mirror symmetry, they were able to employ this cohomological framework to calculate the number of rational curves of a given degree on a general quintic hypersurface in projective 4-space; see \cite{CanPhysics}. Mathematicians first rigorously worked out the structure of the quantum cohomology ring for the Grassmanian variety of $k$-planes in $\C^n$, and initial mathematical applications were also to enumerative geometry; see \cite{BertDask, Bert, SiebTian}. Modern mathematical interest focuses on concretely understanding the structure of the quantum cohomology ring for any homogeneous variety $G/P$, where $G$ is a complex reductive algebraic group and $P$ a parabolic subgroup. More precisely, the ring $QH^*(G/P)$ has a basis of Schubert classes, indexed by elements of the corresponding Weyl group.  The driving question is then to find non-recursive, positive combinatorial formulas for expressing the quantum product of two Schubert classes in terms of this basis.  Immediate applications include statistics about mapping projective curves to $G/P$ satisfying various incidence conditions, but the impact now extends beyond enumerative geometry into many other aspects of algebraic geometry, combinatorics, representation theory, number theory, and also back to physics.

\subsection{Main theorem and applications}\label{S:MainInformal}

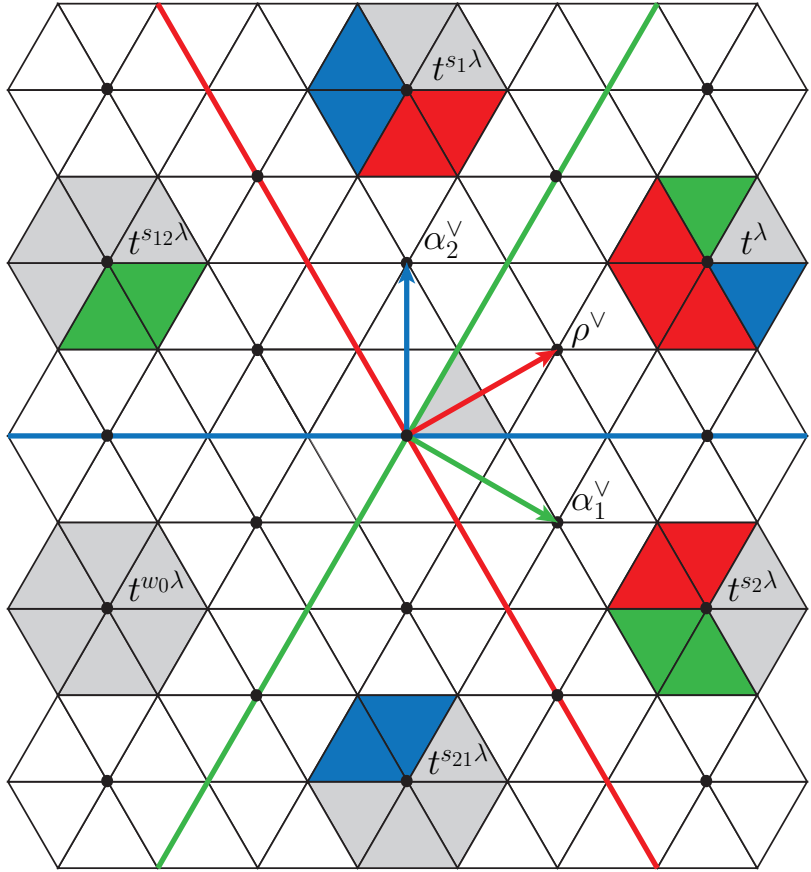
\begin{figure}
\begin{center}
 \resizebox{4.25in}{!}
{
\begin{overpic}{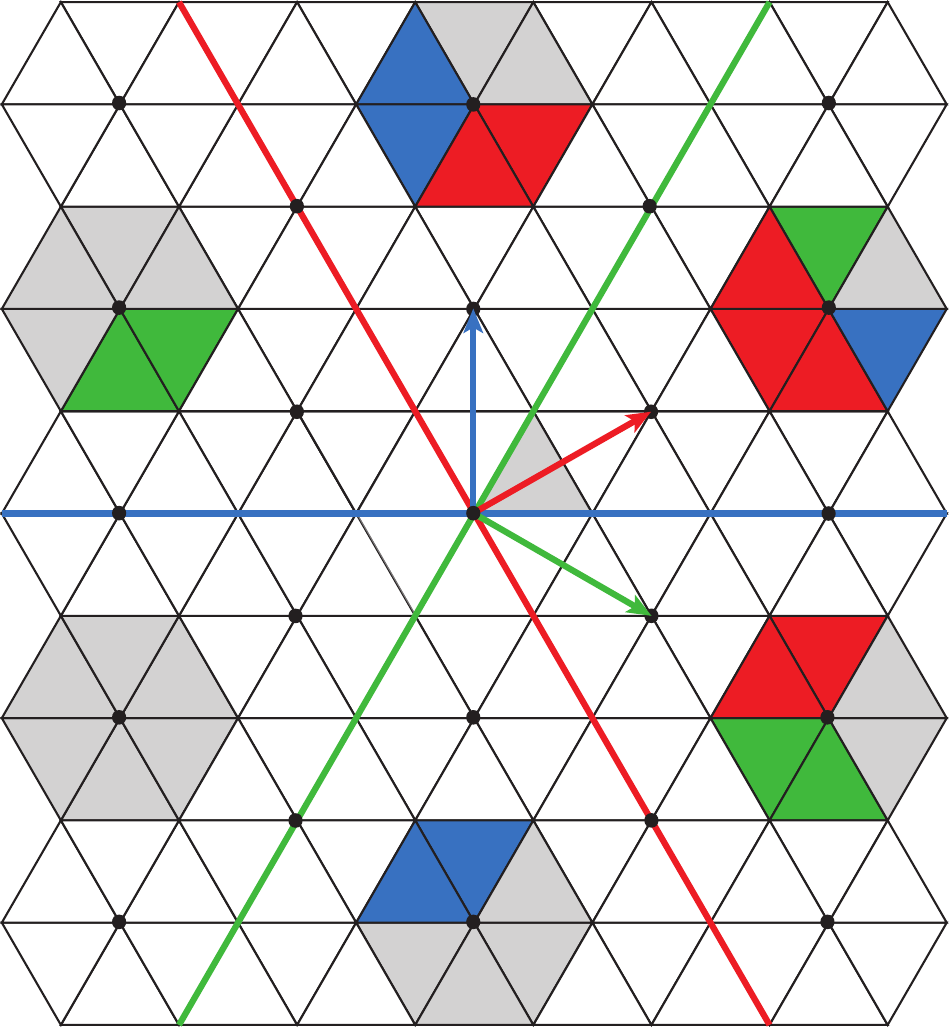}
\put(84.5,71){\bf \huge{$t^{\lambda}$}}
\put(49,91){\bf \huge{$t^{s_1\lambda}$}}
\put(83,31){\bf \huge{$t^{s_2\lambda}$}}
\put(14,71){\bf \huge{$t^{s_{12}\lambda}$}}
\put(48.5,11){\bf \huge{$t^{s_{21}\lambda}$}}
\put(14,31){\bf \huge{$t^{w_0\lambda}$}}
\put(48,71.5){\bf \huge{$\alpha_2^\vee$}}
\put(65,41.5){\bf \huge{$\alpha_1^\vee$}}
\put(65,61){\bf \huge{$\rho^\vee$}}
\end{overpic}
}
\caption{Compute the maximum Newton point for $x=t^{v\lambda}w$ by subtracting the coroot of the same color from $\lambda$. A gray alcove requires no correction factor.}\label{fig:SL3MaxNPs}
\end{center}
\end{figure}

The main result in this paper provides a closed combinatorial formula for the maximum element in the poset of Newton points associated to a fixed affine Schubert cell.  We express this formula in terms of paths in the \emph{quantum Bruhat graph}, a directed graph with vertices indexed by the finite Weyl group and weights given by the reflections used to get from one element to the other; see Figure \ref{fig:S_3QBG} and Section \ref{S:QBGdef} for an example and the formal definition. We illustrate the main theorem in the case of $G=SL_3$ in Figure \ref{fig:SL3MaxNPs}, and we informally state our main result below.  For the precise statements, see Theorem \ref{T:main} and Corollary \ref{T:MainCor}.

\begin{main}\label{T:MainWords}
Let $x = t^{v\lambda}w$ be an element of the affine Weyl group, where $w$ denotes the finite part of $x$, and $v$ the Weyl chamber in which $x$ lies. Suppose that $x$ is suitably far from the walls of any Weyl chamber.  Then the maximum Newton point in the affine Schubert cell indexed by $x$ is given by taking the coroot $\lambda$ and subtracting the weight of any path of minimal length in the quantum Bruhat graph from $w^{-1}v$ to $v$.
\end{main}

This result answers a question about Newton points, which are connected to affine Deligne-Lusztig varieties, but the proof makes heavy use of the tools developed to study quantum Schubert calculus.  As such, we are also able to prove corollaries in each of these two fields, albeit also under the hypothesis that $x$ lies suitably far from the walls of any Weyl chamber.  Corollary \ref{T:QSchNP} illustrates that our question about the maximum Newton point is equivalent to determining the minimum monomial occurring in any quantum product of two Schubert classes in $G/B$ over $\C$.  On the affine side, Corollary \ref{T:ADLVapp} proposes an analog of Mazur's inequality for the affine flag variety, placing a sharp upper bound on the Newton point for $b$ in the context in which the affine Deligne-Lusztig variety $X_x(b)$ is non-empty.  For readers primarily interested in applications to Shimura varieties, we point out that this analog of Mazur's inequality holds equally well in the $p$-adic context; see Remark \ref{R:padic} for details.  While the cases in which $b$ is basic and/or $x$ lies close to the walls of the Weyl chambers are typically the most important for extracting applications to Shimura varieties, the combinatorics developed in this paper nevertheless yields interesting geometric information about certain Newton strata; see Section \ref{S:Future}.

\subsection{Future directions}\label{S:Future}

The poset $\mathcal{N}(G)_x$ of Newton points associated to the affine Schubert cell indexed by $x=t^{v\lambda}w$ perfectly detects the non-emptiness pattern for any affine Deligne-Lusztig variety $X_x(b)$.  The main theorem thus says that the quantum Bruhat graph senses non-emptiness of $X_x(b)$ when the element $b \in G(F)$ has the largest possible Newton point.  Jointly with Schwer and Thomas in \cite{MST1}, the author has developed machinery involving labeled folded alcove walks and root operators which is effective for predicting non-emptiness for elements $b$ with Newton points which have integral slopes and lie below $\lambda -2\rho$.  Although each of the two techniques presents different challenges when $x$ lies outside of the ``shrunken'' Weyl chambers, one might hope that an interpolation between these two methods might result in a complete picture for the non-emptiness problem, complementing our detailed knowledge of the basic case established in \cite{GHN}. In joint work with Viehmann, the author identifies a family of elements $x$ such that the Newton poset $\mathcal{N}(G)_x$ is saturated, from which one deduces geometric information: formulas for codimensions of the Newton strata in the affine Schubert cell $IxI$, as well as their equidimensionality \cite{MilVie}.

More mystifying are the connections which arise between the main theorem and quantum Schubert calculus.  Although Peterson established an isomorphism between suitable localizations of the equivariant homology of the affine Grassmannian and the equivariant quantum cohomology of the complete flag variety \cite{Pet}, it remains to explain the precise relationship of these cohomology theories to the geometry of other subvarieties of the affine flag variety, especially in characteristic $p>0$.  For example, this paper shows that the maximal element of the poset of Newton points $\mathcal{N}(G)_x$ and the minimal monomial $q^d$ in the quantum product of two Schubert classes are governed by exactly the same combinatorial information---it would therefore be natural to explore whether the posets share other elements besides these extrema. 

Since Lam and Shimozono's proof of the Peterson isomorphism in \cite{LS}, many other applications of the ``quantum equals affine'' phenomenon have been discovered.  We highlight several results beyond Schubert calculus which have made similar critical use of the connection between alcove walks and the quantum Bruhat graph. In a series of papers, Lenart, Naito, Sagaki, Schilling, and Shimozono compute the energy function on tensor products of certain Kirillov-Reshetikhin crystals in terms of the parabolic quantum Bruhat graph; see \cite{LNSSS2}.  Feigin and Makedonskyi describe the representation theory of generalized Weyl modules in terms of a generating function on quantum alcove paths \cite{FM}, based on a similar formula of Orr and Shimozono for a specialization of nonsymmetric Macdonald polynomials \cite{OS}.  More recently, Naito and Watanabe proved a combinatorial formula for periodic $R$-polynomials in terms of paths in a doubled quantum Bruhat graph \cite{WN}.  These $R$-polynomials can be used to compute periodic Kazhdan-Lusztig polynomials, which conjecturally determine the characters of irreducible modules of a reductive group over a field of positive characteristic.  It would be interesting to understand the geometric and/or representation-theoretic phenomena which cause the answers to these seemingly different questions to be governed by the same combinatorics.

Finally, it is our hope that the combinatorics community in particular might invigorate new ideas regarding the poset of Newton points studied here, perhaps at least in the case of $G=GL_n$ in which the setup is quite combinatorial.  For this reason, we make an effort to be very concrete throughout Section \ref{S:NPs} in our discussion of the Newton map, mentioning open problems along the way.  Regarding the poset $\mathcal{N}(G)_x$ of Newton points associated to an affine Weyl group element $x$, besides the results in this paper on maximal elements, its minimal elements are known \cite{VieMinNP}, as well as some of its integral elements \cite{MST1} and non-integral elements \cite{HeCordial}.  Beyond groups of low rank, however, we have not yet established even the most basic desirable poset properties, such as whether or not it is ranked, a lattice, shellable, or for which $x$ it is a subinterval of the poset of all Newton points for $G$.  Of course, the ideal goal would then be to apply such poset combinatorics back to the geometry of the associated varieties in characteristic $p>0$.

\subsection{Overview of the paper}

This paper is written with two potentially disjoint audiences in mind: those interested in arithmetic geometry and the Newton map, and those interested in combinatorics and quantum Schubert calculus.  After establishing notation, we thus open in Section \ref{S:NPs} with an elementary review of Newton polygons, presenting explicit formulas in many special cases along with examples.  The reader familiar with reductive groups over local fields and the Newton map can safely skip to Section \ref{S:MaxNPs} for a refresher on the maximum Newton point.  Section \ref{S:QBG} begins with a review of the quantum Bruhat graph and its main combinatorial properties.  Experts in quantum Schubert calculus can move straight into Section \ref{S:MainPrecise}, which contains a precise statement of the main result as Theorem \ref{T:main}, continuing directly to Section \ref{S:qSch} for the main application to quantum cohomology.  The main application to affine Deligne-Lusztig varieties is then presented and proved in Section \ref{S:ADLVs}. 

The remainder of the paper is dedicated to the proof of Theorem \ref{T:main}.  We start by invoking an alternative combinatorial formula of Viehmann for the maximum Newton point from \cite{VieTrunc}, which is expressed in terms of both affine Bruhat order and dominance order; see Theorem \ref{T:Vform}.  The connection to quantum Schubert calculus is then made by using an observation of Lam and Shimozono in their proof of the Peterson isomorphism \cite{LS}, which places covering relations in affine Bruhat order in two-to-one correspondence with edges in the quantum Bruhat graph.  In the remainder of Section \ref{S:QBGChains}, we iterate Proposition \ref{T:cocover}, stitching the relations from this correspondence together to form saturated chains.  Section \ref{S:RootHyps} then lays the groundwork to compare the Newton points for all of the elements lying below a given $x$ in Bruhat order.  This section represents the technical heart of the paper, involving careful combinatorics on root hyperplanes in order to bound the maximum Newton point from below.  We put the resulting sequence of lemmas to work in Section \ref{S:TransPf}, making the primary reduction step to considering only pure translations less than $x$.  The proof of the main theorem then follows immediately in Section \ref{S:MainPf}.  We conclude in Section \ref{S:HypRmks} with a discussion of the role of the superregularity hypothesis in the proof of Theorem \ref{T:main}.

\subsection*{Acknowledgments.}
The author thanks John Stembridge and Thomas Lam for  instrumental conversations during the conception and development of this project.  Part of this work was conducted during a visit to the University of Melbourne, and the author thanks Arun Ram for fostering an exceptionally mathematically stimulating environment.  This paper was completed during the a stay at the Max-Planck-Institut f\"{u}r Mathematik, and the author wishes to gratefully acknowledge the institute for its excellent working conditions.  The author is graciously indebted to the anonymous referees for their careful readings which resulted in opportunities to clarify, correct, and simplify many parts of the paper.  A number of these revisions were implemented while the author served as a Director's Mathematician in Residence at the Budapest Semesters in Mathematics program.

This paper represents the full version of extended abstract \cite{BeMaxNPsFPSAC}, which was published in the proceedings of the $24^{\text{th}}$ International Conference on ``Formal Power Series and Algebraic Combinatorics''.  Since that announcement, the statements have been sharpened, and the conventions were also changed to make clearer the correspondence to alcoves in the affine hyperplane arrangement.  Spatial constraints on the extended abstract permitted only a brief outline of the proof, whereas this paper contains all of the details and a full discussion of the subtleties involved in the argument.


\section{The Poset of Newton Points}\label{S:NPs}

After we establish some basic notation for root systems and Weyl groups, the primary purpose of this section is to introduce the Newton map and discuss some of the basic properties of partially ordered sets of Newton points.  The formal definition is fairly abstract, and so we offer the reader three alternative ways to think about constructing the Newton polygon for an element of $GL_n$.  We discuss the most important special case alongside the general definition in Section \ref{S:NPsDef}.  Each method we review poses its own computational challenges, but an implicit common theme is the need to choose a suitable representative of each $\sigma$-conjugacy class in $G(F)$.  We do not intend to provide a comprehensive survey of this subject, nor do we provide an exhaustive list of constructions for the Newton point, even in the special case of $GL_n$.  Rather, our goal is to provide some general exposure to the methods, particularly for the audience more interested in the combinatorial aspects.

\subsection{Notation}\label{S:notation}

Let $G$ be a split connected semisimple algebraic group of rank $r$, and $B$ a fixed Borel subgroup with $T$ a maximal torus in $B$.  Let $R = R^+ \sqcup R^-$ be the corresponding set of roots, which can be viewed as a subset of the group of characters $X^*(T)$. Note that $R$ is reduced since $G$ is semisimple. Denote the set of positive integers $\{1, \dots, r\}$ by $[r]$. The set $\Delta = \{ \alpha_i \mid i \in [r] \}$ is an ordered basis of simple roots in $X^*(T)$, and $\{\alpha_i^{\vee} \mid i \in [r]\}$ a basis for $R^\vee$ of simple coroots in $X_*(T)$, which are dual with respect to the pairing $\langle \cdot , \cdot \rangle: X_*(T) \times X^*(T) \rightarrow \Z$. The parabolic subgroups $P$ containing $B$ are in bijection with subsets $\Delta_P$ of $\Delta$, each of which has an associated root system $R_P$ consisting of the roots of the Levi factor. The finite Weyl group $W$ is the quotient $N_G(T)/T$.  Denote by $r_{\alpha}$ the reflection in $W$ corresponding to the root $\alpha \in R$, and then write $s_i$ for the simple reflection in $W$ corresponding to the simple root $\alpha_i \in \Delta$.  There is a natural action of $W$ on $R$ which induces a permutation on the set of reflections in $W$ via $r_{v\alpha} = vr_\alpha v^{-1}$ for any $\alpha \in R$ and $v \in W$.  Define $\rho$ to be the half-sum of the positive roots.

Denote by $Q = \bigoplus \Z \alpha_i$ and $Q^\vee = \bigoplus \Z \alpha_i^\vee$ the root and coroot lattices, respectively.  We also occasionally need the coweight lattice $P^\vee = \bigoplus \Z \omega_i^\vee \supset X_*(T)$ spanned by the fundamental coweights $\omega_i^\vee$, which are dual to the simple roots $\alpha_i$.  Define $\rho^\vee$ to be the sum of the fundamental coweights. The finite Weyl group $W$ acts on $\R^r \cong Q^\vee \otimes_{\Z}\R$ as a finite reflection group. Denote by $H_{\alpha}$ the hyperplane in $\R^r$ orthogonal to the root $\alpha$, which is the reflecting hyperplane corresponding to $r_\alpha$ in this representation.  We say that $\lambda \in P^{\vee}$ is dominant if $\langle \lambda, \alpha_i \rangle \geq 0$ for all $\alpha_i \in \Delta$. Denote by $Q^+$ and $P^+$ the set of dominant elements of $Q^\vee$ and $P^\vee$, respectively.  By $\lambda^+$ we mean the unique dominant coroot in the $W$-orbit of $\lambda \in Q^{\vee}$. More generally, define the (closed) dominant Weyl chamber as
\begin{equation}
C = \{\lambda \in \R^r \mid \langle \lambda, \alpha \rangle \geq 0,\ \forall \alpha \in R^+\}.
\end{equation}
Dominance order forms a natural partial ordering on the coroot lattice.  Given $\lambda, \mu \in Q^\vee$, we say that $\lambda \geq \mu$ if and only if $\lambda - \mu$ is a nonnegative linear combination of positive coroots.  We say that $\lambda \in Q^\vee$ is regular if the stabilizer of $\lambda$ in $W$ is trivial.  Following \cite{LS}, a coroot $\lambda$ is said to be \emph{superregular} if $ \langle \lambda, \alpha \rangle \geq M$ for all $\alpha \in R^+$, for some sufficiently large $M$.  In practice, the constant $M$ is fixed; see Corollary \ref{T:MainCor} for a crude approximation, and refer to Equation \eqref{E:MFormula} for a precise formula.  Although the notion of superregularity always depends on $M$, we occasionally omit reference to this constant and simply refer to $\lambda$ as superregular.

In the context of the Newton map, we will work over the discretely valued field $F = \overline{\F_q}((t))$ for $q = p^s$, which has characteristic $p >0$. We denote the discrete valuation by $\operatorname{val}: F \rightarrow \Z$, and this map picks out the smallest power of $t$ occurring with nonzero coefficient in the Laurent series; for our purposes, we define $\val(0) = -\infty$. The ring of integers in $F$ is given by the ring of formal power series $\mathcal{O} = \overline{\F_q}[[t]]$.  For any $\mu \in X_*(T)$, let $t^{\mu}$ denote the image of $t$ under $\mu: \mathbb{G}_m \rightarrow T$. We can extend the usual Frobenius automorphism $x \mapsto x^q$ on $\ov{\F_q}$ to a map $\sigma: F \rightarrow F$ by defining $\sigma$ to act on the coefficients: $\sum a_it^i \mapsto \sum a_i^qt^i$. Two elements $g_1,g_2 \in G(F)$ are said to be \emph{$\sigma$-conjugate} if there exists an $h \in G(F)$ such that $hg_1\sigma(h)^{-1} = g_2$.  

The affine Weyl group of $G(F)$ is isomorphic to the semi-direct product $\widetilde{W} = Q^{\vee} \rtimes W$, and any $x \in \widetilde{W}$ may be written as $x = t^{\lambda}w$ for some $\lambda \in Q^\vee$ and $w \in W$. If $\lambda$ is regular, then there exists a unique $v \in W$ such that $t^{\lambda}w = t^{v\lambda^+}w$.  The affine Weyl group is generated by the following affine reflections on $\R^r$:
\begin{equation}
r_{\alpha,m}(\lambda) = \lambda - (\langle \lambda, \alpha \rangle-m)\alpha^\vee.
\end{equation}
The element $r_{\alpha,m}$ is then identified with the reflection across the affine hyperplane
\begin{equation}
H_{\alpha,m} = \{ \lambda \in \R^r \mid \langle \lambda, \alpha \rangle = m\}.
\end{equation}
Note that $r_{\alpha,0} = r_\alpha$ and $H_{\alpha,0}=H_\alpha$.  The connected components of $\R^r \ba \{H_{\alpha,m} \mid \alpha \in R^+, m \in \Z \}$ are called alcoves, and we use freely the natural bijection between elements of the affine Weyl group and alcoves.  We define the base alcove to be the unique alcove in $C$ whose closure contains the origin, and the base alcove then corresponds to the identity element in $\widetilde{W}$ under this bijection.  The Iwahori subgroup $I$ of $G(F)$ is the stabilizer of the base alcove, which coincides with the  preimage of the opposite Borel subgroup of $B$ under the evaluation map $G(\overline{\F_q}[[t]]) \rightarrow G(\overline{\F_q})$ sending $t \mapsto 0$.

In terms of the isomorphism $\widetilde{W} = Q^{\vee} \rtimes W$, we can write $r_{\alpha,m} = t^{m\alpha^\vee}r_\alpha$, which acts by left multiplication as the reflection across the hyperplane $H_{\alpha,m}$. Because each affine reflection is also associated to a unique positive root in the affine Lie algebra, we typically write $r_\beta$ rather than $r_{\alpha,m}$ for brevity.  It should always be clear from context whether $r_\beta$ represents a linear or an affine reflection, especially because we typically reserve the letters $v,w$ for finite Weyl group elements and $x,y$ for affine Weyl group elements.  Finally, denote by $\ell : \widetilde{W} \rightarrow \Z_{\geq 0}$ the length function, and by $w_0$ the element of longest length in the finite Weyl group.

\begin{subsection}{Newton polygons via isocrystals}  

The notion of an isocrystal over a perfect field of characteristic $p>0$ was introduced by Dieudonn\'{e} and generalized by Grothendieck \cite{Gro}. In the classification theorem proved by Dieudonn\'{e} \cite{Dieudonne} and Manin \cite{Man}, isomorphism classes of isocrystals are naturally indexed by Newton polygons, which then became the starting point for the development of the Newton map in the context of algebraic groups by Kottwitz \cite{KotIsoI}.

\begin{defn} An \emph{isocrystal} $(V,\Phi)$ is a finite-dimensional vector space $V$ over $F$ together with a
$\sigma$-linear bijection $\Phi: V\rightarrow V$; that is, $\Phi(v+w) = \Phi(v)+\Phi(w)$ and $\Phi(av)=\sigma(a)\Phi(v)$ for all $a \in
F$ and $v,w \in V$. 
\end{defn}

A simple example of an isocrystal is $(F^n,\Phi)$, where $\Phi = A\circ \sigma$ for some $A \in GL_n(F)$, and $\sigma$ acts on $V$ coordinate-wise.  Conversely, if we fix a basis $\{e_1, \dots, e_n\}$ for $V$, then note that we can associate a matrix $A \in GL_n(F)$ to $(V,\Phi)$ defined by $\Phi(e_i) = \sum_{i=1}^n A_{ji}e_j$, in which case we write $\Phi = A\circ \sigma$ for $A = (A_{ij})$.  More generally, for $G$ any connected reductive group over $F$, the choice of an element $g \in G(F)$ together with a finite-dimensional representation of $G$ determines an isocrystal over $F$; see \cite{RR} for a discussion of $F$-isocrystals with $G$-structure, following \cite{KotIsoI}.

The Dieudonn\'{e}-Manin classification shows that the category of isocrystals over $F$ is semisimple, and that the simple objects are naturally indexed by $\Q$.  That is, any isocrystal $(V,\Phi)$ is isomorphic to a direct sum of the form $V = \oplus V_{s_i/r_i}$ where $\gcd(r_i,s_i)=1$ and the $V_{s_i/r_i}$ are simple objects in the category; for a proof of semisimplicity and a characterization of the simple isocrystals, see \cite{Dem}.  
We present our first definition of the Newton polygon using this Dieudonn\'e-Manin classification.  Although we phrase this first definition in the language of $G=GL_n$ for the sake of clarity, the construction fully extends to the context of $F$-isocrystals with $G$-structure; see \cite{KotIsoI, KotIsoII}.

\begin{defn}\label{D:NPiso} Let $(V,\Phi)$ be a finite-dimensional isocrystal over $F$.
\begin{enumerate}
\item If $V = \oplus_{i=1}^nV_{s_i/r_i}$ by the Dieudonn\'e-Manin classification, then the \emph{Newton slope sequence} for $(V,\Phi)$ is defined to be $\lambda = (\lambda_1, \dots, \lambda_n)\in  (Q^\vee \otimes_{\Z} \Q)^+,$ where $\lambda_i = s_i/r_i$, each $\lambda_i$ is repeated $r_i$ times, and $\lambda_1 \geq \dots \geq \lambda_n$.  

\item The \emph{Newton polygon} for an isocrystal $(V,\Phi)$ with Newton slope sequence $\lambda$ is the graph of the function $\overline{\nu}: [0,n] \longrightarrow \R$ given by
\begin{equation*}
\begin{cases}
\overline{\nu}(i)=0, &\text{if}\ i=0,\\
\overline{\nu}(i) = \lambda_1 + \cdots \lambda_i, &\text{if}\ i= 1, \dots, n,
\end{cases}
\end{equation*}
and then extended linearly between successive integers.  

\item If we denote by $\mathcal{N}(G)$
the set of all possible Newton polygons for isocrystals arising from elements in $G(F)$, then the \emph{Newton map} $\nu: G \longrightarrow \mathcal{N}(G)$ sends $g\in G$ to the \emph{Newton point} $\nu(g)=\lambda$ for the isocrystal corresponding to $g$.
\end{enumerate}
\end{defn}

\begin{figure}[b]
\begin{center}
 \resizebox{2.5in}{!}
{
\begin{overpic}{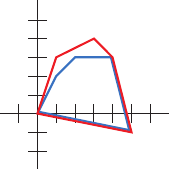}
\put(42,76){\bf \scriptsize{$\color{red} \lambda$}}
\put(42,56){\bf \scriptsize{$\color{blue} \mu$}}
\end{overpic}
}
\caption{A pair of Newton polygons for $GL_5(F)$ with slope sequences $\lambda \geq \mu$ .}\label{fig:NPfig}
\end{center}
\end{figure}

For example, the red Newton polygon in Figure \ref{fig:NPfig} corresponds to a 5-dimensional isocrystal having Newton slope sequence $\lambda = (3, \frac{1}{2}, \frac{1}{2}, -1, -4)$.  Because a Newton slope sequence uniquely determines a Newton polygon and vice versa in the case of $G=GL_n$, we occasionally refer to elements $\nu(g)$ interchangeably as both polygons in the plane and slope sequences in $(Q^\vee \otimes_{\Z}\Q)^+$.

\end{subsection}

\begin{subsection}{Newton polygons via characteristic polynomials}

We now provide a more concrete definition of the Newton polygon associated to an element $g \in GL_n(F)$ which only requires basic linear algebra; for the details, we refer the reader to \cite{Kedlaya}.  Although we do not appeal to the Dieudonn\'{e}-Manin classification here, the definition presented in this section is in fact equivalent to Definition \ref{D:NPiso}.  

Define a ring $R := F[\sigma]$ by formally adjoining the Frobenius automorphism.  Multiplication in $R$ is non-commutative, defined such that $\sigma a = \sigma(a) \sigma$ for $a \in F$. There exist both a right and left division algorithm, and so $R$ is
a principal ideal domain. Given an isocrystal $(V,\Phi)$ over $F$, identifying $\sigma^iv:=\Phi^i(v)$ makes $V$ into an $R$-module.  In fact, the Frobenius twist makes $(V,\Phi)$ into a cyclic module over the ring $R$; that is,
$Rv = V$ for some $v$ in $V$.  In this context, we call the generator $v$ a \textit{cyclic vector}.  Upon choosing a cyclic vector $v$, we may thus write $V \cong R/Rf$ for some $f =
\Phi(v)^n + a_1\Phi(v)^{n-1} \cdots + a_{n-1}\Phi(v) + a_nv \in R$ with $a_i \in F$, where $n=\mbox{dim}_F(V)$.   We call $f$ the
\emph{characteristic polynomial} associated to this isocrystal and cyclic vector $(V,\Phi,v)$.

\begin{defn}[Lemma 5.2.4 \cite{Kedlaya}]\label{T:charNP} Given an isocrystal and a choice of cyclic vector $(V,\Phi,v)$, define the associated Newton polygon as the result of the following algorithm:
\begin{enumerate}
\item Find the characteristic polynomial $f =
\Phi(v)^n + a_1\Phi(v)^{n-1} \cdots + a_{n-1}\Phi(v) + a_nv \in R$,  satisfying $V \cong R/Rf$.
\item Plot the set of points $\{(0,-\val(a_{n})),(1,-\val(a_{n-1})), \dots, (n-1,-\val(a_{1})), (n, 0) \}$ recording the negative valuations of the coefficients of this characteristic polynomial in reverse order.
\item Take the upper convex hull of this set of points; \textit{i.e.} form the tightest-fitting polygon which passes either through or above all of the plotted points.
\end{enumerate}
While the characteristic polynomial $f$ clearly depends on the choice of a cyclic vector, the Newton polygon associated to $(V,\Phi,v)$ is actually independent of $v$.  We can thus safely refer to the result of this construction as the \emph{Newton polygon} for the isocrystal $(V,\Phi)$.  By recording the slopes of each edge of the Newton polygon left to right, repeated with multiplicity, we obtain the corresponding Newton slope sequence $\lambda \in (Q^\vee \otimes_{\Z}\Q)^+$.
\end{defn}

We illustrate via example how one can find a suitable choice of a cyclic vector, calculate the characteristic polynomial, and then construct the associated Newton polygon.

\begin{example}\label{T:GL_2}
Let $(F^2, \Phi)$ be an isocrystal, where $\Phi = g\circ \sigma$ for some $g:=\begin{pmatrix} a&b\\c&d\end{pmatrix} \in GL_2(F)$. The first standard basis vector $e_1$ is a cyclic vector for $(F^2,\Phi)$ if and only if $c \neq 0$, since  \begin{equation}e_1 \wedge \Phi(e_1) = \begin{pmatrix} 1 \\0 \end{pmatrix} \wedge \begin{pmatrix} a  \\ c  \end{pmatrix} = c (e_1 \wedge e_2).\end{equation}  If $c \neq 0$, then we can compute the characteristic polynomial for $(F^2,\Phi,e_1)$ by solving for $\alpha$ and $\beta$ in the following $F$-linear system of equations:
\begin{align}
\Phi^2(e_1) + \alpha \cdot \Phi(e_1) + \beta\cdot e_1 &= 0,\\
\begin{pmatrix} a & b \\ c & d \end{pmatrix} \begin{pmatrix} \sigma(a) \\ \sigma(c) \end{pmatrix} + \alpha \begin{pmatrix} a \\ c \end{pmatrix} + \beta \begin{pmatrix} 1 \\ 0 \end{pmatrix} &= \begin{pmatrix} 0 \\ 0 \end{pmatrix}
.
\end{align}
Therefore, 
\begin{align}
\begin{pmatrix} \alpha \\ \beta \end{pmatrix} &= - \begin{pmatrix} a & 1 \\ c & 0 \end{pmatrix}^{-1}\begin{pmatrix} a & b \\ c & d \end{pmatrix} \begin{pmatrix} \sigma(a) \\ \sigma(c) \end{pmatrix}= \begin{pmatrix} -\sigma(a)-\frac{\sigma(c)}{c}d \\ \frac{\sigma(c)}{c}(ad - bc). \end{pmatrix}
\end{align}
Note that if $a, c \in F^\sigma$ so that $\sigma(a)=a$ and $\sigma(c)=c$, then this $\sigma$-twisted version of the characteristic polynomial coincides with the usual characteristic polynomial for $GL_2$.

For a concrete example, now suppose that $g = \begin{pmatrix} t^{-2} & t^{-1} \\ 1 & t^3 \end{pmatrix}$.  Using the formula derived above, the coefficients of the characteristic polynomial then equal $\alpha = -t^{-2}-t^3$ and $\beta = -t^{-1}+t$.  Taking valuations, we see that $\val(\alpha) = -2$ and $\val(\beta) = -1$.  Therefore, the Newton polygon $\nu(g)$ is the convex hull of the three points $(0,1), (1,2),$ and $(2,0)$ and has slope sequence $\lambda=(1,-2)$.  

\end{example}

Using either definition of the Newton polygon presented so far, there are some challenges to explicitly calculating $\nu(g)$, even given a specific element $g \in G(F)$. Definition \ref{D:NPiso} requires a detailed understanding of the  simple objects in the category of isocrystals.  Although Definition \ref{T:charNP} is relatively concrete, calculations such as those performed in Example \ref{T:GL_2} can become unwieldy for groups of large rank.  To get a sense for how complexity grows with the rank of $G$, see \cite{Be1} for a treatment of the case $G=SL_3$, which is the only other group for which this calculation has been fully carried out in the literature.

\end{subsection}

\subsection{Newton points via extended affine Weyl group elements}\label{S:NPsDef}

The most general definition of the Newton map was given by Kottwitz in \cite{KotIsoI, KotIsoII}, in which he characterized its image on the set $B(G)$ of $\sigma$-conjugacy classes in $G(F)$. The image of an element $g \in G(F)$ under the Newton map is a $\sigma$-conjugacy class invariant, as is the connected component of $G(F)$.  Putting the Newton map together with the \emph{Kottwitz homomorphism} $\kappa$ identifying the connected component, one obtains an injective map
\begin{equation}
(\nu, \kappa) : B(G) \hookrightarrow (P^\vee\otimes_{\Z}\Q)^+ \times \pi_1(G).
\end{equation}
Moreover, the restriction of the map $G(F) \rightarrow B(G)$ to the normalizer $(N_GT)(F)$ factors through the \emph{extended affine Weyl group} $\widetilde{W}_e := (N_GT)(F)/T(\mathcal{O}) \cong X_*(T) \rtimes W$, and the map $\widetilde{W}_e \twoheadrightarrow B(G)$ is surjective; see Corollary 7.2.2 in \cite{GHKRadlvs}.  Therefore, in order to define the Newton map on an algebraic group $G(F)$, it suffices to be able to compute its image on elements of the extended affine Weyl group. Both for this reason and for the purpose of our own argument, the following formula for the image of the Newton map on an element in $\widetilde{W}_e$ is the most important special case.  

\begin{prop}\label{T:NPforW}
Let $y = t^\lambda w \in \widetilde{W}_e$, and suppose that the order of $w$ in $W$ equals $m$.  Then the \emph{Newton point} for $y$ equals
\begin{equation}\label{E:NPforW}
\nu(y) = \left( \frac{1}{m}\sum\limits_{i=1}^m w^i(\lambda)\right)^+,
\end{equation}
where we recall that $\mu^+$ denotes the unique dominant coroot in the $W$-orbit of $\mu \in P^{\vee}$.
\end{prop}

\noindent Formula \eqref{E:NPforW} is a standard fact in the literature; for example, see \cite{KotIsoI} or Section 4.2 in \cite{GoertzSurvey}. 

For general $G(F)$, the image of the Newton map is thus simply an element of $(P^\vee\otimes_{\Z}\Q)^+$, and from now on we typically refer to $\nu(y)$ as the Newton \emph{point} for $y$, rather than the Newton \emph{polygon}.  As we have seen, when $G=GL_n$ then the Newton point does in fact correspond to the slope sequence for a Newton polygon, and so we use either term when there is no risk of confusion.

We now provide an explicit example illustrating Proposition \ref{T:NPforW}, since Equation \eqref{E:NPforW} will play such a fundamental role in the proof of Theorem \ref{T:main}.  

\begin{example}
Let $G = SL_3(F)$ so that $W = S_3$ is generated by two simple reflections $s_1$ and $s_2$.  Let $y = t^{(-2,0,2)}s_1 \in \widetilde{S_3}$.  Then the Newton point for $y$ equals 
\[\nu(y) = \left[\frac{1}{2}\left( (0,-2,2)+(-2,0,2)\right)\right]^+ = \left[\frac{1}{2}(-2,-2,4)\right]^+= (2,-1,-1).\] The corresponding Newton polygon is the convex hull of the three points $(0,0), (1,2)$, and $(3,0)$.
\end{example}

In proving our main theorem we will only ever need to compute $\nu(y)$ for elements $y \in \widetilde{W}$.  Therefore, for our purposes, Equation \eqref{E:NPforW} will serve as our definition of the Newton map for general $g \in G(F)$; in practice, the only additional required step is to start by finding an element of $(N_GT)(F)$ which is $\sigma$-conjugate to $g$.

\begin{subsection}{The partial ordering on the set of Newton points}

There is a natural partial ordering on the set $\mathcal{N}(G)$ of Newton points occurring for elements in $G(F)$.  We compare two Newton points $\lambda, \mu \in (P^\vee \otimes_{\Z}\Q)^+$ by extending the dominance order on $P^\vee$ to $\Q^r$.  Namely, we say that $\lambda \geq \mu$ if and only if $\lambda-\mu$ is a nonnegative rational combination of positive coweights.   In Section \ref{S:Convexity}, we discuss another useful interpretation of the partial ordering on $\mathcal{N}(G)$ in terms of a convexity condition on Weyl orbits of rational points in the closed Weyl chambers.

\begin{example}
If $\lambda = (3, \frac{1}{2}, \frac{1}{2}, -1, -4)$ is the slope sequence for the red Newton polygon from Figure \ref{fig:NPfig}, then $\lambda$ is greater than the blue Newton point $\mu = (2, 1,0, 0, -4)$, since $\lambda-\mu=\alpha_1^\vee+ \frac{1}{2}\alpha_2^\vee+\alpha_3^\vee$.  Equivalently, all partial sums of the form $\lambda_1 + \cdots + \lambda_i$ are greater than or equal to those for $\mu$.  
\end{example}

Given a Newton polygon $\lambda$ in the plane, we say that another Newton polygon $\mu$ satisfies $\mu \leq \lambda$ if they share a left and rightmost vertex and all edges of $\lambda$ lie either on or above those of $\mu$.  Compare the red and blue Newton polygons in Figure \ref{fig:NPfig}, which illustrates that dominance order coincides with containment of Newton polygons in the case of $G=GL_n$.

 The full poset of Newton polygons $\mathcal{N}(G) = \{ \nu(g) \mid g \in G\}$ was initially studied in \cite{RR, KotIsoII, Ch} from the perspective of arithmetic algebraic geometry.  In particular, Chai established that $\mathcal{N}(G)$ has many desirable combinatorial properties of partially ordered sets.  For example, he proves that $\mathcal{N}(G)$ is a \emph{ranked} poset; \textit{i.e.} all maximal chains have the same length, and he also shows that $\mathcal{N}(G)$ is a lattice.

\end{subsection}

\begin{subsection}{Newton points in affine Schubert cells}

In the context of a reductive group over a local field, one version of the Bruhat decomposition says that 
\begin{equation}\label{E:Bruhat}
G(F) = \bigsqcup_{x \in \widetilde{W}} IxI,
\end{equation}
where recall that the Iwahori subgroup $I$ is the inverse image of the Borel subgroup opposite to $B$ under the map $G(\overline{\F_q}[[t]]) \rightarrow G(\overline{\F_q})$ sending $t \mapsto 0$. For example, if $G=SL_n(F)$ and $B$ is the subgroup of upper triangular matrices, then \begin{equation} I = \begin{pmatrix} \mathcal{O}^{\times} & t\mathcal{O} & \dots & t\mathcal{O} \\ \mathcal{O} & \mathcal{O}^{\times} & \dots & t\mathcal{O} \\ \vdots & \vdots & \ddots & \vdots \\ \mathcal{O} & \mathcal{O} & \dots & \mathcal{O}^{\times} \\ \end{pmatrix}.\end{equation} 

Motivated by applications to Shimura varieties and affine Deligne-Lusztig varieties, it is useful to study the combinatorics of subsets of $\mathcal{N}(G)$ in which one restricts to Newton points which arise from a fixed cell in this affine Bruhat decomposition.  The question then becomes to study the set \begin{equation}\label{E:N(G)_x} \mathcal{N}(G)_x := \{ \nu(g) \mid g \in IxI \}\end{equation}  of Newton points occurring for elements in the \emph{affine Schubert cell} $IxI$.  These subsets clearly inherit the partial ordering on $\mathcal{N}(G)$.  

 The posets $\mathcal{N}(G)_x$ have only been fully characterized for groups of low rank and/or when $x$ has a special form, but many nice combinatorial properties of $\mathcal{N}(G)$ also hold for $\mathcal{N}(G)_x$ in these cases.  For example, in  \cite{Be1} the author proves that if $G=GL_2$ or $G=SL_3$, then the poset $\mathcal{N}(G)_x$ is a ranked lattice. For another special case, if $x=t^{\lambda}$, then $\mathcal{N}(G)_x = \{ \lambda^+\}$ is a single element set; see Corollary 9.2.1 in \cite{GHKRadlvs}.  The precise relationship between $\mathcal{N}(G)_x$ and $\mathcal{N}(G)$ in general remains quite opaque. For example, outside of these special cases, it is not known for which $x$ the poset $\mathcal{N}(G)_x$ is a full subinterval of $\mathcal{N}(G)$.
 
 We remark that if we use a different decomposition on $G(F)$ than the Bruhat decomposition from \eqref{E:Bruhat}, then the study of the corresponding posets of Newton points can become simpler.  For example, using the Cartan decomposition into double cosets of the maximal compact subgroup $K=G(\mathcal{O})$, the Newton poset for elements in $Kt^\lambda K$ consists of all $\mu \in (Q^\vee \otimes_{\Z}\Q)^+$ such that $\mu \leq \lambda^+$, subject to the the integrality condition that the denominator of any rational slope divides its multiplicity.  That is, in this case one always obtains a full subinterval of $\mathcal{N}(G)$, as was proved in the sequence of papers \cite{KRFcrystals, Luc, Gas, GasGLn}.

\end{subsection}

\begin{subsection}{Maximal Newton points}\label{S:MaxNPs}

While the poset $\mathcal{N}(G)_x$ remains rather mysterious in many ways, we review the well-known fact that it possesses a unique maximum element.

\begin{claim}
The poset $\mathcal{N}(G)_x$ contains a unique maximum element.
\end{claim}

\begin{proof}
For a fixed $x \in \widetilde{W}$, the double coset $IxI$ is irreducible. Denote by $G_{\lambda} = \{ g \in G \mid \nu(g) =\lambda\}$, and consider the intersections $(IxI)_{\lambda} := G_{\lambda}\cap IxI$.  Then $IxI$ is the finite union of subsets of the form $(IxI)_{\lambda}$, any two of which are disjoint. If $\lambda \in \mathcal{N}(G)_x$ is maximal, then $(IxI)_{\lambda}$ is an open subset of $IxI$. But since $IxI$ is irreducible, there must exist a unique maximum element in $\mathcal{N}(G)_x$.
\end{proof}

\begin{defn}
Given $x \in \widetilde{W}$, we define the \emph{maximum Newton point} $\nu_x \in (Q^\vee \otimes_{\Z}\Q)^+$, to be the unique maximum element in $\mathcal{N}(G)_x$; \textit{i.e.} for all $\lambda \in \mathcal{N}(G)_x$, we have $ \nu_x \geq \lambda$.  
\end{defn}

The first closed formula for the unique maximum element $\nu_x$ in $\mathcal{N}(G)_x$ was proved by Viehmann, who reduced the computation to an interplay between the combinatorics of two natural partial orderings associated to elements of the affine Weyl group.
\begin{theorem}[Corollary 5.6 \cite{VieTrunc}]\label{T:Vform}
The maximum Newton point associated to $x\in \widetilde{W}$ is
\begin{equation}\label{E:Vform}\nu_x = \max \{ \nu(y) \mid y \in \widetilde{W},\  y \leq x \},
\end{equation} 
where the maximum is taken with respect to dominance order and the elements $y$ and $x$ are related by Bruhat order.
\end{theorem}
\noindent Although elegant, in practice this formula is difficult to implement without the aid of a computer except in certain special cases, involving first finding every element less than the fixed affine Weyl group element $x$ in Bruhat order, and further computing and comparing the Newton points for each of those elements. Our main theorem, which is formally stated in Section \ref{S:MainPrecise}, provides a combinatorial formula for $\nu_x$ which may be directly computed by hand, in addition to having an easy implementation by computer.

\end{subsection}


\section{The Quantum Bruhat Graph and Applications}\label{S:QBG}

The main result in this paper shows that there is a closed combinatorial formula for the maximum element in the poset of Newton points in terms of paths in the \emph{quantum Bruhat graph}.  The nomenclature comes from the fact that this graph was introduced by Brenti, Fomin, and Postnikov in \cite{BFP} to capture the multiplicative structure of the quantum cohomology ring of the complex flag variety, in particular the Chevalley-Monk rule for multiplying by a divisor class.

\begin{subsection}{The quantum Bruhat graph}\label{S:QBGdef}

We now formally define the quantum Bruhat graph, which will be our primary combinatorial tool.  The vertices are given by the elements of the finite Weyl group $w \in W$.   Two elements are connected by an edge if they are related by a reflection satisfying one of two ``quantum relations.''   More precisely, there is a directed edge $w \longrightarrow wr_{\alpha}$ for some $\alpha \in R^+$ if and only if one of two length relationships between $w$ and $wr_{\alpha}$ is satisfied:
\begin{align*}w\ \textcolor{blue}{\longrightarrow}\ wr_{\alpha}\ \ \text{if}\ \ & \ell(wr_{\alpha}) = \ell(w)+1,\ \ \text{or}  \\
w\ \textcolor{red}{\longrightarrow}\ wr_{\alpha}\ \ \text{if}\ \ & \ell(wr_{\alpha}) = \ell(w)-\langle \alpha^{\vee}, 2 \rho \rangle +1.
\end{align*}
 The first type of edges are simply those corresponding to covers in the usual Hasse diagram for the strong Bruhat order on $W$.  The second type of edges, all of which are directed downward in the graph, are ``quantum'' edges coming from the quantum Chevalley-Monk formula of \cite{Pet}. The edges are then labeled by the root corresponding to the reflection used to get from one element to the other, so that the edge from $w \longrightarrow wr_{\alpha}$ is labeled by $\alpha$.  Figure \ref{fig:S_3QBG} shows the quantum Bruhat graph for $W=S_3$, in which we abbreviate $s_is_j$ simply as $s_{ij}$.
\begin{figure}
\begin{center}
 \resizebox{1.9in}{!}
{
\begin{overpic}{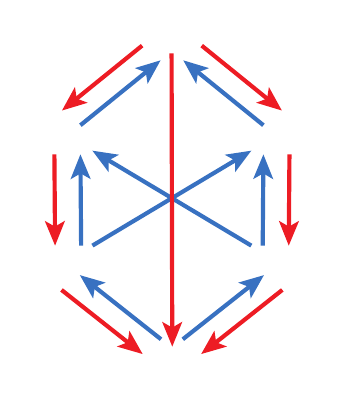}
\put(65.5,63){\bf \huge{$s_{21}$}}
\put(9,63.5){\bf \huge{$s_{12}$}}
\put(65.5,31){\bf \huge{$s_{2}$}}
\put(11.5,31){\bf \huge{$s_{1}$}}
\put(41.5,4){\huge{1}}
\put(35,90.5){\bf \huge{$s_{121}$}}
\put(20,15){\bf \large{\textcolor{dkgreen}{$\alpha_1$}}}
\put(62,15){\bf \large{\textcolor{dkgreen}{$\alpha_2$}}}
\put(75,49){\bf \large{\textcolor{dkgreen}{$\alpha_1$}}}
\put(6.5,49){\bf \large{\textcolor{dkgreen}{$\alpha_2$}}}
\put(20,83){\bf \large{\textcolor{dkgreen}{$\alpha_1$}}}
\put(60.5,83){\bf \large{\textcolor{dkgreen}{$\alpha_2$}}}
\put(35,40){\bf \large{\textcolor{dkgreen}{$\alpha_1$}}}
\put(48.5,40){\bf \large{\textcolor{dkgreen}{$\alpha_2$}}}
\put(42.5,40){\bf \large{\textcolor{dkgreen}{$+$}}}
\end{overpic}
}
\caption{The quantum Bruhat graph for $S_3$.}\label{fig:S_3QBG}
\end{center}
\end{figure}
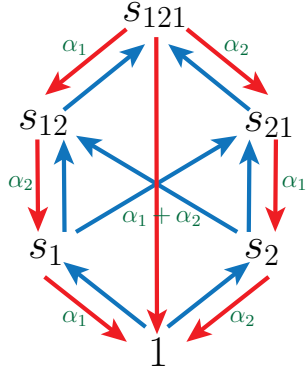

We now define the weight of any path in the quantum Bruhat graph.   For an edge $w\ \textcolor{blue}{\longrightarrow}\ wr_{\alpha}$ resulting from the relation $\ell(wr_{\alpha}) = \ell(w)+1$, there is no contribution to the weight.  On the other hand, an edge $w\ \textcolor{red}{\longrightarrow}\ wr_{\alpha}$ arising from the relation $\ell(wr_{\alpha}) = \ell(w)-\langle \alpha^{\vee}, 2 \rho \rangle +1$ contributes a weight of $\alpha^{\vee}$. The \emph{weight of a path in the quantum Bruhat graph} is then defined to be the sum of the weights of the edges in the path.  For example, in Figure \ref{fig:S_3QBG}, the weight of any of the three paths of minimal length from $s_{12}$ to $s_2$, all of which have length 3, equals $\alpha^{\vee}_1+\alpha^{\vee}_2$. 

It will also sometimes be convenient to record the weight of a path in the quantum Bruhat graph as a vector in $\Z^r$, where recall that $r$ is the rank of $G$. If we express the weight $\mu \in Q^\vee$ of a path in terms of the basis of simple coroots, say $\mu = d_1\alpha^{\vee}_1 + \cdots + d_r\alpha^{\vee}_r$, then we define $d = (d_1, \dots, d_r) \in \Z^r$ and equivalently also refer to the vector $d$ as the weight of the path.  The purpose of this alternative is that we may then associate monomials in a certain set of commuting variables $q_1, \dots, q_r$ to each path.  In particular, denote by $q^d$ the monomial $q_1^{d_1}\cdots q_r^{d_r}$; see Section \ref{S:qSch} for the motivation.

We now record several combinatorial statements about paths in the quantum Bruhat graph from \cite{PostQBG}, whose proofs rely on the combinatorics of the tilted Bruhat order introduced in \cite{BFP}.  

\begin{prop}[Lemma 1, Theorem 2 \cite{PostQBG}]\label{T:QBGcomb}
Let $u, v \in W$ be any Weyl group elements. 
\begin{enumerate}
\item There exists a directed path from $u$ to $v$ in the quantum Bruhat graph.
\item The length of a shortest path between any two elements in the quantum Bruhat graph it at most $\ell(w_0)$. 
\item All shortest paths from $u$ to $v$ have the same weight, say $d_{\text{min}}$.
\item If $d$ is the weight of any path from $u$ to $v$, then $d_{\text{min}} \leq d$ or equivalently $q^d$ is divisible by $q^{d_{\text{min}}}$.
\end{enumerate}
\end{prop}

We use a special case of property (3) concerning the uniqueness of the weight of a minimal length path in the proof of Theorem \ref{T:main}, although we point out that an independent proof of the final step in Proposition \ref{T:trans} would provide a \emph{geometric} explanation for each of these combinatorial properties of the quantum Bruhat graph.  Further, using the automorphisms of the Iwahori subgroup discussed in \cite{Be1}, the author expects that it should also be possible to provide independent geometric proofs of the symmetries of the quantum Bruhat graph for $W = S_n$ appearing in \cite{PostSymm}, which in turn correspond to symmetries of Gromov-Witten invariants.

\end{subsection}

\begin{subsection}{Statement of the main theorem}\label{S:MainPrecise}

We are now prepared to formally state our main result, which provides a readily computable combinatorial formula for the maximum Newton point $\nu_x$ in $\mathcal{N}(G)_x$. The general shape of $\nu_x$ for $x=t^{v\lambda}w$ with $\lambda$ dominant is that $\nu_x = \lambda - \mu$, where $\mu$ is a correction factor obtained by looking at the weight of any minimal length path in the quantum Bruhat graph between two vertices uniquely determined by the pair of finite Weyl group elements associated to $x$. Figures \ref{fig:SL3MaxNPs} and \ref{fig:S_3QBG} together illustrate Theorem \ref{T:main} in the case of $G=SL_3$.

\begin{theorem}\label{T:main}
Let $x = t^{v\lambda}w \in \widetilde{W}$, and consider any path of minimal length $k$ from $w^{-1}v$ to $v$ in the quantum Bruhat graph for $W$. If $\langle \lambda, \alpha_i \rangle > M$ for all simple roots $\alpha_i \in \Delta$, then the maximum Newton point associated to $x$ equals 
\begin{equation}\nu_x = \lambda - \alpha^{\vee}_x,\end{equation} where $\alpha^{\vee}_x$ is the weight of the chosen path from $w^{-1}v$ to $v$.
\end{theorem}

\noindent Recall that both here and throughout the paper, $M$ is the constant defined in \eqref{E:MFormula}.

We point out that Conjecture 2 from \cite{Be1} about elements in the $v=w_0$ Weyl chamber follows as an immediate corollary of Theorem \ref{T:main}, since all minimal paths to $w_0$ in the quantum Bruhat graph consist exclusively of upward edges and thus carry no weight.  On the other hand, this observation can also already be made as a consequence of Theorem \ref{T:Vform}, since in the antidominant chamber the translations are the minimal length coset representatives for $\widetilde{W}/W$.

The superregularity hypothesis on $\lambda$ stated in Theorem \ref{T:main} is the sharpest that the current method of proof permits; see Section \ref{S:HypRmks} for a detailed discussion of the components of the proof which introduce the required superregularity conditions.  We remark, however, that property (2) of Proposition \ref{T:QBGcomb} provides a uniform bound on $k$ for any finite Weyl group $W$.  In particular, we know that $k \leq \ell(w_0)$, and so the uniform superregularity hypothesis $\langle \lambda, \alpha_i \rangle > 8\ell(w_0)$ for all $\alpha_i \in \Delta$ implies the stated hypothesis involving $M$ for all classical groups.  Of course, for any given pair $w,v \in W$, the minimum length of any path from $w^{-1}v$ to $v$ might be considerably shorter than $\ell(w_0)$, and thus Theorem \ref{T:main} places a strictly weaker superregularity hypothesis on $\lambda$.  However, for the reader interested in a uniform statement for any $x \in \widetilde{W}$ in a fixed affine Weyl group, we make this observation formal in the following immediate corollary.

\begin{cor}\label{T:MainCor}
Let $x = t^{v\lambda}w \in \widetilde{W}$, and suppose that  for all simple roots $\alpha_i \in \Delta$
\begin{equation}
\langle \lambda, \alpha_i \rangle >
\begin{cases}
8\ell(w_0) & \text{if $G$ is classical,}\\
16\ell(w_0)  & \text{if $G$ is exceptional.}
\end{cases}
\end{equation} 
Then the maximum Newton point associated to $x$ equals $\nu_x = \lambda - \alpha^{\vee}_x,$ where $\alpha^{\vee}_x$ is the weight of any path of minimal length from $w^{-1}v$ to $v$ in the quantum Bruhat graph for $W$. 
\end{cor}

We remark that one can also formally restate Corollaries \ref{T:QSchNP} and \ref{T:ADLVapp} in the sections which follow using the same uniform superregularity hypothesis as in Corollary \ref{T:MainCor}.

\end{subsection}

\begin{subsection}{Connections to quantum Schubert calculus}\label{S:qSch}

We now discuss a surprising connection between the main result about Newton points, which arises from questions in the algebraic geometry of affine flag varieties in characteristic $p>0$, to the quantum cohomology of standard complete flag varieties over $\C$. Given a complex reductive group $G$, the classical cohomology of the complete flag variety $G/B$ over $\C$ is a free $\Z$-module generated by Schubert classes, which are indexed by elements in the Weyl group $W$.  If we define $\Z[q]:=\Z[q_1, \dots, q_r]$, then the quantum cohomology ring of $G/B$ equals $QH^*(G/B) = H^*(G/B, \Z) \otimes_{\Z} \Z[q]$ as a $\Z[q]$-module, and will also have a $\Z[q]$-basis of Schubert classes $\sigma_w$ where $w \in W$.  The main problem in modern quantum Schubert calculus is to explicitly compute the products \begin{equation}\label{E:QSchProb}\sigma_u * \sigma_v = \sum_{w,d} c^{w,d}_{u,v} q^d \sigma_w,\end{equation} by finding non-recursive, positive combinatorial formulas for the \emph{Gromov-Witten invariants} $c^{w,d}_{u,v}$ and the quantum parameters $q^d$.  Roughly speaking, these Gromov-Witten invariants count the number of curves of \emph{degree} $d$ meeting a triple of Schubert varieties determined by $u,v,w \in W$.  
 
Thereom \ref{T:main} turns out to be related to the question of determining which degrees arise in the product of two Schubert classes.  In \cite{PostQBG}, Postnikov strengthens a result of \cite{FW}, which both proves the existence of and then provides a combinatorial formula for the unique minimal monomial $q^d$ which occurs with nonzero coefficient in any quantum Schubert product.

\begin{theorem}[Corollary 3 \cite{PostQBG}]\label{T:Minq^d}
Given any pair $u, v \in W$, the unique minimal monomial that occurs in the quantum product $\sigma_{u} * \sigma_{v}$ equals $q^d$, where $d$ is the weight of any path of minimal length in the quantum Bruhat graph from $u$ to $w_0v$.  
\end{theorem}

  The following corollary relates this result in quantum Schubert calculus to our problem of finding maximal Newton points.

\begin{cor}\label{T:QSchNP}
Fix any $u,v \in W$, and define $k$ to be the length of any minimal path from $u$ to $w_0v$ in the quantum Bruhat graph.  Let $\lambda \in Q^\vee$ be any coroot such that $\langle \lambda, \alpha_i \rangle > M$ for all $\alpha_i \in \Delta$.  Then the following are equivalent:
\begin{enumerate}
\item $q^d=q_1^{d_1}\cdots q_r^{d_r}$ is the minimal monomial in the quantum product $\sigma_u * \sigma_v$ 
\item $\lambda- d_1\alpha_1^\vee - \cdots - d_r\alpha_r^\vee$ is the maximum Newton point in $\mathcal{N}(G)_x,$ where $x = t^{w_0v(\lambda)}w_0vu^{-1}$. 
\end{enumerate}
\end{cor}

\begin{proof}
By Theorem \ref{T:Minq^d}, $q^d$ is the minimal monomial in the quantum product $\sigma_u * \sigma_v$ if and only if $d$ is the weight of any path of minimal length in the quantum Bruhat graph from $u$ to $w_0v$.  Set $v' = w_0v$ and $w' = w_0vu^{-1}$, and compute that $(w')^{-1}v' = (w_0vu^{-1})(w_0v)=u$.  Therefore, Theorem \ref{T:main} says that the weight of any such path also gives the correction factor required to calculate $\nu_x$ for $x = t^{v'\lambda}w'$.
\end{proof}

\end{subsection}

\begin{subsection}{Affine Deligne-Lusztig varieties and Mazur's inequality}\label{S:ADLVs}

In \cite{DL}, Deligne and Lusztig constructed a family of varieties $X_w$ in $G(\overline{\F_q})/B$ indexed by elements $w\in W$ to study the representation theory of finite Chevalley groups. Rapoport introduced \emph{affine Deligne-Lusztig varieties} in \cite{RapSatake}, defined as generalizations of these classical Deligne-Lusztig varieties. Although Deligne and Lusztig's original construction was motivated by applications to representation theory \cite{LuszChev}, interest in affine Deligne-Lusztig varieties is rooted in their intimate relationship to reductions modulo $p$ of Shimura varieties, among other arithmetic applications, many of which lie at the heart of the Langlands program; see \cite{RapShimura}.

For $x \in \widetilde{W}$ and $b \in G(F)$, the associated affine Deligne-Lusztig variety is defined as \begin{equation} X_x(b) := \{ g \in G(F)/I \mid g^{-1}b\sigma(g) \in IxI \}. \end{equation}
Unlike in the classical case in which Lang's Theorem automatically says that $X_w$ is non-empty for every $w \in W$, affine Deligne-Lusztig varieties frequently tend to be empty.  Providing a complete characterization for the pairs $(x,b)$ for which the associated affine Deligne-Lusztig variety is non-empty has proven to be a surprisingly challenging problem.  In the context of affine Deligne-Lusztig varieties inside the affine Grassmannian, the non-emptiness question is phrased in terms of \emph{Mazur's inequality}, which relates the coroot $\lambda$ from the translation part of $x$ and the Newton point of $b$; see \cite{Maz, Kat}.  If $\nu(b)$ denotes the Newton point associated to $b$, Mazur's inequality says that $\nu(b) \leq \lambda^+$.

Although no simple analog of Mazur's inequality can perfectly predict whether or not $X_x(b)$ is non-empty, Theorem \ref{T:main} yields an Iwahori analog of Mazur's inequality, providing a necessary condition for non-emptiness under a superregularity hypothesis on the coroot. The following corollary can be viewed as a refinement of Mazur's inequality on the affine Grassmannian for the context of the  affine flag variety.

\begin{cor}\label{T:ADLVapp}
Let $x = t^{v\lambda}w \in \widetilde{W}$, and consider any path of minimal length $k$ from $w^{-1}v$ to $v$ in the quantum Bruhat graph for $W$.  Suppose that $\langle \lambda, \alpha_i \rangle > M$ for all simple roots $\alpha_i \in \Delta$.  Fix $b \in G(F)$, and denote by $\nu(b)$ the Newton point for $b$. If $X_x(b)$ is non-empty, then \begin{equation}\label{E:IwMaz} \nu(b) \leq \lambda - \alpha^{\vee}_x,\end{equation} where $\alpha^{\vee}_x$ is the weight of the chosen path from $w^{-1}v$ to $v$.
\end{cor}

\begin{proof}
Denote by $[b]$ the $\sigma$-conjugacy class of $b$.  By definition, if $X_x(b) \neq \emptyset$, then $[b] \cap IxI \neq \emptyset$ as well.  For any $g \in [b]\cap IxI$, since the Newton point is a $\sigma$-conjugacy class invariant, we know that $\nu(g) = \nu(b)$.  By maximality of $\nu_x$ and the fact that $g \in IxI$, we thus also have $\nu(b) \leq \nu_x.$  Finally, recall that $\nu_x = \lambda - \alpha^\vee_x$ by Theorem \ref{T:main}.
\end{proof}

\begin{remark}\label{R:padic}
For readers primarily interested in applications to Shimura varieties, we remark that the non-emptiness statements obtained by translating this analog of Mazur's inequality to the corresponding affine Deligne-Lusztig varieties also holds when $F$ is the maximal unramified extension of the field of $p$-adic numbers, since the non-emptiness questions for $X_x(b)$ and $X_x(b)_{\Q_p}$ were shown to be equivalent in \cite{GHKRadlvs}.  We also refer such readers to the remarks in Section \ref{S:HypRmks}, where we make explicit the extent to which the superregularity hypothesis can be relaxed beyond those bounds recorded in the results formally stated in Section \ref{S:QBG}.
\end{remark}

\end{subsection}


\section{The Quantum Bruhat Graph and Affine Bruhat Order}\label{S:QBGChains}

This section generalizes a result of Lam and Shimozono from \cite{LS} which we reformulate as Proposition \ref{T:cocover}, proving that covering relations in affine Bruhat order correspond to edges in the quantum Bruhat graph.  In particular, if one is interested in singling out translations below a given $x \in \widetilde{W}$, then iterated application of this observation yields a correspondence stated in Proposition \ref{T:QBGpaths} between saturated chains in affine Bruhat order and paths in the quantum Bruhat graph.  Proposition \ref{T:QBGpaths} is the first of two key propositions required for the proof of Theorem \ref{T:main}.

\subsection{Edges and covering relations}

For affine Weyl group elements whose translation part is superregular, we now discuss a correspondence between edges in the quantum Bruhat graph and covering relations in affine Bruhat order.  This correspondence plays a central role in verifying the equivariant quantum Chevalley-Monk rule in the proof of the Peterson isomorphism in \cite{LS}.  

We start by recalling a standard length formula for affine Weyl group elements written in terms of the translation part and the pair of naturally associated finite Weyl group elements.

\begin{lemma}[Lemma 3.4 \cite{LS}]\label{T:xlength}
Let $\lambda \in Q^+$ be regular dominant, and let $x = t^{v\lambda}w \in \widetilde{W}$.  Then
\begin{equation}\label{E:xlength}
\ell(x) = \ell(t^{\lambda}) - \ell(v^{-1}w) + \ell(v) = \ell(t^{\lambda}) - \ell(w^{-1}v) + \ell(v) = \langle \lambda, 2\rho \rangle - \ell(w^{-1}v)+\ell(v).
\end{equation}
\end{lemma}

When we write $x \gtrdot y$, we mean that $x$ \emph{covers} $y$ in Bruhat order; \textit{i.e.} both $x \geq y$ and $\ell(x) = \ell(y) + 1$, or equivalently we say that $y$ is a \emph{cocover} of $x$.

\begin{prop}[Reformulation of Proposition 4.4 \cite{LS}]\label{T:cocover}
Let $x = t^{v\lambda}w \in \widetilde{W}$, and let $r_{\beta} = t^{nv\alpha^{\vee}}r_{v\alpha}$ be the affine reflection that reflects $x$ across the $H_{v\alpha,n}$ hyperplane, where $\alpha \in R^+$.  Further suppose that for all $\alpha_i \in \Delta$,
\begin{equation}\label{E:cocoverhyp}
\langle \lambda, \alpha_i \rangle > 
\begin{cases}
2\ell(w_0)+2 & \text{if $G \neq G_2$},\\
3\ell(w_0)+3 & \text{if $G = G_2$}.
\end{cases}
\end{equation}
  Then $x \gtrdot r_{\beta}x$ is a covering relation if and only if one of the following four conditions holds:
\begin{enumerate}
\item  $\ell(vr_\alpha)=\ell(v)-1$ and $n=0$, in which case $r_\beta x = t^{vr_\alpha(\lambda)}r_{v\alpha}w$.
\item $\ell(vr_\alpha) = \ell(v)+\langle \alpha^{\vee}, 2\rho\rangle -1$ and $n=1$, in which case $r_\beta x= t^{vr_{\alpha}(\lambda- \alpha^{\vee})}r_{v\alpha}w$.
\item $\ell(w^{-1}vr_\alpha) = \ell(w^{-1}v)+1$ and $n=\langle \lambda, \alpha \rangle$, in which case $r_\beta x = t^{v(\lambda)}r_{v\alpha}w$.
\item $\ell(w^{-1}vr_\alpha) = \ell(w^{-1}v)-\langle \alpha^{\vee}, 2\rho\rangle+1$ and $n=\langle \lambda, \alpha \rangle-1$, in which case $r_\beta x = \linebreak t^{v(\lambda - \alpha^{\vee})}r_{v\alpha}w$.
\end{enumerate}
\end{prop}

\begin{figure}
\centering
\begin{overpic}[width=0.26\linewidth]{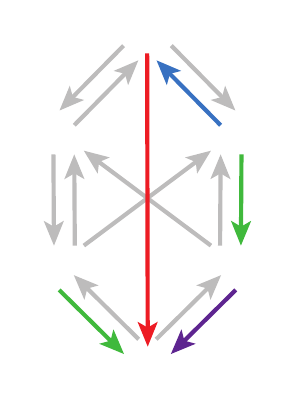}
\put(56,64){\bf \huge{$s_{21}$}}
\put(8,63.5){\bf \huge{$s_{12}$}}
\put(56.5,30.5){\bf \huge{$s_{2}$}}
\put(9,31){\bf \huge{$s_{1}$}}
\put(35,3){\huge{1}}
\put(29,91.5){\bf \huge{$s_{121}$}}
\put(15,16){\bf $\alpha_1$}
\put(54,16){\bf $\alpha_2$}
\put(63,49){\bf $\alpha_1$}
\put(53,81){\bf $\alpha_2$}
\put(29,38){\bf $\alpha_1$}
\put(36.5,38){\bf $+$}
\put(42,38){\bf $\alpha_2$}
\end{overpic}
\hspace{20pt}
 \begin{overpic}[width=0.45\linewidth]{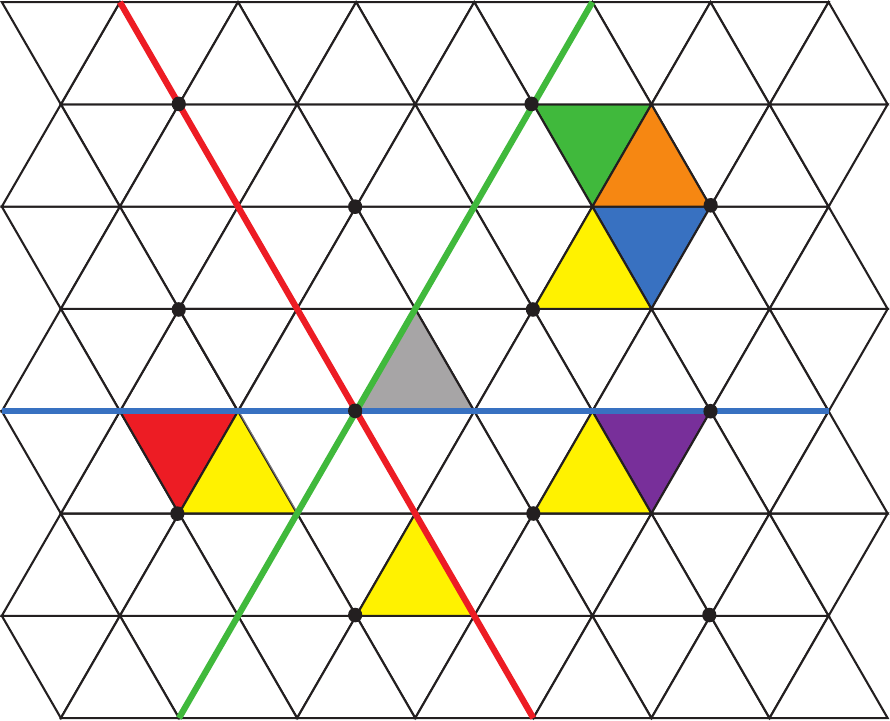}
\put(70.5,60){\bf \huge{$x$}}
\end{overpic}
\caption{Cocovers of $x=t^{2\rho^\vee}s_{12}$ with corresponding edges in the quantum Bruhat graph drawn in the same color; maximal translations less than $x$ are yellow.}
\label{fig:QBGchain}
 \end{figure}
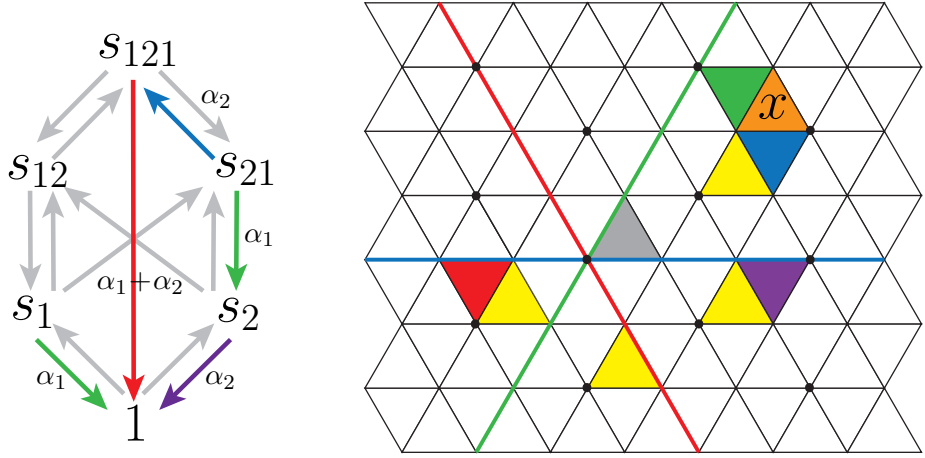

\begin{remark}\label{T:cocoverQBG}
Observe that each of the length conditions in the four cases of this proposition corresponds to an edge in the quantum Bruhat graph.  More precisely, writing the affine reflection $r_\beta = t^{nv\alpha^{\vee}}r_{v\alpha}$, we have the following association between the four cases in Proposition \ref{T:cocover} and edges in the quantum Bruhat graph:
\begin{enumerate}
\item $x \gtrdot r_\beta x$ corresponds to an upward edge into $v$ of the form $vr_{\alpha} \longrightarrow v$
\item $x \gtrdot r_\beta x$ corresponds to a downward edge into $v$ of the form $vr_{\alpha} \longrightarrow v$
\item $x \gtrdot r_\beta x$ corresponds to an upward edge out of $w^{-1}v$ of the form $w^{-1}v \longrightarrow w^{-1}vr_{\alpha}$
\item $x \gtrdot r_\beta x$ corresponds to a downward edge out of $w^{-1}v$ of the form $w^{-1}v \longrightarrow w^{-1}vr_{\alpha}$
\end{enumerate}
\end{remark}

\begin{example}
Before proceeding with the proof, we provide an example which illustrates the correspondence established by Proposition \ref{T:cocover} and Remark \ref{T:cocoverQBG}.  Consider $x=t^{2\rho^\vee}s_{12}$, which is the orange alcove in Figure \ref{fig:QBGchain}.  In order to find all cocovers $r_{\beta}x \lessdot x$, we look at edges going \emph{into} $v=1$ and \emph{out of} $w^{-1}v=s_{21}$ in the quantum Bruhat graph; these are the five colored edges in Figure \ref{fig:QBGchain}.  The three downward edges into $1$ correspond to three cocovers of type (2), each coming from a different type of reflection, and the corresponding alcoves are colored green, purple, and red, respectively.  There is one upward edge and one downward edge directed out of $s_{21}$, corresponding to cocovers of type (3) and (4), respectively.  Note that two different edges give the same green cocover which shares a face with $x$---this situation can arise when $x$ is close to the wall of a Weyl chamber.  Proposition \ref{T:QBGpaths} below explains how each of the two minimal length \emph{paths} from $w^{-1}v=s_{21}$ to $v=1$ give rise to 8 different chains of length 2 from $x$ to one of the yellow translation alcoves.  Importantly, observe that these translations are based at coroots which all lie in the same $W$-orbit.
\end{example}

Our proof of Proposition \ref{T:cocover} closely follows the strategy in the proof of Proposition 4.4 in \cite{LS}, but we include the details both in order to extract a precise superregularity hypothesis on $\lambda$, and also because we adopt several different conventions in this paper.

\begin{proof}[Proof of Proposition \ref{T:cocover}] 
Let $x = t^{v\lambda}w \in \widetilde{W}$, and let $r_{\beta} = t^{nv\alpha^{\vee}}r_{v\alpha}$ be the affine reflection that reflects $x$ across the $H_{v\alpha,n}$ hyperplane.
First compute directly that
\begin{equation}
r_{\beta}x  = t^{nv\alpha^{\vee}}r_{v\alpha}t^{v\lambda}w  = t^{(nv\alpha^{\vee} + r_{v\alpha}v\lambda)}r_{v\alpha}w = t^{v(n\alpha^{\vee}+r_{\alpha}\lambda)}r_{v\alpha}w.
\end{equation}
We can then rewrite this expression in two equivalent ways by either factoring out $r_\alpha$ or using the action of $r_\alpha$ on $\lambda$:
\begin{equation}\label{E:cocoverform}
r_{\beta}x = t^{vr_{\alpha}(\lambda - n\alpha^{\vee})}r_{v\alpha}w = t^{v(\lambda - (\langle \lambda, \alpha\rangle-n)\alpha^{\vee})}r_{v\alpha}w.
\end{equation}
Further suppose that $r_\beta x$ is a cocover of $x$ so that $\ell(x)-\ell(r_\beta x) = 1$. 

Following \cite{LS}, define the function $f: \Z \to \Z_{\geq 0}$ by
\begin{equation}\label{E:fvalues}
f(n) = \ell(t^{v(\lambda-n\alpha^\vee)}) = \ell(t^{vr_\alpha(\lambda-n\alpha^\vee)}) = \ell( t^{v(\lambda - (\langle \lambda, \alpha\rangle-n)\alpha^{\vee})}),
\end{equation}
and note that $f(0) =\ell(t^{v\lambda})=\langle \lambda, 2\rho \rangle = f(\langle \lambda, \alpha \rangle)$ since $\lambda$ is dominant.  As shown in the proof of \cite[Proposition 4.4]{LS}, the function $f$ is a convex function of $n$.
Informally, the idea of this proof is that the function value $f(0)=f(\langle \lambda, \alpha \rangle)$ gives a reasonable approximation for $\ell(x)$, whereas $f(n)$ gives a reasonable approximation for $\ell(r_\beta x)$.  Therefore,  since $r_\beta x$ is a cocover of $x$ and their lengths differ by exactly 1, we will see that $f(n)$ cannot be too far away from either $f(0)$ or $f(\langle \lambda, \alpha \rangle)$, and consequently deduce that $n$ is also quite close to either 0 or $\langle \lambda, \alpha \rangle$.  We shall make each of these claims precise in the argument that follows.

Note that since $x = t^{v\lambda}w$ and $f(0) = \ell(t^{v\lambda})$, then 
\begin{equation}\label{E:xball}
f(0) - \ell(w_0) \leq \ell(x) \leq f(0)+\ell(w_0).
\end{equation}
Similarly, since $r_\beta x =  t^{vr_{\alpha}(\lambda - n\alpha^{\vee})}r_{v\alpha}w$ and $f(n) =  \ell(t^{vr_\alpha(\lambda-n\alpha^\vee)}),$ then 
\begin{equation}\label{E:cocoverball}
f(n) - \ell(w_0) \leq \ell(r_\beta x) \leq f(n)+\ell(w_0).
\end{equation}
Subtracting \eqref{E:cocoverball} from \eqref{E:xball} and using the fact that $\ell(x)-\ell(r_\beta x) = 1$, we see that $-2\ell(w_0)+1 \leq f(0)-f(n) \leq 2\ell(w_0)+1$, and so the maximum distance between $f(0)$ and $f(n)$ is $2\ell(w_0)+1$, or 
\begin{equation}\label{E:ballineq}
|f(0)-f(n)| \leq 2 \ell(w_0)+1.
\end{equation}
 Similarly, since $f(\langle \lambda, \alpha \rangle) = f(0)$, by the same argument, we have
\begin{equation}\label{E:ballineq2}
|f(\langle \lambda, \alpha\rangle) - f(n) | \leq 2 \ell(w_0)+1.
\end{equation}
We proceed next to analyze the function $f$ locally in the neighborhoods around 0 and $\langle \lambda, \alpha \rangle$.

First consider the case where $|n| \leq \ell(w_0)+1$, which means that we are momentarily focusing on $f$ locally in a neighborhood around 0.  We claim that whenever $|n| \leq \ell(w_0)+1$, then the coroot $\lambda - n\alpha^\vee$ is regular and dominant.  To prove this, we must show that $\langle \lambda - n\alpha^\vee, \alpha_i \rangle \geq 1$ for all $\alpha_i \in \Delta$.  Directly compute that by our hypothesis on $\langle \lambda, \alpha_i \rangle$, we have 
\begin{equation}\label{E:dom}
\langle \lambda - n\alpha^\vee, \alpha_i \rangle = \langle \lambda, \alpha_i \rangle - n \langle\alpha^\vee, \alpha_i \rangle \geq 
\begin{cases}
(2 \ell(w_0)+3) - 2n,  & \text{if $G \neq G_2$},\\
(3 \ell(w_0)+4)-3n, & \text{if $G = G_2$,}
\end{cases} 
\end{equation}
where we have also used the fact that the maximum value of $\langle \beta^{\vee}, \alpha_i \rangle$  in any reduced root system equals 2 in every Lie type, except for $G_2$ in which it equals 3; see \cite[Ch.\ VI \textsection 1, no.\ 3]{Bour46}. But for $|n| \leq \ell(w_0)+1$, we further have that 
\begin{equation}
(2\ell(w_0)+3)-2n \geq (2\ell(w_0)+3)-2(\ell(w_0)+1) = 1,
\end{equation}
 and similarly for $G=G_2$.  Altogether, this shows that $\langle \lambda - n\alpha^\vee, \alpha_i \rangle \geq 1$ for all $\alpha_i \in \Delta$ whenever $|n| \leq \ell(w_0)+1$, and so $\lambda - n\alpha^\vee$ is both regular and dominant for these values of $n$.  Further, when $\lambda - n\alpha^\vee$ is dominant, we can write
\begin{equation}\label{E:domfn}
f(n) =  \ell(t^{vr_\alpha(\lambda-n\alpha^\vee)}) = \ell(t^{\lambda - n\alpha^\vee}) = \langle \lambda- n \alpha^\vee, 2\rho \rangle,
\end{equation}
and so 
\begin{equation}\label{E:diff1}
f(0) - f(n) = \langle \lambda, 2 \rho\rangle -  \langle \lambda- n \alpha^\vee, 2\rho \rangle = n \langle \alpha^\vee, 2\rho \rangle.
\end{equation}
In particular, for $-\ell(w_0)-1 \leq n \leq \ell(w_0)+1$, the function $f$ is linear with slope $-\langle \alpha^\vee, 2 \rho \rangle$.  
Since $\alpha \in R^+$, this slope is negative so that $f$ is decreasing on this interval.

We now apply the same argument to values of $n$ in the neighborhood of $\langle \lambda, \alpha \rangle$ to perform a similar local analysis on $f$ there.  We remark that this case follows by the symmetry of the function $f$, but we include the details for the sake of completeness. 
Consider the case where $| \langle \lambda, \alpha \rangle -n| \leq \ell(w_0)+1$.  As in Equation \eqref{E:dom}, for any $\alpha_i \in \Delta$ we have 
\begin{equation}\label{E:dom2}
\langle \lambda - (\langle \lambda, \alpha\rangle -n)\alpha^{\vee}, \alpha_i \rangle  \geq 
\begin{cases}
(2 \ell(w_0)+3) - 2(\langle \lambda, \alpha \rangle - n),  & \text{if $G \neq G_2$},\\
(3 \ell(w_0)+4)-3(\langle \lambda, \alpha \rangle - n), & \text{if $G = G_2$,}
\end{cases} 
\end{equation}
and by our momentary hypothesis $| \langle \lambda, \alpha \rangle -n| \leq \ell(w_0)+1$, we know that $(2(\ell(w_0))+3) - 2(\langle \lambda, \alpha \rangle - n) \geq 1$, and similarly for $G= G_2$.  Therefore, for these values of $n$, the coroot $\lambda - (\langle \lambda, \alpha\rangle -n)\alpha^{\vee}$ is also regular and dominant.  When $\lambda - (\langle \lambda, \alpha\rangle -n)\alpha^{\vee}$ is dominant, we can use the last version of Equation \ref{E:fvalues} to write
\begin{equation}\label{E:domfn2}f(n) = \ell(t^{v(\lambda-(\langle \lambda, \alpha \rangle -n)\alpha^\vee)}) = \langle \lambda-(\langle \lambda, \alpha \rangle -n)\alpha^\vee, 2\rho \rangle,
\end{equation}
and so 
\begin{equation}
f(\langle \lambda, \alpha \rangle) - f(n) = \langle \lambda, 2 \rho\rangle -  \langle \lambda-(\langle \lambda, \alpha \rangle -n)\alpha^\vee, 2\rho \rangle= (\langle \lambda, \alpha \rangle - n) \langle \alpha^\vee, 2\rho \rangle.
\end{equation}
In particular, for $\langle \lambda, \alpha \rangle - \ell(w_0)-1 \leq n \leq \langle \lambda, \alpha \rangle + \ell(w_0)+1$, the function $f$ is linear with slope $\langle \alpha^\vee, 2 \rho \rangle$.  Since $\alpha \in R^+$, this slope is positive so that $f$ is increasing on this interval.

Now consider the function $f$ globally.  First note that the interval $[0,\langle \lambda, \alpha \rangle]$ has length at least $2\ell(w_0)+3$ (respectively $3\ell(w_0) +4$ for $G=G_2$) by our hypothesis on $\langle \lambda, \alpha_i \rangle$.  Note therefore that the intervals $[-\ell(w_0)-1, \ell(w_0)+1]$ and $[\langle \lambda, \alpha \rangle - \ell(w_0)-1, \langle \lambda, \alpha \rangle + \ell(w_0)+1]$ are disjoint.  By the convexity of $f,$ we can only have one scenario for the global shape of $f$.   For values of $n$ between these two intervals $\ell(w_0)+1 < n < \langle \lambda, \alpha \rangle - \ell(w_0)-1$, we must have that  $f(n)\leq f(\ell(w_0)+1)$ by convexity.  Therefore, if $n \in (\ell(w_0)+1, \langle \lambda, \alpha \rangle - \ell(w_0)-1)$, then by Equation \eqref{E:diff1}
\begin{equation}
f(0)-f(n) \geq f(0)-f(\ell(w_0)+1) = (\ell(w_0)+1)\langle \alpha^\vee, 2\rho \rangle \geq 2 \ell(w_0)+2.
\end{equation}
In particular, for $n$ in this gap, we see that $f(0)-f(n) > 2\ell(w_0)+1$, and so comparing Equation \eqref{E:ballineq} we see that $r_\beta x$ cannot possibly be a cocover of $x$ for these values of $n$.  By the same convexity argument, if $n \in (-\infty, -\ell(w_0)-1)$, resp.~ $n \in (\langle \lambda, \alpha \rangle + \ell(w_0)+1, \infty)$, then the difference $|f(0)-f(n)|$, resp.~$|f(\langle \lambda, \alpha \rangle)-f(n)|$, is again too large for $r_\beta x$ to be a cocover of $x$.  

To summarize, if $x \gtrdot r_\beta x$, we may conclude thus far that either $|n| \leq \ell(w_0)+1$ or  $| \langle \lambda, \alpha \rangle -n| \leq \ell(w_0)+1$, or informally that $n$ is reasonably close to either 0 or $\langle \lambda, \alpha \rangle$.  In addition, we may then safely assume that either $\lambda - n\alpha^{\vee}$ is regular dominant or $\lambda - (\langle \lambda, \alpha\rangle -n)\alpha^{\vee}$ is regular dominant.  The remainder of the proof thus naturally breaks up into these two cases.

First suppose that $\lambda - n\alpha^{\vee}$ is regular dominant, or equivalently that $-\ell(w_0)-1 \leq n \leq \ell(w_0)+1$.  Applying the length formula in \eqref{E:xlength}, we obtain 
\begin{equation}\label{E:n0cocoverform}
\ell(r_{\beta}x) = \ell(t^{(\lambda - n\alpha^{\vee})}) - \ell((vr_{\alpha})^{-1}r_{v\alpha}w)+\ell(vr_\alpha) =  \langle \lambda - n\alpha^{\vee}, 2\rho\rangle - \ell(v^{-1}w)+\ell(vr_\alpha),
\end{equation}
where we have used that $r_{v\alpha} = vr_\alpha v^{-1}$ for any $v \in W$. Now  use Equation \eqref{E:xlength} to compare $\ell(x) = \langle \lambda, 2\rho \rangle - \ell(v^{-1}w)+\ell(v)$, and compute the difference
\begin{equation}
\ell(x)-\ell(r_{\beta}x) =  n\langle\alpha^{\vee}, 2\rho \rangle + \ell(v) - \ell(vr_\alpha).
\end{equation}
The element $r_\beta x$ is a cocover of $x$ if and only if $\ell(x)-\ell(r_\beta x) = 1$, which we see by the previous equation occurs if and only if
\begin{equation}
\ell(v) - \ell(vr_\alpha) = 1- n\langle\alpha^{\vee}, 2\rho \rangle. 
\end{equation}
The difference between $\ell(vr_\alpha)$ and $\ell(v)$ for any $v \in W$ is bounded by $\ell(r_\alpha) \leq \langle \alpha^{\vee}, 2\rho \rangle -1$.  (We remark that if $G$ is simply laced, then  in fact $\ell(r_\alpha) = \langle \alpha^{\vee}, 2\rho\rangle -1$.)  Therefore, we have
\begin{equation}\label{E:negcont}
|1-n\langle \alpha^\vee, 2\rho \rangle| \leq \langle \alpha^{\vee}, 2\rho \rangle -1.
\end{equation}

Now further suppose by contradiction that $n<0$.  Since $\langle \alpha^\vee, 2\rho\rangle \geq 2$ because $\alpha$ is a positive root, the expression $1-n\langle \alpha^\vee, 2\rho \rangle$ is positive when $n<0$, and so Equation \eqref{E:negcont} says that 
\begin{equation}
1-n\langle \alpha^\vee, 2\rho \rangle \leq  \langle \alpha^{\vee}, 2\rho \rangle -1 \quad \iff \quad 2 \leq (n+1)\langle \alpha^\vee, 2\rho\rangle.
\end{equation}
Note, however, that for any $n<0$, we have $(n+1)\langle \alpha^\vee, 2\rho \rangle \leq 0$, which is a contradiction.  Therefore, in the case in which $\lambda - n\alpha^\vee$ is dominant, in fact we have shown that $n \in [0,\ell(w_0)+1]$.
We can further conclude from Equation \eqref{E:negcont} that either $n=0$ or the expression $1-n\langle \alpha^\vee, 2\rho \rangle$ is negative, in which case
\begin{equation}
-1+n\langle \alpha^\vee, 2\rho \rangle \leq \langle \alpha^\vee, 2 \rho \rangle -1 \quad \iff \quad 0 \leq (1-n) \langle \alpha^\vee, 2 \rho \rangle,
\end{equation}
which only holds for $n=1$ since $n>0$ and $\langle \alpha^\vee, 2\rho \rangle \geq 2$.
 This means that there are really only two choices for $n$ in the situation when $x \gtrdot r_\beta x$ and $\lambda - n\alpha^\vee \in Q^+$; namely
\begin{equation}
n = \begin{cases}
0 & \text{and}\ \ \ell(v) - \ell(vr_\alpha) = 1 \\
1 & \text{and} \ \ \ell(v) - \ell(vr_\alpha)  = 1 -\langle \alpha^{\vee}, 2\rho \rangle.
\end{cases}
\end{equation}
Recall from Equation \eqref{E:cocoverform} that when $\lambda - n\alpha^\vee$ is dominant, we write $r_{\beta}x = t^{vr_{\alpha}(\lambda - n\alpha^{\vee})}r_{v\alpha}w$.  Substituting these two possible values for $n$ into this formula for $r_\beta x$, we obtain cases (1) and (2) of the proposition.

In the second case, we suppose that $\lambda - (\langle \lambda, \alpha \rangle-n)\alpha^{\vee} $ is regular dominant.  Applying Equation \eqref{E:xlength}, we see that 
\begin{equation}\label{E:cocoverform2}
\ell(r_{\beta}x) = \ell(t^{(\lambda - (\langle \lambda, \alpha \rangle-n)\alpha^{\vee})}) - \ell(v^{-1}r_{v\alpha}w)+\ell(v) =  \langle \lambda- (\langle \lambda, \alpha\rangle-n)\alpha^{\vee}, 2\rho\rangle - \ell(r_{\alpha}v^{-1}w)+\ell(v).
\end{equation}
We now compute the difference between $\ell(x)$ and the cocover $\ell(r_\beta x)$ to be
\begin{equation}
 \ell(x)-\ell(r_{\beta}x) = \ell(r_\alpha v^{-1}w) - \ell(v^{-1}w) +  \langle (\langle \lambda, \alpha \rangle -n)\alpha^{\vee}, 2\rho \rangle=1,
\end{equation}
which we rearrange to obtain
\begin{equation}
\ell(r_\alpha v^{-1}w) - \ell(v^{-1}w) = 1 +  (n-\langle \lambda, \alpha \rangle )\langle \alpha^{\vee}, 2\rho \rangle.
\end{equation}
Again use that $| \ell(r_\alpha u) - \ell(u)| \leq \langle \alpha^{\vee}, 2\rho \rangle -1$ for any $u \in W$ to see that 
\begin{equation}\label{E:absbound2}
|1 +  (n-\langle \lambda, \alpha \rangle )\langle \alpha^{\vee}, 2\rho \rangle| \leq  \langle \alpha^{\vee}, 2\rho \rangle -1.
\end{equation}

Now further suppose by contradiction that $n>\langle \lambda, \alpha \rangle$, which means that $(n-\langle \lambda, \alpha \rangle) \langle \alpha^\vee, 2\rho \rangle \geq 0$.  If $n> \langle \lambda, \alpha \rangle$, then Equation \eqref{E:absbound2} says that 
\begin{equation}
1 +  (n-\langle \lambda, \alpha \rangle )\langle \alpha^{\vee}, 2\rho \rangle \leq  \langle \alpha^{\vee}, 2\rho \rangle -1 \quad \iff \quad (n-\langle \lambda, \alpha \rangle - 1)\langle \alpha^\vee, 2\rho \rangle \leq -2.
\end{equation}
Note, however, that when $n > \langle \lambda, \alpha \rangle \geq 2$, then both $n-\langle \lambda, \alpha \rangle - 1$ and $\langle \alpha^\vee, 2\rho \rangle$ are nonnegative, and their product is thus also nonnegative, which is a contradiction.  Therefore, in the case in which $\lambda - (\langle \lambda, \alpha \rangle - n)\alpha^\vee$ is dominant, then in fact $n \in [\langle \lambda, \alpha \rangle - \ell(w_0)-1, \langle \lambda, \alpha \rangle]$.  We can further conclude from Equation \eqref{E:absbound2} that either $n = \langle \lambda, \alpha \rangle$ or the expression $1 +  (n-\langle \lambda, \alpha \rangle )\langle \alpha^{\vee}, 2\rho \rangle$ is negative, in which case 
\begin{equation}
(n-\langle \lambda, \alpha \rangle )\langle \alpha^{\vee}, 2\rho \rangle \leq  \langle \alpha^{\vee}, 2\rho \rangle -1 \quad \iff \quad 0 \leq (n-\langle \lambda, \alpha \rangle + 1 )\langle \alpha^{\vee}, 2\rho \rangle,
\end{equation}
which only holds for $n = \langle \lambda, \alpha \rangle -1$ since the expression $n-\langle \lambda, \alpha \rangle+1$ is negative for those other $n< \langle \lambda, \alpha \rangle$.  This means that there are really only two viable choices for $n$ in this case; namely,
\begin{equation}
n = \begin{cases}
\langle \lambda, \alpha \rangle & \text{and}\ \ \ell(r_\alpha v^{-1}w) - \ell(v^{-1}w) = 1 \\
\langle \lambda, \alpha \rangle -1 & \text{and} \ \ \ell(r_\alpha v^{-1}w) - \ell(v^{-1}w) = 1 -\langle \alpha^{\vee}, 2\rho \rangle.
\end{cases}
\end{equation}
Observe that $\ell(r_\alpha v^{-1}w) - \ell(v^{-1}w) = \ell(w^{-1}vr_\alpha) - \ell(w^{-1}v)$, recall from Equation \eqref{E:cocoverform} that when $\lambda - (\langle \lambda, \alpha \rangle - n)\alpha^\vee$ is dominant, we write $r_{\beta}x = t^{v(\lambda + (n-\langle \lambda, \alpha\rangle)\alpha^{\vee})}r_{v\alpha}w$.  Substituting these two possible values for $n$ into this formula for $r_\beta x$, we obtain cases (3) and (4) of the proposition.  

Conversely, if we are in either case (1) or (2) of the proposition, since $n\in \{0,1\}$ then $\lambda - n\alpha^\vee $ is regular dominant, so we may use Equation \eqref{E:n0cocoverform} to directly show that $\ell(r_\beta x) = \ell(x)-1$ in either of these two cases. Finally, if we are in either case (3) or (4) of the proposition, since $n\in \{\langle \lambda, \alpha \rangle, \langle \lambda, \alpha \rangle - 1\}$ then $\lambda - (\langle \lambda, \alpha \rangle - n)\alpha^\vee $ is regular dominant, so we may use Equation \eqref{E:cocoverform2} to directly show that $\ell(r_\beta x) = \ell(x)-1$ in either of these two cases, concluding the proof.
\end{proof}

\subsection{Paths and saturated chains}

As we will see in Proposition \ref{T:trans}, in order to find the maximum Newton point $\nu_x$, we will need to look for saturated chains in Bruhat order from $x$ which terminate at a pure translation element.  Repeated application of Proposition \ref{T:cocover} provides an interpretation of such chains in terms of paths in the quantum Bruhat graph, and vice versa.  Although we shall not use the full strength of the correspondence as stated in Proposition \ref{T:QBGpaths} in our proof of Theorem \ref{T:main}, we prove the most precise statement possible in case this proposition might be of independent combinatorial interest.

\begin{prop}\label{T:QBGpaths}
Let $w,v \in W$, and let $k$ be the minimum length of any path in the quantum Bruhat graph from $w^{-1}v$ to $v$.  Define $x = t^{v\lambda}w \in \widetilde{W}$, and suppose that 
\begin{equation}\label{E:PathPropRegHyp}
\langle \lambda, \alpha_i \rangle > 
\begin{cases}
2\ell(w_0)+2k & \text{if $G \neq G_2$},\\
3\ell(w_0)+3k & \text{if $G = G_2$},
\end{cases}
\end{equation}
 for all $\alpha_i \in \Delta$. Then,
\begin{enumerate} 
\item[($i$)] any minimal length saturated chain from $x$ to a translation $t^\mu$ has length $k$, and can be associated to a unique path in the quantum Bruhat graph from $w^{-1}v$ to $v$ of length $k$.  Moreover,  $\mu^+ = \lambda - \sum \beta^{\vee}$, where the sum records the weights of the downward edges in the associated path from $w^{-1}v$ to $v$.
\item[($ii$)] any path of length $k$ in the quantum Bruhat graph from $w^{-1}v$ to $v$ corresponds to $2^k$ saturated chains in Bruhat order of length $k$ from $x$ to a pure translation. Moreover, each such translation $t^\mu$ satisfies $\mu^+ = \lambda - \sum \beta^{\vee}$, where the sum records the weights of the  downward edges in the path from $w^{-1}v$ to $v$.
\end{enumerate}
\end{prop}

\begin{proof} We first prove ($ii$) and then proceed to ($i$).   

($ii$) The proof of this part proceeds by induction on $k$.  Note that there is nothing to prove in the $k=0$ case, which corresponds to the situation in which $x$ itself is already a translation, since $w^{-1}v=v$ if and only if $w=1$.

Now consider any path of length $k\geq 1$ from $w^{-1}v$ to $v$ in the quantum Bruhat graph, say 
\begin{equation}\label{E:path}
w^{-1}v \longrightarrow w^{-1}vr_{\beta_{1}} \longrightarrow \cdots \longrightarrow w^{-1}vr_{\beta_{1}}\cdots r_{\beta_{k-1}} \longrightarrow w^{-1}vr_{\beta_{1}}\cdots r_{\beta_{k}} = v.
\end{equation}
First apply cases (3) and (4) of Proposition \ref{T:cocover} and Remark \ref{T:cocoverQBG} to the initial edge of this path $w^{-1}v \longrightarrow w^{-1}vr_{\beta_1}$.  In either case we obtain a unique cocover $x_1 = t^{v\lambda '}r_{v\beta_1}w$ of $x$, where $\lambda' = \lambda$ in case (3) and $\lambda ' = \lambda - \beta_1^{\vee}$ in case (4), corresponding to whether the edge $w^{-1}v \longrightarrow w^{-1}vr_{\beta_1}$ is directed upward or downward in the graph, respectively.  Now define $w' = r_{v\beta_1}w$, which is the finite part of $x_1$, and compute that 
\begin{equation}\label{E:QBGe1}
(w')^{-1}v = (r_{v\beta_1}w)^{-1}v = w^{-1}(vr_{\beta_1}v^{-1})v = w^{-1}vr_{\beta_1}.
\end{equation}
Note that since $k$ is the minimum length of any path from $w^{-1}v$ to $v$, then the truncated path
\begin{equation}
(w')^{-1}v = w^{-1}vr_{\beta_{1}} \longrightarrow \cdots \longrightarrow  v
\end{equation} 
of length $k-1$ is a path of minimal length from $(w')^{-1}v$ to $v$ in the quantum Bruhat graph.  Finally, by our hypothesis on $\langle \lambda, \alpha_i \rangle$, for any $\alpha_i \in \Delta$,
\begin{equation}\label{E:lambdaindhyp}
\langle \lambda', \alpha_i \rangle \geq \langle \lambda - \alpha^{\vee}, \alpha_i \rangle  > 
\begin{cases}
(2\ell(w_0)+2k)-2 = 2\ell(w_0)+2(k-1) & \text{if $G \neq G_2$,}\\
(3\ell(w_0)+3k)-3 = 3\ell(w_0)+3(k-1) & \text{if $G = G_2$.}
\end{cases}
\end{equation}
Therefore, the induction hypothesis is satisfied on the pair $w', v \in W$ and the element $x_1 = t^{v\lambda '}w' \in \widetilde{W}$.  By induction, we obtain $2^{k-1}$ distinct saturated chains of length $k-1$ from $x_1$ to a pure translation.  To each of these chains, we can pre-append the covering relation $x \gtrdot x_1$ to obtain $2^{k-1}$ saturated chains from $x$ to a pure translation, all of which are length $k$.  

We now consider the path in \eqref{E:path} from a different perspective.  Looking at the final edge in this path which terminates at $v$, we see that $w^{-1}vr_{\beta_1}\cdots r_{\beta_{k-1}} = vr_{\beta_k}$.  It will thus be more convenient to write this final edge in \eqref{E:path} as $vr_{\beta_k} \longrightarrow v$.  Now apply cases (1) and (2) of Proposition \ref{T:cocover} and Remark \ref{T:cocoverQBG} to this final edge.  In either case, we obtain a unique cocover $y_1 = t^{vr_{\beta_k}(\mu)}r_{v\beta_k}w$ of $x$, where $\mu = \lambda$ in case (1) and $\mu = \lambda - \beta_k^{\vee}$ in case (2), corresponding to whether the edge $vr_{\beta_k} \longrightarrow v$ is directed upward or downward in the graph, respectively.  Now define $v'=vr_{\beta_k}$, which indexes the Weyl chamber in which the cocover $y_1$ lies, and denote by $w'' =  r_{v\beta_k}w$, which is the finite part of $y_1$.  Compute that
\begin{equation}\label{E:1vertexchanges}
(w'')^{-1}v' = (r_{v\beta_k}w)^{-1}vr_{\beta_k} = w^{-1}(vr_{\beta_k}v^{-1})vr_{\beta_k} = w^{-1}v.
\end{equation}
By the same calculation as in \eqref{E:lambdaindhyp}, we know that 
\begin{equation}\label{E:muindhyp}
\langle \mu, \alpha_i \rangle > 
\begin{cases}
2\ell(w_0)+2(k-1) & \text{if $G \neq G_2$,}\\
3\ell(w_0)+3(k-1) & \text{if $G = G_2$.}
\end{cases}
\end{equation}
Therefore, the induction hypothesis also applies to the pair $w'', v' \in W$, the element $y_1 \in \widetilde{W}$, and the truncated path
\begin{equation}
(w'')^{-1}v' = w^{-1}v \longrightarrow \cdots \longrightarrow  vr_{\beta_k} = v'
\end{equation}
of length $k-1$ in the quantum Bruhat graph.  We thus obtain $2^{k-1}$ distinct saturated chains of length $k-1$ from $y_1$ to a pure translation, which can be concatenated with the covering $x \gtrdot y_1$ to give $2^{k-1}$ saturated chains from $x$ to a pure translation of length $k$.  The form of the coroot in the translation at the end of these saturated chains follows by the induction hypothesis and the construction of the cocovers $x_1$ and $y_1$.

Finally, we argue that under the superregularity hypothesis on $\lambda$ we have $x_1 \neq y_1$.  Suppose for a contradiction that $x_1 = y_1$.  We then have that 
\begin{equation}
x_1=t^{v(\lambda - m\beta_1^{\vee})}r_{v\beta_1}w = t^{vr_{\beta_k}(\lambda - n\beta_k^{\vee})}r_{v\beta_k}w = y_1,
\end{equation}
where $m,n \in \{0,1\}$.  Comparing the finite parts, we first see that $\beta_1=\beta_k$, and so in the remainder of the argument we omit subscripts for convenience.  Setting the translation parts equal to each other then says that 
\begin{align*}
v(\lambda - m\beta^\vee) &= vr_\beta(\lambda - n\beta^\vee) \quad \quad\quad\quad\quad \iff \\
v\lambda - mv\beta^\vee & = v\lambda - \langle \lambda, \beta \rangle v \beta^\vee + nv\beta^\vee \quad \iff \\
\vec{0} & = (m+n-\langle \lambda, \beta \rangle)v\beta^\vee \quad\quad\hskip 1pt \iff \\
\langle \lambda, \beta \rangle & = m+n.  
\end{align*}
On the other hand, since $m,n \in \{0,1\}$, then $m+n \leq 2$. Comparing this inequality to our superregularity hypothesis \eqref{E:PathPropRegHyp}, we deduce that $G = A_1$ and $\beta = \alpha_1$ so that we can attain the minimum possible value of $\langle \lambda, \beta\rangle = 2$.  A direct calculation shows that there are only two possible elements $x \in \widetilde{S_2}$ which satisfy these criteria.  Namely, either $x = s_0 = t^{\alpha_1}s_1$ (for which the statement of the proposition fails, since in this case there is a single cocover of $x$ rather than two as the proposition predicts), or $x=t^{s_1\alpha_1}s_1$, and both of these elements are excluded by the superregularity hypothesis.  Therefore, in any case permitted by our hypotheses, the cocovers $x_1$ and $y_1$ are distinct.

Since $x_1 \neq y_1$, then each of the chains of the form $x \gtrdot x_1 \gtrdot \cdots \gtrdot t^\nu$ is distinct from each of the chains of the form $x \gtrdot y_1 \gtrdot \cdots \gtrdot t^\gamma$.  Altogether we thus have $2^{k-1} + 2^{k-1} = 2^k$ saturated chains of length $k$ from $x$ to a pure translation, and so the result in ($ii$) follows by induction.

($i$) First suppose that $k=0$.  In this case, $w^{-1}v=v$, which means that $w=1$ and so $x=t^{v\lambda}$ is already a translation.  We can thus associate to $x$ the path of length 0 in the quantum Bruhat graph from $v$ to itself.

Now suppose that $k \geq 1$, and consider a saturated chain of minimal length from $x=t^{v\lambda}w$ to a translation, say $x \gtrdot x_1 \gtrdot \cdots \gtrdot x_m = t^\mu$.  Recall by hypothesis that the minimum length of any path in the quantum Bruhat graph from $w^{-1}v$ to $v$ equals $k$.  Consider any such minimal length path:
\begin{equation}\label{E:pt1path}
w^{-1}v \longrightarrow w^{-1}vr_{\beta_{1}} \longrightarrow \cdots \longrightarrow w^{-1}vr_{\beta_{1}}\cdots r_{\beta_{k-1}} \longrightarrow w^{-1}vr_{\beta_{1}}\cdots r_{\beta_{k}} = v.
\end{equation}
By part ($ii$) proved independently above, this path corresponds to $2^k$ saturated chains of length $k$ from $x$ to a pure translation, so in particular, there exists a saturated chain  of length $k$ from $x$ to a translation.  Therefore, if $m$ is the minimum length from $x$ to a translation, we must have $m \leq k$.

Now using Proposition \ref{T:cocover}, we can write $x_1 = t^{v_1\lambda_1}w_1,$ where the finite part of $x_1$ equals $w_1=r_{\beta_1} w$ with $\beta_1 = r_{v\gamma_1}$ for some $\gamma_1 \in R^+$.  Recall that either $v_1=vr_{\gamma_1}$ in cases (1) and (2) of Proposition \ref{T:cocover}, or $v_1=v$ in cases (3) or (4).  Compute using $r_{v\gamma_1} = vr_{\gamma_1} v^{-1}$ that 
\begin{equation}\label{E:newwv}
(w_1)^{-1}v_1 = 
\begin{cases}
w^{-1}v, & \text{if}\ v_1=vr_{\gamma_1}\\
w^{-1}vr_{\gamma_1}, & \text{if}\ v_1=v.
\end{cases}
\end{equation} In particular, observe that either $w_1^{-1}v_1 = w^{-1}v$ or $v_1 = v$, while the other gets right multiplied by $r_{\gamma_1}$.  Note in addition that either $\lambda_1= \lambda$ or $\lambda_1 = \lambda - \gamma_1^{\vee} \in Q^+$ by Proposition \ref{T:cocover} according to whether the corresponding edge in the quantum Bruhat graph is upward or downward, so that 
\begin{equation}\label{E:pt1superreg}
\langle \lambda_1, \alpha_i \rangle \geq \langle \lambda - \gamma_1^{\vee}, \alpha_i \rangle  > 
\begin{cases}
 2\ell(w_0)+2(k-1) & \text{if $G \neq G_2$,}\\
 3\ell(w_0)+3(k-1) & \text{if $G = G_2$.}
\end{cases}
\end{equation}
If $k=1$ we stop here.  Otherwise, $k>1$ and \eqref{E:pt1superreg} says that we can apply Proposition \ref{T:cocover} now to $x_1$ to write the cocover $x_2 \lessdot x_1$ as $x_2 = t^{v_2\lambda_2}w_2$, where $w_2 = r_{\beta_2}w_1$ with $\beta_2 = r_{v_1\gamma_2}$ for some $\gamma_2 \in R^+$, and either $v_2 = v_1r_{\gamma_2}$ or $v_2=v_1$, depending on which case of Proposition \ref{T:cocover} applies.  

Continuing in this manner, for all $1 \leq j \leq m$, we have
\begin{equation}
\langle \lambda_j, \alpha_i \rangle \geq \langle \lambda_{j-1} - \gamma_j^{\vee}, \alpha_i \rangle  > 
\begin{cases}
 2\ell(w_0)+2(k-j) & \text{if $G \neq G_2$,}\\
 3\ell(w_0)+3(k-j) & \text{if $G = G_2$.}
\end{cases}
\end{equation}
The hypothesis \eqref{E:cocoverhyp} required to iteratively apply Proposition \ref{T:cocover} to the sequence $x, x_1, x_2, \dots, x_{m-1}$ thus permits precisely $k$ applications.  Since $m \leq k$, we can repeatedly use Proposition \ref{T:cocover} to express all $m$ covering relations $x \gtrdot x_1 \gtrdot \cdots \gtrdot x_m=t^\mu$  as $x_i = t^{v_i\lambda_i}w_i$, where each cocover is associated to a unique edge in the quantum Bruhat graph either of the form $w_{i}^{-1}v_{i} \longrightarrow w_{i}^{-1}v_{i}r_{\gamma_{i+1}} =w_{i+1}^{-1}v_{i+1}$ or $v_{i+1} = v_{i}r_{\gamma_{i+1}}\longrightarrow v_{i}$, depending on the type of the cocover as in \eqref{E:newwv}.  

Putting these $m$ edges in the quantum Bruhat graph together according to whether they are outward from $w_i^{-1}v_i$ or inward toward $v_i$, we then have two paths in the quantum Bruhat graph:
\begin{equation}
w^{-1}v \longrightarrow w^{-1}vr_{\gamma_{j_1}} \longrightarrow w^{-1}vr_{\gamma_{j_1}}r_{\gamma_{j_2}} \longrightarrow \cdots \longrightarrow w^{-1}vr_{\gamma_{j_1}}r_{\gamma_{j_2}} \cdots r_{\gamma_{j_p}} = w_m^{-1}v_m,
\end{equation}
\begin{equation}
v \longleftarrow vr_{\gamma_{\ell_1}}  \longleftarrow vr_{\gamma_{\ell_1}}r_{\gamma_{\ell_2}} \longleftarrow \cdots \longleftarrow vr_{\gamma_{\ell_1}}r_{\gamma_{\ell_2}} \cdots r_{\gamma_{\ell_q}} = v_m,
\end{equation}
where the indices $j_1<j_2 < \cdots < j_p$ and $\ell_1< \ell_2 < \cdots < \ell_q$, and the total count is $p+q=m$.  Now recall that $x_m =t^{v_m\lambda_m}w_m= t^\mu$ is a translation, which means that $w_m = 1$, and so in fact $w_m^{-1}v_m = v_m$.  We have thus actually constructed a path in the quantum Bruhat graph of length $m$, starting at vertex $w^{-1}v$ and ending at vertex $v$, which is uniquely determined by the sequence of covering relations $x \gtrdot x_1 \gtrdot \cdots \gtrdot x_m = t^\mu$.  Note by our iterated application of Proposition \ref{T:cocover} that $\mu^+ = \lambda - \sum\gamma^\vee$, where the coroot $\gamma^\vee$ is subtracted if and only if the corresponding edge in the path from $w^{-1}v$ to $v$ is directed downward.  Finally, if $m<k$, the existence of this path from $w^{-1}v$ to $v$ would contradict the minimality of $k$, and so in fact $m=k$.
\end{proof}


\section{Convexity, Dominance Order, and Root Hyperplanes}\label{S:RootHyps}

This section lays the necessary technical groundwork for the proof of our second key proposition in Section \ref{S:Reduction}. The idea is to construct upper and lower bounds on the maximum Newton point $\nu_x$, using the geometry of the affine hyperplane arrangement to compare Newton points in dominance order. For example, Lemma \ref{T:Projection} constructs an element which bounds the Newton point $\nu(x)$ for $x$ from above.  We also mention Lemma \ref{T:FiFormula} defining a linear functional which we then use to bound $\nu_x$ from below in Lemma \ref{T:FiTranslation}. In Section \ref{sec:DomBruhat}, we prove several lemmas relating the dominance and affine Bruhat orders, which may be of independent interest.

\subsection{Convexity and dominance order}\label{S:Convexity}

One key ingredient in the proof of the main theorem uses a geometric interpretation of the dominance order on the poset $(Q^\vee \otimes_{\Q} \R)^+$ in terms of convex subsets of Euclidean space.  We thus state the following lemma due to Atiyah and Bott, which they attribute to earlier work of Horn \cite{Horn} and Kostant \cite{Kostant}.   

\begin{lemma}[Lemma 12.14 \cite{AtiyahBott}]\label{T:AtiyahBott}
Let $x, y \in C$ be any points in the closed dominant Weyl chamber.  Then
\begin{equation}\label{E:ConvexOrbit}
y \leq x \iff y \in \operatorname{Conv}(W x),
\end{equation}
where $\operatorname{Conv}(W x)$ denotes the convex hull of the $W$-orbit of the point $x$.
\end{lemma}

\noindent We remark that this statement also holds when $x$ and $y$ are elements of the integral weight lattice; see Theorem 1.9 in \cite{StPO}.  In fact, it is possible to adapt Stembridge's proofs to the case of points in $(Q^\vee \otimes_{\Z} \Q)^+$ with bounded denominator, which is sufficient for our purposes given that all Newton points with rational slopes satisfy an integrality condition.  We nevertheless appeal to the result from \cite{AtiyahBott}, since it holds for any vectors in $\R^r$ in the dominant Weyl chamber.

\begin{subsection}{Variations on Mazur's inequality}

Given an element $x = t^{\lambda}w$ in the affine Weyl group, Proposition \ref{T:NPforW} and Lemma \ref{T:AtiyahBott} show that 
\begin{equation}\label{E:Mazur}
\nu(x) \leq \nu(t^\lambda) = \lambda^+.
\end{equation}
Alternatively, this observation can be thought of as an application of Mazur's inequality.  Generically, however, the Newton point of an affine Weyl group element whose finite part is non-trivial will lie \emph{strictly} below the Newton point of the corresponding pure translation element.  We state this slight strengthening of Mazur's inequality in the following simple lemma.

\begin{lemma}\label{T:StrongMazur}
If $x=t^{\lambda}w \in \widetilde{W}$ is such that $\lambda$ is regular and $w \neq 1$, then $\nu(x) < \nu(t^{\lambda})$.
\end{lemma}

\begin{proof}
Observe by Proposition \ref{T:NPforW} that $\nu(t^{\lambda}) = \lambda^+$.  Proposition \ref{T:NPforW} also says that in general the Newton point $\nu(x)$ for the element $x$ is given by averaging the orbit of $\lambda$ under the powers of $w$.  Therefore, $\nu(x) = \lambda^+$ if and only if $w(\lambda) = \lambda$; \textit{i.e.} this occurs when $w$ is in the stabilizer in $W$ of $\lambda$.  However, \cite[Ch.~V, \textsection 3.3 Prop.~2]{Bour46} says that $\text{Stab}_W(\lambda)$ is generated by the reflections in $W$ that fix $\lambda$.  Therefore, $\text{Stab}_W(\lambda)$ is trivial as long as $\lambda$ is regular and does not lie on any root hyperplane.
\end{proof}

Generally speaking, if $x = t^\lambda w$ and the order of $w \in W$ is large, the Newton point $\nu(x)$ is actually considerably smaller than $\lambda^+$.   However, when the finite part of $x$ is a reflection which is simple with respect to the Weyl chamber containing $x$, then the Newton point of $x$ is as close as possible to $\lambda^+$.  The next lemma shows that this extreme case in which $w=r_{\beta_i}$ interpolates between the prediction of Mazur's inequality and the actual Newton point $\nu(x)$.

\begin{lemma}\label{T:Projection}
Let $x=t^{v\lambda}w \in \widetilde{W}$ with $\lambda \in Q^+$ and $w\neq 1$.  Then there exists a reflection $r_{\beta_i} \in W,$ where $\beta_i = v\alpha_i$ for some $\alpha_i \in \Delta$, such that
$\lambda \geq \nu(t^{v\lambda}r_{\beta_i})\geq \nu(x).$
\end{lemma}

We remark that Lemma \ref{T:Projection} follows from two applications of \cite[Lemma 3.6]{Ch}, although the proof of Lemma 3.6 is omitted in \cite{Ch}, so we provide a direct proof for the sake of completeness.

\begin{proof} 
By Mazur's inequality \eqref{E:Mazur} and dominance of $\lambda$, we automatically have that $\lambda \geq \nu(x)$.  In addition, for any reflection $r_{\beta} \in W$, we also have that $\lambda \geq \nu(t^{v\lambda}r_\beta)$ again by Mazur's inequality \eqref{E:Mazur}. It therefore remains to choose a suitable reflection $r_{\beta_i} \in W$ so that $\nu(t^{v\lambda}r_{\beta_i}) \geq \nu(x)$. Clearly, if we already have $x=t^{v\lambda} r_{\beta_i}$ for $\beta_i = v\alpha_i$, then there is nothing more to do.  

More generally, since $w \neq 1$, then the order of $w$ equals $m \geq 2$. Recall from Theorem \ref{T:NPforW} that $\nu(x) = \left(\frac{1}{m}\sum\limits_{i=1}^m w^i(v\lambda)\right)^+$.   Denote by $o(x) = \sum\limits_{i=1}^m w^i(v\lambda)$ the sum of all of the elements in the orbit of $w$ on $v\lambda$, and note that $w\cdot o(x) = o(x)$.  By \cite[Ch.~V, \textsection 3.3 Rem.~3]{Bour46}, the element $o(x)$ lies on the wall of some Weyl chamber, and hence so does $\frac{1}{m}o(x)$.  More specifically, both $o(x)$ and $\frac{1}{m}o(x)$ are contained in any hyperplane $H_\alpha$ fixed by a reflection $r_\alpha$ occurring in a reduced decomposition for $w$ with respect to those reflections which are simple with respect to the Weyl chamber containing $o(x)$.  The Newton point $\nu(x)$ equals the unique dominant element in the $W$-orbit of $\frac{1}{m}o(x)$.  Therefore, $\nu(x)$ lies on a wall of the dominant Weyl chamber (perhaps even the intersection of several walls).  Choose any $1 \leq i \leq r$ such that $\nu(x) \in H_{\alpha_i}$, and define $y = t^{v\lambda}r_{\beta_i}$ where $\beta_i = v\alpha_i$.  By Theorem \ref{T:NPforW}, we have $\nu(y) = \frac{1}{2}\left( v\lambda + r_{\beta_i}(v\lambda)\right)^+  = v^{-1}\left( v\lambda - \frac{\langle v\lambda, \beta_i \rangle}{2}\beta_i^\vee\right) =  \lambda - \frac{\langle \lambda, \alpha_i \rangle}{2}\alpha_i^\vee$.

We claim that $\nu(y) \geq \nu(x)$. Recall that $\lambda \geq \nu(x)$ by Mazur's inequality, which means by definition that $\lambda - \nu(x) = \sum\limits_{i=1}^n r_i\alpha_i^{\vee}$ for some nonnegative rational numbers $r_i \in \Q_{\geq 0}$.  Of course, since $\lambda \in Q^+$ and $\lambda - \nu(y) = \frac{\langle \lambda, \alpha_i \rangle}{2}\alpha_i^\vee$, then it is also clear that $\lambda \geq \nu(y)$. Taking differences, we see that $\nu(y) - \nu(x) = \sum\limits_{j\neq i} r_j \alpha^\vee_j + \left(r_i - \frac{\langle \lambda, \alpha_i\rangle}{2}\right) \alpha_i^\vee$.  Note, however, that both $\nu(x)$ and $\nu(y)$ are in the hyperplane $H_{\alpha_i}$. Denote by $\operatorname{pr}_i:\R^r \longrightarrow H_{\alpha_i}$ the orthogonal projection onto this hyperplane.   We can then equivalently express the difference $\nu(y) - \nu(x) = \operatorname{pr}_i \left( \sum\limits_{j\neq i} r_j\alpha_j^{\vee}\right) = \sum\limits_{j\neq i} r_j\operatorname{pr}_i \left(\alpha_j^{\vee}\right) .$ We directly compute that $\operatorname{pr}_i(\alpha_j^\vee) = \alpha_j^\vee - \frac{\langle \alpha_j^\vee, \alpha_i \rangle}{2}\alpha_i^\vee$.  Since each column of the Cartan matrix contains a unique positive entry on the diagonal, whenever $i \neq j$ we know that $\langle \alpha_j^\vee, \alpha_i \rangle \leq 0$; see \cite[Ch.\ VI \textsection 1, no.\ 5]{Bour46}.  Therefore, each vector $\operatorname{pr}_i(\alpha_j^\vee)$ is a nonnegative rational sum of coroots.  Since each $r_j \in \Q_{\geq 0}$, we have thus shown that $\nu(y) \geq \nu(x)$ in dominance order.  
\end{proof}

\begin{notation}\label{nu_i}
Given $x = t^{v\lambda}w \in \widetilde{W}$ with $\lambda \in Q^+$ and $w\neq 1$, denote by $\nu_i(x)$ the Newton point of the element $t^{v\lambda} r_{\beta_i}$ whose existence is guaranteed by Lemma \ref{T:Projection}. The element $\nu_i(x)$ can be thought of geometrically as the maximal element on the hyperplane $H_{\alpha_i}$ such that $\lambda^+ \geq \nu_i(x) \geq \nu(x)$, although we shall not use this property in its full strength.
\end{notation}

For geometric arguments involving Lemmas \ref{T:AtiyahBott} or \ref{T:Projection}, it can sometimes be useful to replace a given alcove by another which lies in the dominant Weyl chamber, but which has the same Newton point as the original alcove.  The next lemma makes such trades possible. 

\begin{lemma}\label{T:Dominant} 
Let $x = t^{v \lambda} w \in \widetilde{W}$, where $\lambda \in Q^+$ is dominant and $v,w \in W$.  Then,
\begin{equation}\label{E:DomReduction}
\nu(x) = \nu(t^{\lambda}v^{-1}wv).
\end{equation}
\end{lemma}

\begin{proof}
Suppose that the order of $w \in W$ equals $m$, and consider the following expression:
\begin{equation}\label{E:DomReductionCalc}
\frac{1}{m}\sum\limits_{i=1}^m w^i(v\lambda) = \frac{1}{m}\sum\limits_{i=1}^m vv^{-1}w^i(v\lambda)  = v\left( \frac{1}{m}\sum\limits_{i=1}^m \left(v^{-1}wv \right)^i (\lambda) \right).
\end{equation}
Since we know from Proposition \ref{T:NPforW} that $\nu(x) = \left(\frac{1}{m}\sum\limits_{i=1}^m w^i(v\lambda)\right)^+$, then the expression on the right in \eqref{E:DomReductionCalc} has the same Newton point as $x$.  That is, since we take the unique dominant element in the $W$-orbit of this expression, applying $v$ has no effect on the calculation of the Newton point, proving that $\nu(x) = \nu(t^{\lambda}v^{-1}wv)$. 
\end{proof}

\end{subsection}

\begin{subsection}{Root hyperplanes and convexity}

In order to apply Lemma \ref{T:AtiyahBott} to compare pairs of Newton points, we need a criterion which easily determines whether or not a point in $\R^r$ lies in the convex hull of the $W$-orbit of another.  In light of Lemma \ref{T:Projection}, we are particularly interested in the $W$-orbit of points of the form $\nu_i(x)$ which lie on some wall $H_{\alpha_i}$ of the dominant Weyl chamber.  Lemma \ref{T:FiFormula} below defines a linear functional which determines the hyperplane containing a codimension one face of $\operatorname{Conv}(W \nu_i(x))$. 

\begin{defn}
Define a hyperplane $H_{\widehat{\alpha}_i}$ in $\R^r$ which has basis $\Delta \backslash \{\alpha_i\}$ for some $i \in \{ 1, \dots, r\}$. 
\end{defn}

\noindent The hyperplanes $H_{\widehat{\alpha}_i}$ also appear in \cite{HitzForum} where they are called \emph{dual hyperplanes}; the terminology comes from the fact that the hyperplane $H_{\widehat{\alpha}_i}$ is orthogonal to the fundamental weight $\omega_i$, which we formalize in the following lemma.  In order to determine whether a point in $(Q^\vee \otimes_{\Z}\Q)^+$ lies on one side or the other of the hyperplane $H_{\widehat{\alpha}_i}$, we also require a positivity statement.

\begin{lemma}\label{T:FiFormula}
The hyperplane $H_{\widehat{\alpha}_i}$ is determined by the linear functional $p_i : \R^r \longrightarrow \R$ given by
\begin{equation}
p_i(\vec{v}) :=  \left\langle \vec{v}, \omega_i \right\rangle = 0.
\end{equation}
Moreover, for any $\alpha \in R^+$ and any $1 \leq i \leq r$, we have $p_i(\alpha^{\vee}) \geq 0.$
\end{lemma}

\begin{proof}
Write $\vec{v} = c_1\alpha_1^\vee + \cdots + c_r \alpha_r^\vee$, and recall that $\langle \alpha_j^\vee, \omega_i \rangle = \delta_{ji}$.  Therefore, $p_i(\vec{v})=0$ if and only if $c_i = 0$, or equivalently, if and only if $\vec{v} \in H_{\widehat{\alpha}_i}$.  Similarly, if $\alpha \in R^+$ is a positive root, then $\alpha^\vee = d_1\alpha_1^\vee + \cdots + d_r \alpha_r^\vee$ where all $d_i \geq 0$, and so $p_i(\alpha^\vee) = d_i \geq 0$.
\end{proof}

\begin{remark}
The functional $p_i$ can be described geometrically as the orthogonal projection of $\R^r$ onto the subspace $\R\omega_i$ and thus coincides with the map $\operatorname{pr}_{(i)}$ in \cite[\textsection 6]{Ch}, which also motivates our choice of notation.
\end{remark}

\end{subsection}

\begin{subsection}{Compatibility of superregularity hypotheses}

Our next lemma requires a slight variation on our previous superregularity hypothesis on $\lambda$, and the new ingredient arises from decomposing the highest coroot in $R^\vee$ in terms of the simple coroots.  Denote the highest \emph{root} by $\theta$ (also often denoted in the literature by $\tilde{\alpha}$), which has the property that $\beta \leq \theta$ for any $\beta \in R$.  Since $R^\vee$ is also a root system, there is similarly a coroot which is the maximum element in $R^\vee$ with respect to dominance order.  We call this coroot the \emph{highest coroot}, and denote it by $\tilde{\mu}$.  

We provide a table of values for $\tilde{\mu}$ for the reader's convenience in Table \ref{table:HighestCoroot}, although this table can easily be obtained from the corresponding information on the highest root by duality; compare Table 4.1 in \cite{Hum}.  (Note that $\tilde{\mu} \neq \theta^\vee$ unless $G$ is simply laced; \textit{i.e.} $\tilde{\mu} = \theta^\vee$ if and only if $G$ is of type $A, D$ or $E$.)  We follow the conventions in \cite{Bour46} for labeling the simple (co)roots, which are consistent with the conventions in \cite{Hum}.  

We now extract the essential information from Table \ref{table:HighestCoroot} for the purpose of our next lemma.  

\begin{table}[t]
\begin{center}
    \begin{tabular}{| c | c |}
    \hline
    $G$ & expression for the highest coroot \\ \hline
    $A_n$ & $\alpha_1^\vee + \cdots + \alpha_n^\vee$ \\ \hline
    $B_n$ & $2\alpha_1^\vee + \cdots +2\alpha^\vee_{n-1}+\alpha_n^\vee$ \\ \hline
    $C_n$ & $\alpha_1^\vee + 2\alpha_2^\vee + \cdots +2\alpha^\vee_n$ \\ \hline
    $D_n$ & $\alpha_1^\vee + 2\alpha_2^\vee + \cdots +2\alpha^\vee_{n-2}+\alpha_{n-1}^\vee + \alpha_n^\vee$ \\ \hline
    $E_6$ & $\alpha_1^\vee + 2\alpha_2^\vee + 2\alpha_3^\vee + 3\alpha_4^\vee + 2\alpha_5^\vee + \alpha_6^\vee$ \\ \hline
    $E_7$ & $2\alpha_1^\vee + 2\alpha_2^\vee + 3\alpha_3^\vee + 4\alpha_4^\vee + 3\alpha_5^\vee + 2\alpha_6^\vee + \alpha_7^\vee$ \\ \hline
    $E_8$ & $2\alpha_1^\vee + 3\alpha_2^\vee + 4\alpha_3^\vee + 6\alpha_4^\vee + 5\alpha_5^\vee + 4\alpha_6^\vee + 3\alpha_7^\vee + 2\alpha_8^\vee$ \\ \hline
    $F_4$ & $2\alpha_1^\vee + 4\alpha_2^\vee + 3\alpha_3^\vee + 2\alpha_4^\vee$ \\ \hline
    $G_2$ & $2\alpha_1^\vee + 3\alpha_2^\vee$ \\ \hline
    \end{tabular}
\end{center}
\vskip 5pt
\caption{The highest coroot $\tilde{\mu}$ satisfies $\tilde{\mu} \geq \beta^\vee$ for all $\beta^\vee \in R^\vee$.}
\label{table:HighestCoroot}
\end{table}

\begin{defn}\label{D:cG}
Define a constant $1 \leq c_G \leq 6$ to be the maximum integer required to express the highest coroot in the basis of simple coroots.  More specifically, define
\begin{equation}\label{E:cG}
c_G = 
\begin{cases}
1 & \text{if $G = A_n$,}\\
2 & \text{if $G = B_n, C_n,$ or $D_n$,}\\
3 & \text{if $G = E_6$ or $G_2$,}\\
4 & \text{if $G = E_7$ or $F_4$,}\\
6 & \text{if $G = E_8$}.
\end{cases}
\end{equation}
\end{defn}

Before proceeding with the proof of our next lemma, we define a single constant which unifies the various required bounds on $\langle \lambda, \alpha_j \rangle$ from previous sections with the superregularity hypothesis necessary for Lemma \ref{T:FiTranslation} below.

\begin{defn}\label{D:Mk}
Given a pair $w,v\in W$ of finite Weyl group elements, denote by $k$ the minimum length of any path from $w^{-1}v$ to $v$ in the quantum Bruhat graph for $W$. Define
\begin{equation}\label{E:Mkformula}
M_k = 
\begin{cases}
\max \{2\ell(w_0)+2k , 2kc_G\}, & \text{if $G\neq G_2$};\\
\max \{3\ell(w_0)+3k , 2kc_G\}, & \text{if $G= G_2$}.
\end{cases}
\end{equation}
For the interested reader, a straightforward type-by-type calculation using the values of $c_G$ from Definition \ref{D:cG} as well as Property (2) from Proposition \ref{T:QBGcomb} which says that $k \leq \ell(w_0)$ leads to the following values for the constant $M_k$:
\begin{equation}\label{E:Mk}
M_k = 
\begin{cases}
2\ell(w_0)+2k & \text{if $G = A_n, B_n, C_n,$ or $D_n$,}\\
3\ell(w_0)+3k & \text{if $G = G_2$,}\\
2\ell(w_0)+2k & \text{if $G = E_6$ and $k \leq \frac{\ell(w_0)}{2}$,}\\
6k & \text{if $G = E_6$ and $k> \frac{\ell(w_0)}{2}$,}\\
2\ell(w_0)+2k & \text{if $G = E_7$ or $F_4$ and $k \leq \frac{\ell(w_0)}{3}$,}\\
8k & \text{if $G = E_7$ or $F_4$ and $k > \frac{\ell(w_0)}{3}$,}\\
2\ell(w_0)+2k & \text{if $G = E_8$ and $k \leq \frac{\ell(w_0)}{5}$,}\\
12k & \text{if $G = E_8$ and $k > \frac{\ell(w_0)}{5}$.}
\end{cases}
\end{equation}
\end{defn}

\end{subsection}

\subsection{Root hyperplanes for maximal translations}

We now present a lemma which represents both a culmination of the collection of technical results in this section thus far, as well as the crux of the main argument in the second of two key propositions.  In words, Lemma \ref{T:FiTranslation} guarantees that the Newton point for any translation immediately below $x$ in Bruhat order lies outside of the convex hull of the Newton point of $x$ itself, an observation that is critical in the proof of Proposition \ref{T:trans} in the next section.

\begin{lemma}\label{T:FiTranslation}
Let $x = t^{v\lambda}w \in \widetilde{W}$, and assume that $x \gtrdot x_1 \gtrdot  \cdots \gtrdot x_k = t^{\mu}$ is any minimal length saturated chain from $x$ to a pure translation, where $k\geq 1$.  Suppose that $\langle \lambda, \alpha_j \rangle > M_k$ for all simple roots  $\alpha_j \in \Delta$.  Then for any $i \in [r]$ such that $\nu_i(x) \geq \nu(x)$, we have 
\begin{equation}
p_i(\mu^+) > p_i(\nu_i(x)).
\end{equation}
\end{lemma}

\begin{proof} 
Since $k \geq 1$, then we know that the finite part $w \neq 1$.  In addition, the superregularity hypothesis on $\lambda$ implies that $\lambda$ is dominant.  Therefore, Lemma \ref{T:Projection} says there exists $1 \leq i \leq r$ such that $ \nu_i(x) \geq \nu(x)$.

Note that all of the hypotheses of Proposition \ref{T:QBGpaths} apply under our current superregularity assumption on $\lambda$.  Since $x$ and $t^\mu$ differ by a sequence of covering relations, Proposition \ref{T:QBGpaths} says that $\mu^+ = \lambda - \sum \beta^{\vee}$, where this sum is a nonnegative combination of at most $k$ positive coroots $\beta^\vee \in R^\vee$, corresponding to the downward edges in any minimal length path from $w^{-1}v$ to $v$ in the quantum Bruhat graph, which also has length $k$.

We now define an element $\mu' = \lambda - k\tilde{\mu} \in Q^\vee$, where $\tilde{\mu}$ is the highest coroot in $R^\vee$.  The element $\mu'$ is designed to be a lower bound for $\mu^+$ in the following sense:
\begin{equation}
\mu^+ - \mu ' = \left( \lambda - \sum \beta^\vee \right) - \left( \lambda - k\tilde{\mu}\right) = k\tilde{\mu} - \sum\beta^\vee.
\end{equation}
Since there are at most $k$ summands in the rightmost sum, and for each of them we know that $\tilde{\mu}\geq \beta^\vee$, then this difference is a nonnegative linear combination of coroots.  Therefore, $\mu^+ - \mu ' \geq 0,$ and so by Lemma \ref{T:FiFormula} and linearity we have $p_i(\mu^+) \geq p_i(\mu')$.

To prove the current lemma, it therefore suffices to verify that $ p_i(\mu') > p_i(\nu_i(x))$.   Compute that
\begin{equation}
p_i(\mu')  = p_i(\lambda - k \tilde{\mu}) = p_i(\lambda) - kp_i(\tilde{\mu}).
\end{equation}
Now recall that $\nu_i(x)$ is the projection of $\lambda$ onto the hyperplane $H_{\alpha_i}$, which means that $\nu_i(x) = \lambda - \frac{1}{2}\langle \lambda, \alpha_i \rangle \alpha_i^{\vee}$.  Therefore,
\begin{equation}
p_i(\nu_i(x))  = p_i(\lambda) - \frac{\langle \lambda, \alpha_i \rangle}{2}.
\end{equation}
Taking the difference of the two previous equations and using the positivity in Lemma \ref{T:FiFormula}, we see that $ p_i(\mu') > p_i(\nu_i(x))$ if and only if 
\begin{equation}\label{E:FiDiff}
kp_i(\tilde{\mu}) <  \frac{\langle \lambda, \alpha_i \rangle}{2}.
\end{equation}
Expanding $\tilde{\mu}$ in terms of the basis of simple coroots as in Table \ref{table:HighestCoroot}, we have $\tilde{\mu} = c_1\alpha_1^{\vee} + \cdots + c_r\alpha_r^{\vee}$ for some $c_i \in \Z_{\geq 0}$.  Therefore,
$p_i(\tilde{\mu}) =  c_i,$
since by construction $p_i(\alpha_j^{\vee}) = \delta_{ji}$.  For each root system we have $c_i \leq c_G$ by the definition of $c_G$.   Therefore, we have
\begin{equation}
p_i(\tilde{\mu}) \leq c_G.
\end{equation}  
By our superregularity hypothesis on $\lambda$, we also know that $\langle \lambda, \alpha_i \rangle > 2kc_G$.  Putting these observations together, we see that 
\begin{equation}
kp_i(\tilde{\mu}) \leq  kc_G  =  \frac{2kc_G}{2} < \frac{\langle \lambda, \alpha_i \rangle}{2},
\end{equation}
and so $p_i(\mu^+) \geq p_i(\mu') > p_i(\nu_i(x))$, as desired.
\end{proof}

\subsection{Relating dominance and affine Bruhat orders}\label{sec:DomBruhat}

In order to treat one very special case in the proof of Proposition \ref{T:trans}, we require several final lemmas relating the affine Bruhat order to the dominance order.  
The next lemma shows that if we fix an element $x \in \widetilde{W}$ sufficiently far from the walls of any Weyl chamber, then there exists a unique maximum element below $x$ which lies in a particular Weyl chamber and has a prescribed finite part.  The proof of Lemma \ref{lem:MaxFinPtw} proceeds by repeated application of Proposition \ref{T:cocover}, and is thus similar to the proof of Proposition \ref{T:QBGpaths}.

\begin{lemma}\label{lem:MaxFinPtw}
Let $x=t^{v\lambda}w \in \widetilde{W}$, and suppose that $\langle \lambda, \alpha_i\rangle > 6\ell(w_0)$ for all simple roots $\alpha_i \in \Delta$. Fix a pair of elements $v', w' \in W$. Then there exists $\lambda' \in Q^+$ such that 
\begin{enumerate}
\item the element defined by $x' = t^{v'\lambda'}w'$ satisfies $x' \leq x$, and 
\item if $y=t^{v'\mu}w' \leq x$ is any element which both lies in the same Weyl chamber and has the same finite part as $x'$, then $\mu \leq \lambda'$.
\end{enumerate}
\end{lemma}

\begin{proof}
First, recall by Proposition \ref{T:cocover} that every edge of the form $vr_\alpha \longrightarrow v$ in the quantum Bruhat graph corresponds to a covering relation of the form $x\gtrdot t^{vr_\alpha(\mu)}r_{v\alpha}w$, where $\mu = \lambda$ or $\mu=\lambda-\alpha^\vee$, depending on whether the edge $vr_\alpha \longrightarrow v$ is directed upward or downward, respectively.  Now consider any path of minimal length in the quantum Bruhat graph from the first given element $v'$ to $v$.  Note by Proposition \ref{T:QBGcomb} that this path has at most $\ell(w_0)$ edges, each of which contributes at most $\tilde{\mu}$ to the weight of the path, where recall that $\tilde{\mu}$ is the highest coroot.  By the superregularity hypothesis on $\lambda$ in the case where $G \neq G_2$, we see that 
\begin{equation}
\langle \lambda, \alpha_i \rangle >  4\ell(w_0) = (2\ell(w_0)+2)+(2\ell(w_0)-2) \geq (2\ell(w_0)+2) + (\ell(w_0)-1)\langle \tilde{\mu}, \alpha_i \rangle,
\end{equation}
where we have used the values of $\tilde{\mu}$ from Table \ref{table:HighestCoroot} to compute that $\langle \tilde{\mu}, \alpha_i \rangle \leq 2$ for all $\alpha_i \in \Delta$.  Rearranging, we obtain 
\begin{equation}
\langle \lambda - (\ell(w_0)-1)\tilde{\mu}, \alpha_i \rangle > 2\ell(w_0)+2,
\end{equation}
which means that the superregularity hypothesis required to apply Proposition \ref{T:cocover} to each edge in this path from $v'$ to $v$ is satisfied.  Therefore, by repeated application of Proposition \ref{T:cocover}, any path of minimal length from $v'$ to $v$ in the quantum Bruhat graph corresponds to a sequence of covering relations from $x$ to some element in the $v'$-chamber. (Using the same argument, if $G = G_2$, then the superregularity hypothesis $\langle \lambda-(\ell(w_0)-1)\tilde{\mu}, \alpha_i \rangle > 3\ell(w_0)+3$ is required to apply Proposition \ref{T:cocover} to every edge in the path from $v'$ to $v$.   Since $\langle \tilde{\mu}, \alpha_i \rangle \leq 1$ in type $G_2$, we instead require $\langle \lambda, \alpha_i \rangle > 4\ell(w_0)+2$.)

Next, recall by Proposition \ref{T:cocover} that every edge in the quantum Bruhat graph of the form $w^{-1}v \longrightarrow w^{-1}vr_\alpha$ corresponds to a covering relation of the form $x \gtrdot t^{v(\mu)}r_{v\alpha}w$, where again $\mu = \lambda$ or $\mu=\lambda -\alpha^\vee$ depending on the direction of this edge.  Compute as in \eqref{E:QBGe1} that $(r_{v\alpha}w)^{-1}v = w^{-1}vr_\alpha$.  Therefore, by the same argument as above, under the same superregularity hypothesis $\langle \lambda, \alpha_i \rangle >4\ell(w_0)$ used in the previous argument (or $\langle \lambda,\alpha_i \rangle > 4\ell(w_0)+2$ in type $G_2$), any path of minimal length in the quantum Bruhat graph from $w^{-1}v$ to the product $(w')^{-1}v$, determined by the second in the pair of given elements, corresponds to a sequence of covering relations from $x$ to an element in the same Weyl chamber, but having finite part $w'$.

Putting these observations together, by first applying the affine reflections corresponding to any path of minimal length in the quantum Bruhat graph from $v'$ to $v$, we obtain an element $y' = t^{v'(\mu)}w'' \leq x$ in the $v'$-chamber, where $\mu \geq \lambda - \ell(w_0)\tilde{\mu}$. Using the full strength of the superregularity hypothesis $\langle \lambda, \alpha_i \rangle > 6\ell(w_0)$ for $G \neq G_2$, and the fact that $\langle \tilde{\mu}, \alpha_i \rangle \leq 2$ easily verified uising Table \ref{table:HighestCoroot}, we then have
\begin{equation}
\langle \mu, \alpha_i \rangle \geq \langle \lambda-\ell(w_0)\tilde{\mu}, \alpha_i \rangle \geq \langle \lambda, \alpha_i \rangle - \ell(w_0)\langle \tilde{\mu}, \alpha_i \rangle > 6\ell(w_0) - 2\ell(w_0) = 4\ell(w_0).
\end{equation}
(In type $G_2$, since $\langle \lambda, \alpha_i \rangle > 6\ell(w_0) > 5\ell(w_0)+2$ and $\langle \tilde{\mu}, \alpha_i \rangle \leq 1$, we are also guaranteed that $\langle \mu, \alpha_i \rangle > 4\ell(w_0)+2,$ as required to repeatedly apply Proposition \ref{T:cocover}.)
Therefore, applying the second of the two previous arguments, now instead fixing any path of minimal length from $(w'')^{-1}v'$ to $(w')^{-1}v'$ in the quantum Bruhat graph, we obtain a saturated chain from $y'$ to an element of the form $y=t^{v'(\mu')}w'$ which still lies in the $v'$-chamber and now also has finite part $w'$.  In addition, by our construction of repeatedly applying covering relations as in Proposition \ref{T:cocover}, we have $y \leq y' \leq x$.  

What we have shown thus far is that given $x=t^{v\lambda}w$ such that $\langle \lambda, \alpha_i \rangle >6\ell(w_0)$ and any pair of elements $v',w' \in W$, there exists an element of the form $y = t^{v'\mu}w'$ such that $y \leq x$.  It thus remains to show the uniqueness of a maximum $\mu \in Q^+$ such that $y \leq x$ is of this form.  Consider any saturated chain $x=x_0 \gtrdot x_1 \gtrdot \cdots \gtrdot x_m = x'$, where $x' = t^{v'\lambda'}w'$ with  $\lambda' \in Q^+$, and assume without loss of generality that $m$ is minimal such that $x'$ has this desired form.  Since $m$ is minimal and thus $m \leq 2\ell(w_0)$, by the same arguments provided above, each covering in this saturated chain corresponds to an edge in the quantum Bruhat graph.  Moreover, each edge of type (1) or (2) in the statement of Proposition \ref{T:cocover} can be concatenated to form a path from $v'$ to $v$ in the quantum Bruhat graph.  Define $j \in [m]$ to be the smallest value such that $x_{j-1} \gtrdot x_j$ is of type (1) or (2), meaning that each of the covering relations $x=x_0 \gtrdot x_1 \gtrdot \cdots \gtrdot x_{j-1}$ is of type (3) or (4) in the statement of Proposition \ref{T:cocover}, so that $x=x_0, x_1, \dots, x_{j-1}$ all lie in the $v$-chamber.  Record the path in the quantum Bruhat graph corresponding to these first $j-1$ covering relations as
\begin{equation}\label{E:localseq}
w^{-1}v \longrightarrow w^{-1}vr_{\gamma_1}  \longrightarrow \cdots \longrightarrow w^{-1}vr_{\gamma_1}\cdots r_{\gamma_{j-2}}r_\beta,
\end{equation}
 where we write $\beta = \gamma_{j-1}$ to distinguish this edge.  By definition of $j$ and Proposition \ref{T:cocover}, the next covering relation $x_{j-1} \gtrdot x_j$ instead corresponds to an edge of the form $vr_\alpha \longrightarrow v$ in the quantum Bruhat graph. Denote the elements in this subsequence of covering relations by $x_i = t^{v_i(\lambda_i)}w_i$ for all $i \in [j]$.  We record in particular that by Proposition \ref{T:cocover}, we have 
  \begin{alignat}{7}
w_{j-1} &={}&& r_{v\beta}r_{v\gamma_2}\cdots r_{v\gamma_1}w &\quad\quad\quad
v_{j-1}&={}&&v \\
 w_j &={}&& r_{v\alpha}r_{v\beta}r_{v\gamma_2}\cdots r_{v\gamma_1}w & 
 v_{j}&={}&&vr_{\alpha}.
\end{alignat}

We now modify this saturated chain by keeping $x_i$ for all $0\leq i \leq j-2$ the same, and then replacing the elements $x_i$ for $i \in \{j-1, j\}$ with two elements $x_i' = t^{v'_i(\lambda_i')}w_i'$ defined as follows:
 \begin{alignat}{7}
w'_{j-1} &={}&& r_{v\alpha}r_{v\gamma_2}\cdots r_{v\gamma_1}w &\quad\quad\quad
v'_{j-1}&={}&&vr_{\alpha}\\
 w'_j &={}&& r_{vr_\alpha(\beta)}r_{v\alpha}r_{v\gamma_2}\cdots r_{v\gamma_1}w & 
 v'_{j}&={}&&vr_{\alpha}.
 \end{alignat}
The corresponding modified subsequence of covering relations $x=x_0  \gtrdot \cdots \gtrdot x_{j-2} \gtrdot x'_{j-1} \gtrdot x'_j$ differs from the original by interchanging the type of the two final covering relations. Note that $v_j = v'_j$, and we claim that $w_j = w'_j$ as well.  To see this second equality, we simply compute that 
\begin{equation}
r_{vr_\alpha(\beta)}r_{v\alpha}  = (vr_\alpha r_\beta r_\alpha v^{-1})(vr_\alpha v^{-1}) = vr_\alpha v^{-1} v r_\beta v^{-1} = r_{v\alpha}r_{v\beta}.
\end{equation}
By the same argument as in the proof of Proposition \ref{T:QBGpaths}, exactly one of the two elements $v_i$ or $w_i^{-1}v_i$ changes with each covering relation; see \eqref{E:1vertexchanges}.  Therefore, $x_{j-2} \gtrdot x'_{j-1}$ is now the covering relation of type (1) or (2) corresponding to the edge $vr_\alpha \longrightarrow v$, while $x'_{j-1} \gtrdot x'_j$ is the same covering relation of type (3) or (4) corresponding to the final edge in \eqref{E:localseq}.  Since both the (truncated) original and this new saturated chain correspond to the same pair of paths of minimum length in the quantum Bruhat graph, then $\lambda_i = \lambda'_i$ as well, again by iterated application of Proposition \ref{T:cocover}.  Therefore, in fact, $x_j = x'_j$, and we may thus continue the chain with the original remaining covering relations
\begin{equation}
x=x_0 \gtrdot x_1 \gtrdot \cdots x_{j-2} \gtrdot x'_{j-1} \gtrdot x_j \gtrdot \cdots \gtrdot x_m = x'.
\end{equation}

By iterating this argument, we can continue to move all covering relations of type (1) or (2) to the left in this sequence, which will then also cluster all covering relations of type (3) or (4) to the right.  In other words, we may assume without loss of generality that this saturated chain from $x$ to $x'$ starts by applying a sequence of all non-local covering relations, each of which changes the Weyl chamber corresponding to the edges in some minimal path from $v'$ to $v$ in the quantum Bruhat graph.  The remainder of the saturated chain then still corresponds to a minimal path from the original vertex $w^{-1}v$ to $(w')^{-1}v'$ in the quantum Bruhat graph; again see \eqref{E:1vertexchanges}.  This second portion of the saturated chain can thus be assumed to perform a sequence of local covering relations entirely within the $v'$-chamber.  However, the altered sequence still terminates at the same element $x'$, which both has finite part $w'$ and lies in the $v'$-chamber.  

Every saturated chain of minimal length from $x=t^{v\lambda}w$ to an element of the form $x' = t^{v'\lambda'}w'$ thus corresponds to a pair of paths of minimal length in the quantum Bruhat graph. Recall by Proposition \ref{T:QBGcomb} that every minimal path from $v'$ to $v$ has the same weight, as does every minimal path from $w^{-1}v$ to $(w')^{-1}v'$.  Since we may assume without loss of generality that we first perform the covering relations corresponding to the edges of the path from $v'$ to $v$, and then follow with the covering relations corresponding to the path from $w^{-1}v$ to $(w')^{-1}v'$, the difference between $\lambda$ and $\lambda'$ is given precisely by the sum of the weights of these two minimal paths.  In particular, by Proposition \ref{T:QBGcomb}, we record for future use that
\begin{equation}\label{eq:lam'bound}
 \lambda' \geq \lambda - 2\ell(w_0)\tilde{\mu},
\end{equation}
but the exact value of $\lambda'$ is unique by Proposition \ref{T:QBGcomb}.
\end{proof}

Our final lemma says that, given an element $x = t^{v\lambda}w \in \widetilde{W}$, if we translate $x$ by a negative simple root with respect to $v$, the resulting alcove remains below $x$ in Bruhat order. It seems likely that this phenomenon is well-known, but we provide a proof for the sake of completeness. The strategy closely mirrors parts of the proof of Proposition \ref{T:cocover}, since the key idea is to express $y$ via a pair of covering relations in the Bruhat order.

\begin{lemma}\label{lem:DomBruhat}
Let $x = t^{v\lambda}w \in \widetilde{W}$ for $\lambda \in Q^+$ dominant, and suppose that $x$ lies in the Weyl chamber corresponding to $v \in W$.  If $y = t^{v(\lambda - \alpha^\vee_i)}w$ also lies in the $v$-chamber, then $y \leq x.$
\end{lemma}

\begin{proof}
Note by definition that $y = t^{-v \alpha_i^\vee} x$.  Our goal will thus be to express the translation $t^{-v\alpha_i^\vee}$ as the product of two affine reflections, such that the length decreases at each step.  
The alcove $x$ lies between two consecutive hyperplanes $H_{v\alpha_i, m}$ and $H_{v\alpha_i, m+1}$, where there are two possible values for $m \in \{\langle v\lambda, v\alpha_i \rangle, \langle v\lambda, v\alpha_i \rangle-1\} = \{ \langle \lambda, \alpha_i \rangle, \langle \lambda, \alpha_i \rangle -1 \}$.
  For any $m \in \Z$, compute that 
\begin{equation}\label{eq:transcomp}
\left(t^{(m-1)v\alpha_i^\vee}r_{v\alpha_i}\right) \left(t^{mv\alpha_i^\vee}r_{v\alpha_i} \right)= t^{(m-1)v\alpha_i^\vee + m(-v\alpha_i^\vee)} = t^{-v\alpha_i^\vee}.
\end{equation}
Therefore, $y$ can certainly be obtained from $x$ by performing a sequence of two affine reflections.  

In order to choose a pair of reflections such that the length decreases at each step, we proceed similarly to the proof of Proposition \ref{T:cocover}.  Our current hypothesis that $x$ lies in the $v$-chamber, even if $\lambda$ is not regular, implies that the length formula \eqref{E:xlength} still applies, since the regularity assumption in Lemma \ref{T:xlength} is required to guarantee that the element $v$ coincides with the Weyl chamber containing $x$.   In particular, $\ell(x)=\langle \lambda, 2\rho \rangle-\ell(w^{-1}v) + \ell(v).$
Now define $r_{\beta_1} = t^{mv\alpha_i^\vee}r_{v\alpha_i}$, and compute as in \eqref{E:cocoverform} that 
\begin{equation}
r_{\beta_1}x = t^{v(\lambda-(\langle \lambda, \alpha_i\rangle-m)\alpha_i^\vee)}r_{v\alpha_i}w,
\end{equation}
where the finite part $r_{v\alpha_i}w = vs_iv^{-1}w$, so that its inverse equals $(r_{v\alpha_i}w)^{-1} = (w^{-1}vs_i)v^{-1}$.
Since $s_i$ is a simple reflection, we know that $\ell(w^{-1}vs_i) = \ell(w^{-1}v) \pm 1$.  As in the proof of Proposition \ref{T:cocover}, since $| \langle \lambda, \alpha_i \rangle -m| \leq \ell(w_0)+1$ for $m \in \{ \langle \lambda, \alpha_i \rangle, \langle \lambda, \alpha_i \rangle -1 \}$, then $\lambda - (\langle \lambda, \alpha_i \rangle -m)\alpha_i^\vee$ is dominant.  Further, since the translation $y$ of $x$ by $-v\alpha_i^\vee$ remains in the $v$-chamber by hypothesis, then so does $r_{\beta_1}x$ for either of these two values of $m$.  Therefore, we may again apply \eqref{E:xlength} to compute
\begin{equation}
\ell(r_{\beta_1}x) = \langle \lambda - (\langle \lambda, \alpha_i \rangle -m)\alpha_i^\vee, 2\rho \rangle - \ell(w^{-1}vs_i)+\ell(v).
\end{equation}
If $\ell(w^{-1}vs_i) = \ell(w^{-1}v)+1$, using the value $m=\langle \lambda, \alpha_i \rangle$ gives $\ell(r_{\beta_1}x) = \ell(x)-1$.  Similarly, if $\ell(w^{-1}vs_i) = \ell(w^{-1}v)-1$, using the value $m=\langle \lambda, \alpha_i \rangle-1$ implies that $\ell(r_{\beta_1}x) = \ell(x)-1$.  In either scenario, we can choose  an affine reflection $r_{\beta_1}$ such that $r_{\beta_1}x \lessdot x$, as in cases (3) and (4) of Proposition \ref{T:cocover}.

Finally, define $r_{\beta_2} = t^{(m-1)v\alpha_i^\vee}r_{v\alpha_i}$ so that $r_{\beta_2}r_{\beta_1}x = t^{-v\alpha_i^\vee}x = y$, as shown in  \eqref{eq:transcomp}.  Since $y$ remains in the $v$-chamber by hypothesis, we again use \eqref{E:xlength} to compute
\begin{equation}
\ell(y) = \langle \lambda - \alpha_i^\vee, 2\rho \rangle - \ell(w^{-1}v)+\ell(v) = \ell(x) - \langle \alpha_i^\vee, 2\rho \rangle = \ell(x)-2.
\end{equation}
Therefore, $\ell(x) > \ell(r_{\beta_1}x) > \ell(r_{\beta_2}r_{\beta_1}x) = \ell(y)$, and $y \leq x$ by the definition of Bruhat order.
\end{proof}

By iterating Lemma \ref{lem:DomBruhat}, we obtain the following corollary relating the two natural partial orderings on all alcoves which lie in the same Weyl chamber and have the same finite part.

\begin{cor}\label{cor:DomBruhat}
Let $x=t^{v\lambda}w, \ y=t^{v\lambda'}w \in \widetilde{W}$, and suppose that both $x$ and $y$ lie in the $v$-chamber.  Further assume that $\langle \lambda, \alpha_i \rangle \geq 3$ for all $\alpha_i \in \Delta$, and also that $\rho^\vee \not> \lambda'$. If $\lambda \geq \lambda'$ in dominance order, then $x \geq y$ in Bruhat order.
\end{cor}

\begin{proof}
First note that our hypothesis that both $x$ and $y$ lie in the $v$-chamber implies that $\lambda, \lambda' \in Q^+$.  Therefore, the vector from $\lambda$ to $\lambda'$ is entirely contained in the dominant Weyl chamber. Now since $\lambda \geq \lambda'$, we may write  $\lambda - \lambda' = \sum c_i\alpha_i^\vee$ where all $c_i \in \Z_{\geq 0}$. Since $\langle \lambda, \alpha_i \rangle \geq 2$ for all $\alpha_i \in \Delta$, for any $j \in [r]$ such that $c_j \neq 0$, the vector from $\lambda-\alpha_j^\vee$ to $\lambda'$ remains in the dominant chamber. Moreover, note that for all $i \neq j$, we have $\langle \lambda - \alpha_j^\vee,\alpha_i \rangle  = \langle \lambda, \alpha_i \rangle - \langle \alpha_j^\vee, \alpha_i \rangle \geq \langle \lambda, \alpha_i \rangle$, since $\langle \alpha_j^\vee, \alpha_i \rangle$ is an off-diagonal entry of the Cartan matrix, and is hence nonpositive. The additional constant in the hypothesis $\langle \lambda, \alpha_i \rangle \geq 3$ further ensures that the corresponding alcove $x' = t^{v(\lambda-\alpha_j^\vee)}w$ remains in the $v$-chamber.  We may thus apply Lemma \ref{lem:DomBruhat} iteratively, choosing the order of application such that each of the intermediate alcoves $z= t^{v(\mu)}w$, where $\mu = \lambda - \sum b_i \alpha_i^\vee$ for some nonnegative integers $b_i \leq c_i,$ remains in the $v$-chamber. The remaining hypothesis $\rho^\vee \not> \lambda'$ guarantees that such a choice is possible, even if the end vertex $\lambda'$ lies on a wall of the dominant Weyl chamber.
\end{proof}


\section{Translations Dominate Newton Points}\label{S:TransPf}

We are now prepared to present our second of two key propositions, followed by the proof of our main theorem.  Recall Theorem \ref{T:Vform} which says that to compute $\nu_x$ we should take a maximum in dominance order over all $\nu(y)$ such that $x \geq y$.  The simple hint suggested by Lemma \ref{T:StrongMazur} and further supported by Lemma \ref{T:FiTranslation}, is that we can in fact reduce the problem to focusing exclusively on the pure translations below $x$.  Proposition \ref{T:trans} makes this reduction step precise, and the proof of Theorem \ref{T:main} follows immediately in Section \ref{S:MainPf}.  We conclude in Section \ref{S:HypRmks} with a discussion of the superregularity hypotheses required at various steps in the proof.

\begin{subsection}{Reduction to pure translations}\label{S:Reduction}

Our second key proposition says that in order to compute the generic Newton point in $IxI$, it suffices to find any maximal pure translation element which is less than or equal to $x$ in Bruhat order.  The proof uses critically both the correspondence to paths in the quantum Bruhat graph established in Proposition \ref{T:QBGpaths}, in addition to the entire sequence of lemmas from Section \ref{S:RootHyps}.  

One very special case in the proof of Proposition \ref{T:trans} requires a slightly stronger superregularlity hypothesis than has arisen in previous sections.  For this purpose, we define one final constant
\begin{equation}\label{E:MFormula}
M = M_k + 4\ell(w_0),
\end{equation}
where $M_k$ is as in Definition \ref{D:Mk}.

\begin{prop}\label{T:trans} 
Let $x=t^{v\lambda}w \in \widetilde{W}$, and consider any minimal length saturated chain in Bruhat order from $x$ to a pure translation, say $x \gtrdot x_1  \gtrdot \cdots \gtrdot x_k = t^{\mu}.$  If $ \langle \lambda, \alpha_j  \rangle > M$ for all simple roots $\alpha_j \in \Delta$,  then $\nu_x = \mu^+$. 
\end{prop}

\begin{proof}
Note by our superregularity hypothesis on $\lambda \in Q^+$ that $x$ is contained in the  Weyl chamber corresponding to $v \in W$. First consider the case when $k=0$.  Then $x= t^{v\lambda}=t^\mu $ is already a pure translation, and so Proposition \ref{T:NPforW} says that $\nu(x) = \lambda$. We also know by Mazur's inequality that $\nu_x \leq \lambda$.   Theorem \ref{T:Vform} says that $\nu_x = \max \{ \nu(y) \mid y \leq x \}$, and so clearly $\lambda \geq \nu_x \geq \nu(x) = \lambda$, which means that $\nu_x = \lambda = \mu^+$ in this case. 

Now consider any $x = t^{v\lambda} w \in \widetilde{W}$ such that the order of $w$ is nontrivial; \textit{i.e.} $x$ is not a translation. Let $x \gtrdot x_1 \gtrdot \cdots \gtrdot x_k = t^{\mu}$ be any minimal length saturated chain from $x$ to a pure translation.  First note that $k\geq 1$ since $x$ itself is not a translation.  In addition, the hypotheses of Proposition \ref{T:QBGpaths} are met by our superregularity condition on $\lambda$. Therefore, $k$ is also equal to the minimal length of any path in the quantum Bruhat graph for $W$ from $w^{-1}v$ to $v$.

We claim that $\nu(x) \not\geq \mu^+$.  First apply Lemma \ref{T:Projection} to obtain an index $1 \leq i \leq r$ such that $\nu_i(x) = \nu(t^{v\lambda} r_{\beta_i}) \geq \nu(x)$, where we are using the definition of $\nu_i(x)$ from Notation \ref{nu_i}. Of course, if it were the case that $\nu(x) \geq \mu^+$, then $\nu_i(x) \geq \nu(x) \geq \mu^+$ as well. We thus aim to prove instead that $\nu_i(x) \not\geq \mu^+$, from which it follows that $\nu(x) \not\geq \mu^+$.

We proceed to relate the elements $\nu_i(x)$ and $\mu^+$.  Recall Lemma \ref{T:AtiyahBott}, which reinterprets dominance order in terms of convexity.  The point $\nu_i(x)$ sits on the hyperplane $H_{\alpha_i}$, which is a wall of the dominant Weyl chamber.  Since both $\nu_i(x)$ and $\mu^+$ are points in the closed dominant Weyl chamber, then Lemma \ref{T:AtiyahBott} says that $\nu_i(x) \geq \mu^+$ if and only if $\mu^+ \in \operatorname{Conv}(W \nu_i(x))$.  On the other hand, one face of $\operatorname{Conv}(W\nu_i(x))$ which is contained in the dominant Weyl chamber is spanned precisely by the simple roots in $\Delta_P = \Delta \backslash \{ \alpha_i \}$, each shifted by the vector $\nu_i(x)$.  That is, this face of $\operatorname{Conv}(W\nu_i(x))$ is contained in an affine translate of the hyperplane $H_{\widehat{\alpha}_i}$ defined in Lemma \ref{T:FiFormula}, shifted so that $\nu_i(x)$ is the origin.  The hyperplane $H_{\widehat{\alpha}_i} + \nu_i(x)$ is therefore determined precisely by the linear functional $p_i: \R^r \rightarrow \R$ of Lemma \ref{T:FiFormula}, which means that $\vec{v} \in H_{\widehat{\alpha}_i} + \nu_i(x)$ if and only if $p_i(\vec{v}) = p_i(\nu_i(x))$.  In particular, since $\operatorname{Conv}(W\nu_i(x))$ contains the origin, if $p_i(\vec{v}) > p_i(\nu_i(x))$, then $\vec{v} \notin \operatorname{Conv}(W\nu_i(x))$.  Since by hypothesis $\langle \lambda, \alpha_j \rangle > M_k$ for all $\alpha_j \in \Delta$, we can apply Lemma \ref{T:FiTranslation}, which says that $p_i(\mu^+) > p_i(\nu_i(x)),$ and therefore $\mu^+ \notin \operatorname{Conv}(W\nu_i(x))$.  By Lemma \ref{T:AtiyahBott}, we thus have that $\nu_i(x) \not\geq \mu^+$, and hence we can also conclude that $\nu(x) \not\geq \mu^+.$  Moreover, since $\nu_x = \max \{ \nu(y) \mid y \leq x \}$ and $t^\mu \leq x$, we now know that $\nu_x \neq \nu(x)$.

More generally, now suppose that $y=t^{v'\lambda'}w' \leq x$, where $y$ is a non-translation alcove which lies in the $v'$-chamber.  There are two cases to consider, depending on whether or not $\lambda' \in Q^+$ satisfies the superregularity hypothesis $\langle \lambda', \alpha_j \rangle > M_k$ for all $\alpha_j \in \Delta$ required to apply Lemma \ref{T:FiTranslation}.   If $\lambda'$ is indeed superregular in this sense, then consider any minimal length saturated chain of the form $x \gtrdot \cdots \gtrdot y \gtrdot \cdots \gtrdot t^\gamma$. We know by the argument in the previous two paragraphs applied instead to $y$ that $\nu(y) \not\geq \gamma^+ = \nu(t^\gamma)$. But since $t^\gamma \leq x$ and $\nu_x  = \max \{ \nu(y) \mid y \leq x \}$, we conclude that $\nu_x \neq \nu(y)$ for any of these $y$.

Now suppose that $\langle \lambda', \alpha_j \rangle \leq M_k$ for at least one $\alpha_j \in \Delta$, and define the set of indices 
\begin{equation}\label{eq:Isr}
I_{sr} = \{ j \in [r] \mid \langle \lambda', \alpha_j \rangle  \leq M_k\} \neq \emptyset.\end{equation} Similarly, denote by $\Delta_{sr}$ the subset of simple roots $\Delta_{sr} = \{ \alpha_j \in \Delta \mid j \in I_{sr}\}$ with respect to which $\lambda'$ fails to meet this superregularity hypothesis.
We first discuss the case in which $I_{sr}= [r]$.  Since $\langle \rho^\vee, \alpha_j \rangle =1$ for all $\alpha_j \in \Delta$, we thus have $\lambda' \leq M_k\rho^\vee$ in this case. By Mazur's inequality \eqref{E:Mazur}, we know that $\nu(y) \leq \lambda' \leq M_k\rho^\vee$.  Recall that $x \gtrdot x_1 \gtrdot \cdots \gtrdot x_k = t^{\mu}$ is a saturated chain of length $k$, where $t^{\mu} = t^{u(\mu^+)}$ for  some $u \in W$.  By Proposition \ref{T:QBGpaths}, we know that $\mu^+ \geq \lambda - k\tilde{\mu}$.  Recalling that $k \leq \ell(w_0)$ by Lemma \ref{T:QBGcomb}, we compute that  for any $\alpha_j \in \Delta$, we have
\begin{equation}
\langle \mu^+, \alpha_j \rangle \geq \langle \lambda, \alpha_j \rangle - k\langle \tilde{\mu}, \alpha_j \rangle > M_k+4\ell(w_0) - 2k \geq M_k + 2\ell(w_0)>M_k,
\end{equation}
where we have calculated using the values of $\tilde{\mu}$ from Table \ref{table:HighestCoroot} that $\langle \tilde{\mu}, \alpha_j \rangle \leq 2$. In particular, we see that $\langle \mu^+, \alpha_j \rangle > M_k \geq 3$ for all $j \in [r]$, and more generally that $\mu^+>M_k\rho^\vee$.  Since $M_k > 1$ for any value of $k$, we also know that $\rho^\vee \not> M_k \rho^\vee$.  Therefore, Corollary \ref{cor:DomBruhat} says that $t^{u(\mu^+)} > t^{u(M_k\rho^\vee)}$, in which case $ t^{u(M_k\rho^\vee)} < t^\mu \leq x$.  Since $M_k\rho^\vee \in Q^+$, then $\nu(t^{u(M_k\rho^\vee)})=M_k\rho^\vee$ by Proposition \ref{T:NPforW}, and so Theorem \ref{T:Vform} implies that $\nu_x = \max \{ \nu(y) \mid y \leq x \} > M_k\rho^\vee$.  Since $\nu(y) \leq M_k \rho^\vee < \nu_x$ in this case, we conclude that $\nu_x \neq \nu(y)$ whenever $I_{sr} = [r]$.

From now on, we thus assume that $y=t^{v'\lambda'}w' \leq x$, where $\emptyset \subsetneq I_{sr} \subsetneq [r]$. Recall that $x=t^{v\lambda}w$, where $\langle \lambda, \alpha_j \rangle > M = M_k + 4\ell(w_0)$.  Since $M_k \geq 2\ell(w_0)$ by \eqref{E:Mkformula}, then $\langle \lambda, \alpha_j \rangle > 6\ell(w_0)$ for all $\alpha_j \in \Delta$.
 Applying Lemma \ref{lem:MaxFinPtw} to $x$ with the pair of elements $v', w' \in W$ associated to $y$, there exists a unique maximum, say $\lambda'_{x} \in Q^+$, such that $y_{x} := t^{v'\lambda'_{x}}w' \leq x$. In particular, the element $y_{x}$ lies in the same Weyl chamber and has the same finite part as $y$, in addition to the fact that $\lambda'_{x} \geq \lambda'$.
We now claim that the element $\lambda'_{x}$ is superregular in the sense of Lemma \ref{T:FiTranslation}.  To see this, recall from \eqref{eq:lam'bound} that $\lambda'_{x} \geq \lambda - 2\ell(w_0)\tilde{\mu}$, in which case for all $\alpha_j \in \Delta$, we have
\begin{equation}\label{eq:lamxsr}
\langle \lambda'_{x}, \alpha_j \rangle \geq \langle \lambda - 2\ell(w_0)\tilde{\mu}, \alpha_j \rangle > M - 2\ell(w_0)\langle \tilde{\mu}, \alpha_j \rangle \geq (M_k + 4\ell(w_0)) - (2\ell(w_0)\cdot 2) = M_k.
\end{equation}
Therefore,  we indeed have $\langle \lambda'_{x}, \alpha_j \rangle > M_k$ for all $\alpha_j \in \Delta$.

We proceed to relate $\lambda'_x$ and $\lambda'$ more explicitly.  Since $\lambda'_x \geq \lambda'$, we know by the definition that 
\begin{equation}\label{eq:lam'lamx}
\lambda' = \lambda'_x - \sum\limits_{i =1}^r m_i \alpha_i^\vee
\end{equation}
for some $m_i \in \Z_{\geq 0}$.  Now fix any $j \in I_{sr}$, and recall by \eqref{eq:Isr} that $\langle \lambda', \alpha_j \rangle \leq M_k$.  Using \eqref{eq:lam'lamx}, we then see that 
\begin{equation}\label{eq:lam'sr}
\left\langle \lambda', \alpha_j \right\rangle = \left\langle \lambda'_x - \sum\limits_{i=1}^r m_i \alpha_i^\vee, \alpha_j \right\rangle = \langle \lambda'_x, \alpha_j \rangle - \sum\limits_{i=1}^r m_i\left\langle \alpha_i^\vee, \alpha_j \right\rangle.
\end{equation}
Suppose for a contradiction that $m_j = 0$.  Since for any $i \neq j$, the value $\langle \alpha_i^\vee, \alpha_j \rangle$ is an off-diagonal entry of the Cartan matrix, then $\langle \alpha_i^\vee, \alpha_j \rangle \leq 0$ for all $i\neq j$.  Since $m_i \geq 0$, the entire sum on the right side of \eqref{eq:lam'sr} would then be nonpositive.  Therefore, $\langle \lambda', \alpha_j \rangle = \langle \lambda'_x, \alpha_j \rangle + N$ for some $N \geq 0$.  By the superregularity of $\lambda'_x$ proved in \eqref{eq:lamxsr}, we then have $\langle \lambda', \alpha_j \rangle > M_k + N \geq M_k$, contradicting the fact that $j \in I_{sr}$.  Therefore, for all $j \in I_{sr}$, we have $m_j \in \Z_{>0}$ in expression \eqref{eq:lam'lamx} for $\lambda'$.

We continue by translating the formula for the Newton point $\nu(y)$ via Proposition \ref{T:NPforW} into more convenient geometric language. Recall from the proof of Lemma \ref{T:Projection} that given $y = t^{v'\lambda'}w'$, we define $o(y) = \sum\limits_{i=1}^{|w'|} (w')^i(v'\lambda')$.  Denote by $R_{w'} \subseteq R^+$ the collection of positive roots such that $o(y) \in \bigcap\limits_{\beta \in R_{w'}} H_{\beta}$.   For brevity, denote this subspace by $H_{w'} = \bigcap\limits_{\beta \in R_{w'}} H_{\beta}$, and then denote by $\operatorname{pr}_{w'}: \R^r \longrightarrow H_{w'}$ the projection onto $H_{w'}$.  Note that $\operatorname{pr}_{w'}(v'\lambda') = \frac{1}{|w'|}o(y)$, and therefore $\nu(y) = \left(\frac{1}{|w'|}o(y)\right)^+ = \left( \operatorname{pr}_{w'}(v'\lambda') \right)^+$ by Proposition \ref{T:NPforW}.  Choose any $u \in W$ such that 
\begin{equation}\label{eq:uchoice}
\nu(y) = \operatorname{pr}_{w'}(v'\lambda')^+ = u(\operatorname{pr}_{w'}(v'\lambda')) \in C.\end{equation}
 This element $u \in W$ will remain fixed throughout the rest of the proof. 

There are now two further cases to consider.  First suppose there exists an index $j \in I_{sr}$ such that $u\left(\operatorname{pr}_{w'}(v'\alpha_j^\vee) \right) \not\leq 0.$  Interpreting antidominance in terms of Lemma \ref{T:FiFormula}, there exists an $i \in [r]$ such that $p_i \left( u\left(\operatorname{pr}_{w'}(v'\alpha_j^\vee) \right)\right) >0$. 
 Recall that since $j \in I_{sr}$, the coefficient $m_j \geq 1$ in expression \eqref{eq:lam'lamx}, in which case $\lambda'_x \geq \lambda'+\alpha_j^\vee$.   
Define an auxiliary element $y' = t^{v'(\lambda'+\alpha_j^\vee)}w'$, which lies in the $v'$-chamber together with $y_x$. We have $\rho^\vee \not> \lambda' + \alpha_j^\vee$, since otherwise $M_k \rho^\vee > \lambda' +\alpha_j^\vee > \lambda'$, contradicting the fact that $\lambda' \not\leq M_k\rho^\vee$ since $\emptyset \subsetneq I_{sr} \subsetneq [r]$. In addition, we have already seen that $\langle \lambda'_x, \alpha_i \rangle> M_k \geq 3$ for all $\alpha_i \in \Delta$.  Therefore, Corollary \ref{cor:DomBruhat} applies, in which case we conclude that $y' \leq y_x$. Recalling that $y_x \leq x$ by Lemma \ref{lem:MaxFinPtw}, then $y' \leq y_x \leq x$ as well. 
 Now compute 
\begin{equation}
\nu(y') = u\left(\operatorname{pr}_{w'}(v'(\lambda' +\alpha_j^\vee)) \right) = u\left(\operatorname{pr}_{w'}(v'\lambda')\right) + u\left( \operatorname{pr}_{w'}(v'\alpha_j^\vee) \right) = \nu(y) + u\left( \operatorname{pr}_{w'}(v'\alpha_j^\vee) \right).
\end{equation}
Applying the linear functional $p_i$ from Lemma \ref{T:FiFormula} then gives
\begin{equation}
p_i(\nu(y'))  = p_i(\nu(y)) + p_i \left( u\left( \operatorname{pr}_{w'}(v'\alpha_j^\vee) \right) \right).
\end{equation}
Since $p_i \left( u\left( \operatorname{pr}_{w'}(v'\alpha_j^\vee) \right) \right) > 0$ by hypothesis, we thus see that $p_i(\nu(y')) > p_i(\nu(y))$, in which case $\nu(y) \not\geq \nu(y')$ by Lemmas \ref{T:AtiyahBott} and \ref{T:FiFormula}.  Since $y' \leq x$ and $\nu_x  = \max \{ \nu(y) \mid y \leq x \}$, we conclude that $\nu_x \neq \nu(y)$ in this case.

It thus remains to consider the case in which $u\left(\operatorname{pr}_{w'}(v'\alpha_j^\vee) \right) \leq 0$ for all $j \in I_{sr}$.
We proceed to construct two local translates of the alcove $y=t^{v'\lambda'}w'$ whose Newton points will be suitable for comparing to $\nu(y)$ in this final case.  For this purpose, we first refine our expression for $\lambda'$ from above as follows:
\begin{equation}\label{eq:lam'srsum}
\lambda' = \sum\limits_{i \notin I_{sr}} a_i\alpha_i^\vee + \sum\limits_{j \in I_{sr}}a_j\alpha_j^\vee.
\end{equation}
Denote by $H_{sh} = \bigcap\limits_{j \in I_{sr}} H_{\alpha_j}$ the intersection of all hyperplanes with respect to which $\lambda'$ fails to meet the superregularity hypothesis, and denote by $\operatorname{pr}_{sh}: \R^r \longrightarrow H_{sh}$ the projection onto this subspace.  Now define $\lambda'_{sh} = \operatorname{pr}_{sh}(\lambda') \in H_{sh}$, and note that $\lambda'_{sh} \in (Q^\vee \otimes_{\Z} \Q)^+$ since the projection to $H_{sh}$ is obtained by subtracting a $\Q$-linear combination of simple coroots in $\Delta^\vee_{sr}$. Write
\begin{equation}\label{eq:lam'sh}
\lambda'_{sh} = \sum\limits_{i \notin I_{sr}} a_i\alpha_i^\vee + \sum\limits_{j \in I_{sr}}b_j\alpha_j^\vee,
\end{equation}
where $b_j \in \Q_{\geq 0}$, and  for all $j \in I_{sr}$, by comparing the second sums in \eqref{eq:lam'srsum} and \eqref{eq:lam'sh}, we have $a_j \geq b_j$ . Define 
\begin{equation}\label{eq:lamsh}
\lambda_{sh} = \sum\limits_{i \notin I_{sr}} a_i\alpha_i^\vee + \sum\limits_{j \in I_{sr}}\lceil b_j \rceil \alpha_j^\vee \in Q^\vee.
\end{equation}
Thus by definition, $\lambda' \geq \lambda_{sh}$, and of course $\lambda'=\lambda_{sh}$ is also possible. 
Let $y_{sh} = t^{v'(\lambda_{sh})}w'$, and note by construction that the element $y_{sh}$ is an alcove  ``close'' to $y$ whose coroot lattice point lies in the same Weyl chamber as that of $y$, but is no longer ``shrunken'' in the sense of \cite{Reu}, inspiring the choice of notation.

In the other direction, the final translate of $y$ will be superregular in the sense of Lemma \ref{T:FiTranslation}. The details are somewhat technical due to the discrepency between $P^\vee$ and $Q^\vee$, but the idea is straightforward: we will construct this new alcove by translating the element $\lambda_{sh} \in Q^\vee$ defined in \eqref{eq:lamsh}  by a prescribed element of $Q^\vee$, such that the new element $\lambda_{sr}$ satisfies $\langle \lambda_{sr}, \alpha_j \rangle > M_k$ for all $j \in I_{sr}$, but barely so.  To make this precise, first note that for all $j \in I_{sr}$, the element $\lambda_{sh} \in Q^\vee$ satisfies $\lambda_{sh} \in H_{\alpha_j, n_j}$ where $n_j = 0$ or $1$, depending on whether the coefficient $b_j$ from \eqref{eq:lam'sh} is an element of $\Z$ or $\Q$, respectively.  We partition the set $I_{sr} = N_0 \sqcup N_1$ to distinguish these two cases:
\begin{equation}
N_0 =\{ j \in I_{sr} \mid n_j=0\} \quad \text{and} \quad N_1 = \{ j \in I_{sr} \mid n_j =1\}.
\end{equation}
Therefore, the second sum in the expression for $\lambda_{sh}' = \operatorname{pr}_{sh}(\lambda')$ in \eqref{eq:lam'sh} can be refined as
\begin{equation}
\sum\limits_{j \in I_{sr}} b_j \alpha_j^\vee = \sum\limits_{j \in N_0}b_j \alpha_j^\vee + \sum\limits_{j \in N_1}b_j \alpha_j^\vee.
\end{equation}
For $l \in \{0,1\}$, denote by $\rho^\vee_{N_l} = \sum\limits_{j \in N_l} \omega^\vee_j \in P^\vee$ the dominant coweight which satisfies $\langle  \rho^\vee_{N_l}, \alpha_j \rangle = 1$ for all $j \in N_l$.   
Now define 
\begin{equation}
\rho^\vee_M = \left(M_k+1 \right)\rho^\vee_{N_0} +  M_k\rho^\vee_{N_1},
\end{equation}
in which case $\rho^\vee_M \in \bigcap\limits_{j \in I_{sr}} H_{\alpha_j, M_k+1}$ by construction.
In the expansion for $\rho^\vee_M = \sum\limits_{j \in I_{sr}}r_j \alpha_j^\vee \in P^\vee$ in terms of simple coroots, the $r_j \in \Q$ are determined by the inverse Cartan matrix.  We thus define 
\begin{equation}\label{eq:lamsr}
\lambda_{sr} = \sum\limits_{i \notin I_{sr}}a_i \alpha^\vee_i + \sum\limits_{j \in I_{sr}} \lceil r_j \rceil \alpha_j^\vee \in Q^\vee,
\end{equation}
and note that $\lambda_{sr} > \lambda'$, comparing \eqref{eq:lam'srsum}. By construction, $\lambda_{sr}$ is the minimum element of $Q^\vee$ which satisfies both $\lambda_{sr} > \lambda'$ and $\langle \lambda_{sr}, \alpha_j \rangle > M_k$ for all $\alpha_j \in \Delta$. Since we have shown above that $\lambda'_x$ also satisfies $\langle \lambda'_x, \alpha_j \rangle > M_k$ for all $\alpha_j \in \Delta$, then Lemma \ref{lem:MaxFinPtw} implies that $\lambda'_x \geq \lambda_{sr}$, 
where of course $\lambda'_x = \lambda_{sr}$ is also possible. Let $y_{sr} = t^{v'(\lambda_{sr})}w'$, which lies in the $v'$-chamber, and note that clearly $\rho^\vee \not> \lambda_{sr}$. Therefore, Corollary \ref{cor:DomBruhat} says that $y_{sr} \leq y_x \leq x$ as well.

Having identified an alcove $y_{sr} \leq x$ which is a local translate of $y$ satisfying the superregularity hypothesis required to apply Lemma \ref{T:FiTranslation}, our next step will be to relate the Newton points $\nu(y_{sr})$ and $\nu(y)$ using this lemma.  Choose any $i \in [r]$ such that $\nu(y_{sr}) \in H_{\alpha_i}$, and define $z = t^{v'(\lambda_{sr})}r_{v'\alpha_i}$.  Since $y$ is not a translation, we know $w' \neq 1,$ and so by Lemma \ref{T:Projection}, we have $\nu(z) = \nu_i(y_{sr}) \geq \nu(y_{sr})$.  In addition, for any minimal length saturated chain $y_{sr} \gtrdot \cdots \gtrdot t^{\gamma}$, Lemma \ref{T:FiTranslation} says that $p_i(\gamma^+) > p_i(\nu_i(y_{sr}))$.  Thus far, we have shown that
 \begin{equation}\label{E:initialineqs}
p_i(\gamma^+) > p_i(\nu_i(y_{sr}))  = p_i \left( \nu(z)\right).
 \end{equation}  
 Analogously, next define $z_{sh} = t^{v'(\lambda_{sh})}r_{v'\alpha_i}$.
  By Lemma \ref{T:Dominant}, we have $\nu(z) = \nu(z^d)$ and $\nu(z_{sh}) = \nu(z_{sh}^d)$, where $z^d =  t^{\lambda_{sr}}r_{\alpha_i}$ and $z_{sh}^d = t^{\lambda_{sh}}r_{\alpha_i}$ are based in the closed dominant Weyl chamber $C$.  
   By Proposition \ref{T:NPforW}, the Newton points $\nu(z^d)$ and $\nu(z_{sh}^d)$ are obtained by orthogonally projecting the vectors $\lambda_{sr}$ and $\lambda_{sh}$, respectively, onto the wall $H_{\alpha_i}$.  In particular, the points $\nu(z^d)$ and $\nu(z_{sh}^d)$ lie on the boundaries of $\operatorname{Conv}(W\lambda_{sr}) \cap C$ and $\operatorname{Conv}(W\lambda_{sh}) \cap C$, respectively.  Since $\lambda_{sr} \geq \lambda_{sh}$ by definition, Lemma \ref{T:AtiyahBott} implies that $\nu(z^d) \geq \nu(z_{sh}^d)$.
Therefore, 
\begin{equation}\label{E:shz}
\nu(z) = \nu(z^d) \geq \nu(z_{sh}^d) =  \nu(z_{sh}).
\end{equation}
By \eqref{eq:lamsh} and \eqref{eq:lamsr}, we see that $\lambda_{sr} - \lambda_{sh} = \sum f_j \alpha_j^\vee$ where $f_j \in \Z_{\geq 0},$ and $f_j \neq 0$ if and only if $j \in I_{sr} \neq [r]$. The fact that $\nu(y_{sr}) \in H_{\alpha_i}$ then implies that $\nu(y_{sh}) \in H_{\alpha_i}$ as well. Therefore, by Lemma \ref{T:Projection},
\begin{equation}\label{E:samenui}
\nu(z_{sh}) = \nu_i(y_{sh}) \geq \nu(y_{sh}).
\end{equation}

More generally, for any $j \in I_{sr}$, we have 
\begin{equation}
u\left(\operatorname{pr}_{w'}(v'(\lambda_{sr} - m\alpha_j^\vee)) \right)\in C
\end{equation}
for all integers $0 \leq m \leq \left\lfloor \frac{\langle \lambda_{sr}, \alpha_j \rangle}{2}\right\rfloor$.
In particular, each of the vectors  $u(\operatorname{pr}_{w'}(v'\lambda')), u(\operatorname{pr}_{w'}(v'\lambda_{sr})),$ and $u(\operatorname{pr}_{w'}(v'\lambda_{sh}))$ is dominant for the same choice of $u \in W$ made in \eqref{eq:uchoice}.     
Recalling that $\lambda' \geq \lambda_{sh}$, comparing \eqref{eq:lam'srsum} and \eqref{eq:lamsh} says that $\lambda' = \lambda_{sh} + \sum\limits_{j \in I_{sr}} c_j \alpha_j^\vee$ for some $c_j \in \Z_{\geq 0}$. 
Therefore, 
 \begin{equation}
 \nu(y) = u(\operatorname{pr}_{w'}(v'\lambda')) = u(\operatorname{pr}_{w'}(v'\lambda_{sh})) + \sum\limits_{j \in I_{sr}} c_j\left( u\left( \operatorname{pr}_{w'} \left( v' \alpha_j^\vee \right) \right)\right).
 \end{equation}
Since $u\left(\operatorname{pr}_{w'}(v'\alpha_j^\vee) \right) \leq 0$ for all $j \in I_{sr}$ by hypothesis, we then have 
\begin{equation}
\nu(y) = \nu(y_{sh}) - \sum\limits_{\ell \in [r]} d_\ell \alpha_\ell^\vee,
\end{equation}
 where $d_\ell \in \Z_{\geq 0}$.  Therefore, $\nu(y_{sh}) \geq \nu(y)$ in this case.  Combining this last observation with the inequalities from \eqref{E:initialineqs}, \eqref{E:shz}, and \eqref{E:samenui},  we have 
 \begin{equation}
p_i(\gamma^+) > p_i(\nu_i(y_{sr}))  = p_i \left( \nu(z)\right) \geq p_i(\nu(z_{sh})) = p_i(\nu_i(y_{sh})) \geq p_i(\nu(y_{sh})) \geq p_i(\nu(y)).
 \end{equation}
The fact that $p_i(\gamma^+) > p_i(\nu(y))$ thus implies by Lemmas \ref{T:AtiyahBott} and \ref{T:FiFormula} that $\nu(y) \not\geq \nu(t^\gamma)$.  Since $t^\gamma \leq y_{sr} \leq x$ by construction and $\nu_x = \max \{\nu(y) \mid y \leq x\}$ by Theorem \ref{T:Vform}, we conclude that $\nu_x \neq \nu(y)$ in this final case.

Altogether, we have thus proved that if $y \leq x$ is not a translation, then $\nu_x \neq \nu(y)$.  Therefore, by Theorem \ref{T:Vform}, we see that 
\begin{equation}\nu_x = \max \{ \gamma^+ \mid t^\gamma \leq x\}.\end{equation}
From Proposition \ref{T:QBGpaths}, we also know that the translation part of the elements in a saturated chain of the form $x \gtrdot \cdots \gtrdot y \gtrdot \cdots \gtrdot t^\gamma$ weakly decreases in dominance order.  Therefore, we can reduce to considering only those saturated chains $x \gtrdot x_1 \gtrdot \cdots \gtrdot x_k = t^\mu$ of minimal length.

On the other hand, by Proposition \ref{T:QBGcomb}, we know that all shortest paths between any pair of elements in the quantum Bruhat graph have the same weight.  Applying Propositions \ref{T:QBGcomb} and \ref{T:QBGpaths}, if there are saturated chains of minimal length from $x$ to both $t^\mu$ and $t^\eta$, then in fact $\mu^+ = \eta^+$.  It thus suffices to choose any of these translations in order to conclude that $\nu_x = \mu^+$.
\end{proof}

\end{subsection}

\begin{subsection}{Proof of the main theorem}\label{S:MainPf}

We are finally prepared to complete the proof of our main result, and so we provide a brief reminder of the theorem statement.  Given an affine Weyl group element $x = t^{v\lambda}w$, we start by considering any path of minimal length $k$ from $w^{-1}v$ to $v$ in the quantum Bruhat graph for $W$.  Under the superregularity hypothesis $\langle \lambda, \alpha_j \rangle > M$ for all simple roots $\alpha_j \in \Delta$, we must prove that the maximum Newton point in $\mathcal{N}(G)_x$ equals $\nu_x = \lambda - \alpha^{\vee}_x,$ where $\alpha^{\vee}_x$ is the weight of the chosen path from $w^{-1}v$ to $v$.  The proof now follows directly from the two key Propositions \ref{T:QBGpaths} and \ref{T:trans}.

\begin{proof}[Proof of Theorem \ref{T:main}]
First note that the superregularity hypothesis in Theorem \ref{T:main} implies the hypotheses required to apply Proposition \ref{T:QBGpaths}.  Use Proposition \ref{T:QBGpaths} ($ii$) to say that the chosen path of minimal length $k$ from $w^{-1}v$ to $v$ corresponds to $2^k$ different minimal length saturated chains in Bruhat order of the form $x \gtrdot x_1 \gtrdot \cdots \gtrdot x_k = t^\mu$.  Choosing any of these chains, we can now directly apply Proposition \ref{T:trans}, which says that $\nu_x = \mu^+$.  To find the value of $\mu^+$, recall from Proposition \ref{T:QBGpaths} that the difference between $\lambda$ and $\mu^+$ corresponds precisely to the sum of the coroots indexing the downward edges in the path from $w^{-1}v$ to $v$.  The weight of a path in the quantum Bruhat graph is defined by summing exactly those coroots labeling the downward edges in the path, and therefore $\mu^+ = \lambda - \alpha_x^\vee$, as desired.
\end{proof}

\end{subsection}

\subsection{Remarks on superregularity}\label{S:HypRmks}

Having proved our main result, we conclude with some remarks about each step in the proof which involves a superregularity hypothesis.  For the reader interested in using Theorem \ref{T:main} to compute smaller examples, we point out several ways in which one expects savings from the bound $\langle \lambda, \alpha_j \rangle > M$ in practice.  For clarity of the exposition, we restrict ourselves to the case $G\neq G_2$, although all comments apply equally to $G=G_2$ with the constants appropriately adjusted.

The first main superregularity hypothesis is introduced in Proposition \ref{T:cocover}, which requires that $\langle \lambda, \alpha_j \rangle > 2\ell(w_0)+2$. The coefficient is forced by noting that the maximum value of $\langle \alpha^\vee, \alpha_j \rangle$ for any $\alpha \in R^+$ and $\alpha_j \in \Delta$ is attained already in the important special case $\alpha = \alpha_j$, but the factor of $\ell(w_0)+1$ is a softer bound.  The basic principle of the proof uses $\ell(t^\lambda)$ as a proxy for $\ell(x)$ when $x=t^{v\lambda}w$, but this estimate is off by exactly $\ell(w^{-1}v)-\ell(v)$ for regular $\lambda$.  Since this difference can actually equal $\ell(w_0)$ in very special cases, this constant is forced by the desire to have a uniform proof for all pairs $w,v \in W$, in addition to the hypothesis from Lemma \ref{T:xlength} that $\lambda$ is regular.  However, in practice this bound is obviously much larger than necessary.

The subsequent need to iterate Proposition \ref{T:cocover} introduces yet an another superregularity criterion.  This additional linear term of $2k$ comes about from the fact that the proofs of the two main Propositions \ref{T:QBGpaths} and \ref{T:trans} iteratively apply Proposition \ref{T:cocover} exactly $k$ times.  However, the assumption made in those proofs, which is once again stronger than necessary, is that \emph{every} step in the chosen path from $w^{-1}v$ to $v$ is a downward edge.  In such circumstances, one does indeed subtract a coroot from $\lambda$ at each step, but in practice these are again fairly special cases.

The superregularity criterion $\langle \lambda, \alpha_j \rangle > 2kc_G$ is then introduced in the course of Lemma \ref{T:FiTranslation}.  This additional hypothesis comes about as a result of using the worst-case-scenario upper bound for the weight of a path of minimal length $k$ from $w^{-1}v$ to $v$ in the quantum Bruhat graph, namely $k\tilde{\mu}$.  In practice, this is an egregious bound, and the reader more well-versed in the combinatorics of the quantum Bruhat graph can quite likely propose a sharp(er) upper bound on the weight of such a path.  The analysis provided in Equation \eqref{E:Mk} shows, however, that the condition $\langle \lambda, \alpha_j \rangle > 2kc_G$ is \emph{barely} stronger than the superregularity hypotheses already required by other components of the proof.  For example, if $G$ is classical, then the condition $\langle \lambda, \alpha_j\rangle > 2kc_G$ is automatically implied by the superregularity hypothesis required by Proposition \ref{T:QBGpaths}.  Therefore, in practice one can completely ignore this additional factor of $2kc_G$, unless one is truly interested in computing examples in the exceptional groups.  Moreover, we point out that our request that $\langle \lambda, \alpha_j \rangle > 2kc_G$ for \emph{all} $\alpha_j \in \Delta$ is overly cautious.  In any given example, one need only be concerned with the roots $\alpha_j$ which actually \emph{occur} as labels on the edges in the path from $w^{-1}v$ to $v$.  In fact, Lemma \ref{T:FiTranslation} only requires the additional superregularity hypothesis $\langle \lambda, \alpha_i \rangle > 2kc_G$ for the \emph{single} root direction associated to $\nu_i(x)$.

One final strengthening of the superregularity hypothesis arises in a special case in the proof of Proposition \ref{T:trans}.  The additional factor of $4\ell(w_0)$ beyond the constant $M_k$ occurs when the coroot $\lambda'$ associated to the element $y \leq x$ fails to be superregular in at least one root direction, and the Newton points \emph{increase} as you shift toward any of these corresponding walls of the Weyl chamber containing $y$.  The key idea in this case is to shift the alcove $y$ deeper into the Weyl chamber so that the superregularity hypothesis required to apply Lemma \ref{T:FiTranslation} is satisfied.  In order for this shifted alcove $y_{sr}$ to provide a useful comparison, we must have $y_{sr} \leq x$ as well.  The additional hypothesis required to shift $y$ deeper into the Weyl chamber while satisfying $y_{sr} \leq x$ is determined by using $\ell(w_0)\tilde{\mu}$ in the proof of Lemma \ref{lem:MaxFinPtw} as an egregious universal upper bound on the relationship between $x$ and the nearest alcove below $x$ in Bruhat order which has the same finite part as $y$.  The quantity $\ell(w_0) \tilde{\mu}$ then further introduces a factor of the maximum value for $\langle \tilde{\mu}, \alpha_j \rangle$ into the new superregularity hypothesis.  This maximum equals 2, but is only attained in type $B_r$; in all other types, $\langle \tilde{\mu}, \alpha_j \rangle \in \{0,1\}$, in which case $\langle \tilde{\mu}, \alpha_j \rangle = 0$ for all but one simple root $\alpha_j$ (two in type $A_r$).

To provide a concrete example illustrating this discussion, if $G=SL_3$ and one traces each of these comments through, one concludes that $\langle \lambda, \alpha_j \rangle > 4$ for $j\in \{1,2\}$ is the sharpest possible superregularity hypothesis required to execute the proof for $G=SL_3$; compare the values of $M \in \{20, 22, 24\}$ for $k \in \{1,2, 3\}$, respectively.  In fact, by comparing \cite{Be1}, one sees that the formula provided in Theorem \ref{T:main} is in fact correct under the weaker hypothesis $\langle \lambda, \alpha_j \rangle >1$ for $j \in \{1,2\}$.  Conversely, the statement of Theorem \ref{T:main} does sometimes fail if $\langle \lambda, \alpha_j \rangle \in \{0,1\}$ for at least one value of $j$.  Although we are not yet prepared to formulate a precise conjecture, it is highly probable that the weakest possible hypothesis one can place on the main theorem is that $x$ lie in the ``shrunken'' Weyl chambers, which is a condition that exactly excludes alcoves which sit between the hyperplanes $H_{\alpha_j,0}$ and $H_{\alpha_j,1}$ for any $\alpha_j \in \Delta$.  See the discussion following the proof of Proposition 4.4 in \cite{LS} for additional evidence in this direction.

\bibliographystyle{alpha}
\bibliography{references}

\end{document}